\documentclass[12pt]{article}
\usepackage{amsmath,amsthm,amsfonts,amssymb}
\usepackage[all]{xy}
\usepackage[usenames]{color} 

\usepackage{rotating}

\newcounter{enumitemp}
\newenvironment{enumeratecontinue}{
 \setcounter{enumitemp}{\value{enumi}}
 \begin{enumerate}
 \setcounter{enumi}{\value{enumitemp}}
}
{
 \end{enumerate}
}

\newcommand\pref[1]{(\ref{#1})}

\newtheorem{thm}{Theorem}[section]

\newtheorem{lemma}[thm]{Lemma}

\newtheorem*{theoremA}{Theorem A}

\newtheorem*{theoremC}{Theorem C}

\newtheorem*{theoremF}{Hyperbolic action theorem (for automorphism subgroups)}

\newtheorem*{corollary*}{Corollary}
\newtheorem*{disintegrationtheorem}{Disintegration Theorem}

\newtheorem{corollary}[thm]{Corollary}
\newtheorem{proposition}[thm]{Proposition}
\newtheorem*{proposition*}{Proposition}

\newtheorem{fact}[thm]{Fact}

\theoremstyle{definition}

\newtheorem{definition}[thm]{Definition} 

\newtheorem*{defn*}{Definition}

\newtheorem{notations}[thm]{Notations}

\theoremstyle{remark}

\newcounter{remarks}
{\paragraph*{Remarks}\smallskip
 \begin{list}{\arabic{remarks}. }{\usecounter{remarks}%
 \setlength{\leftmargin}{0in}%
 \setlength{\rightmargin}{0in}%
 \setlength{\labelsep}{0pt}%
 \setlength{\labelwidth}{0pt}%
 \setlength{\listparindent}{0pt}%
 }
}
{
\end{list}
}

\newcommand\from{:}
\newcommand\inv{{-1}}
\newcommand\subgroup{<}
\newcommand\normal{\,\triangleleft\,}
\newcommand\infinity\infty
\newcommand\na{\text{na}}

\newcommand\disjunion\coprod
\newcommand\act\curvearrowright

\DeclareMathOperator{\Fix}{Fix}
\DeclareMathOperator{\Per}{Per}

\DeclareMathOperator\image{Image}
\DeclareMathOperator\kernel{Kernel}
\DeclareMathOperator\diam{Diam}

\DeclareMathOperator\Isom{Isom}

\newcommand{\R}{{\mathbb R}}
\newcommand\reals{\R}

\newcommand{\Z}{{\mathbb Z}}

\newcommand{\D}{{\mathcal D}}
\renewcommand{\S}{{\mathcal S}}

\newcommand{\A}{\mathcal A}

\newcommand\cQ{\mathcal Q}
\renewcommand\L{\mathcal L}
\newcommand{\PF}{{\text{PF}}}
\newcommand\X{\mathcal X}
\newcommand\Y{\mathcal Y}

\newcommand\semidirect{\rtimes}

\DeclareMathOperator{\Out}{\mathsf{Out}}
\DeclareMathOperator{\Aut}{\mathsf{Aut}}
\DeclareMathOperator{\Inn}{\mathsf{Inn}}

\DeclareMathOperator{\Stab}{\mathsf{Stab}}
\DeclareMathOperator{\wtStab}{\widetilde{\mathsf{Stab}}}

\newcommand{\ffs}{free factor system}

\newcommand{\pg}{PG}

\newcommand{\upg}{UPG}
\newcommand{\F}{\mathcal F}

\def\B{\mathcal B}

\newcommand{\ti} {\tilde}

\newcommand{\eg}{EG}
\newcommand{\noneg}{NEG}
\renewcommand\neg\noneg

\newcommand{\wt}{\widetilde}

\newcommand\gp{\text{gp}}

\newcommand{\ct}{CT}
\newcommand{\cts}{CTs}

\newcommand{\comment}[1]{}

\newcommand\BeFuji{\cite{BestvinaFujiwara:bounded}}

\newcommand\BHTag{BestvinaHandel:tt}
\newcommand\BH{\cite{\BHTag}}
\newcommand\BookZeroTag{BFH:laminations}
\newcommand\BookZero{\cite{\BookZeroTag}}
\newcommand\BookOneTag{BFH:TitsOne}
\newcommand\BookOne{\cite{\BookOneTag}}

\newcommand\BookThree{\cite{BFH:Solvable}}

\newcommand\recognitionTag{FeighnHandel:recognition}
\newcommand\recognition{\cite{\recognitionTag}}
\newcommand\abelianTag{FeighnHandel:abelian}
\newcommand\abelian{\cite{\abelianTag}}

\newcommand\SubgroupsTag{HandelMosher:Subgroups}

\newcommand\SubgroupsOne{\cite[Part I]{\SubgroupsTag}}
\newcommand\SubgroupsTwo{\cite[Part II]{\SubgroupsTag}}
\newcommand\SubgroupsThree{\cite[Part III]{\SubgroupsTag}}
\newcommand\SubgroupsFour{\cite[Part IV]{\SubgroupsTag}}

\newcommand\FSLoxTag{HandelMosher:FreeSplittingLox}
\newcommand\FSLox{\cite{\FSLoxTag}}
\newcommand\HTwoBOneTag{HandelMosher:BddCohomologyI}
\newcommand\HTwoBOne{\cite{\HTwoBOneTag}}
\newcommand\PartOneTag\HTwoBOneTag
\newcommand\PartOne\HTwoBOne

\newcommand\bdy\partial

\newcommand\intersect\cap
\newcommand\union\cup
\newcommand\<\langle
\renewcommand\>\rangle
\newcommand\meet\wedge
\newcommand\composed{\circ}
\newcommand\cross\times
\newcommand\restrict{\bigm |}
\newcommand\wh{\widehat}
\newcommand\inject\hookrightarrow
\DeclareMathOperator\Length{Length}
\newcommand\abs[1]{\bigl| #1 \bigr |}
\newcommand\Id{\text{Id}}

\newcommand\injectto\hookrightarrow
\newcommand\injectfrom\hookleftarrow
\newcommand\surjectto\twoheadrightarrow
\newcommand\surjectfrom\twoheadleftarrow

\DeclareMathOperator\rank{rank}

 \newcommand\surjection\twoheadrightarrow
 
\newcommand\suchthat{\bigm|}
\newcommand\hyp{\mathbf{H}}

\DeclareMathOperator\IA{IA}

\newcommand\IAThree{\IA_n(\Z/3)}

\newcommand\cH{{\cal H}}

\newcommand\cK{\mathcal{K}}

\newcommand\cN{\mathcal{N}}
\newcommand\I{{\cal I}}

\newcommand\dg{\to}

\DeclareMathOperator\lin{{LIN}}

\newcommand\Ad{\text{Ad}}

\newcommand\shsh{{\#\!\#}}
\newcommand\lum{l^-_u}

\newcommand\coarsesim[2]{\mathrel{\substack{{#1}\\\sim\\ {#2}}}}

\newcommand\fTstar{f_{T^*}}

\title{Hyperbolic actions and 2nd bounded cohomology \\of subgroups of $\Out(F_n)$ \\ Part II: Finite lamination subgroups}

\author{Michael Handel and Lee Mosher \thanks{The first author  was supported by the National Science Foundation under Grant No.~DMS-1308710 and by PSC-CUNY under grants in Program Years 46 and 47. The second author was supported by the National Science Foundation under Grant No.~DMS-1406376.}}

\begin{document}

\maketitle

\begin{abstract}
This is the second part of a two part work in which we prove that for every finitely generated subgroup $\Gamma \subgroup \Out(F_n)$, either $\Gamma$ is virtually abelian or its second bounded cohomology $H^2_b(\Gamma;\reals)$ contains an embedding of $\ell^1$. Here in Part II we focus on finite lamination subgroups $\Gamma$ --- meaning that the set of all attracting laminations of elements of $\Gamma$ is finite --- and on the construction of useful hyperbolic actions of those subgroups.
\end{abstract}

\section{Introduction}

In this two part work, we study the second bounded cohomology of finitely generated subgroups of $\Out(F_n)$, the outer automorphism group of a free group of finite rank~$n$. The main theorem of the whole work is an ``$H^2_b$-alternative'' analogous to similar results for subgroups of mapping class groups \BeFuji\ and for groups acting in certain ways on hyperbolic spaces \cite{Brooks:H2bRemarks,BrooksSeries:H2bSurface,EpsteinFujiwara,Fujiwara:H2BHyp,Fujiwara:H2bFreeProduct}:

\begin{theoremA}
For every finitely generated subgroup $\Gamma \subgroup \Out(F_n)$, either $\Gamma$ is virtually abelian or $H^2_b(\Gamma;\reals)$ has an embedded copy of $\ell^1$ and so is of uncountable dimension.
\end{theoremA}

In Part I \PartOne, Theorem~A was reduced to Theorem~C which is the main result here in Part~II. Theorem~C explains how to produce useful hyperbolic actions for a certain class of finite lamination subgroups of $\Out(F_n)$. To fully state Theorem~C we first briefly review the mathematical objects relevant to its statement: hyperbolic actions and their WWPD elements; subgroups of $\Out(F_n)$ which have (virtually) abelian restriction to any invariant proper free factor; and the dichotomy of finite lamination subgroups versus infinite lamination subgroups of $\Out(F_n)$.

Consider a group action $G \act \S$ on a hyperbolic space. Recall that $\gamma \in G$ is \emph{loxodromic} if it acts on the Gromov boundary $\bdy \S$ with a unique repeller-attractor pair $(\bdy_- \gamma, \bdy_+ \gamma) \in \bdy \S \times \bdy \S - \Delta$. Recall also that $G \act \S$ is \emph{nonelementary} if (following \BeFuji) there exist \emph{independent} loxodromic elements $\delta,\gamma \in G$, meaning that the sets $\{\bdy_-\delta,\bdy_+\delta\}$, $\{\bdy_-\gamma,\bdy_+\gamma\}$ are disjoint. Given a loxodromic element $\gamma \in \Gamma$, the WWPD property for $\gamma$ was first defined in \cite{BBF:MCGquasitrees}, and that property has several equivalent reformulations that are collected together in \cite[Proposition 2.6]{HandelMosher:WWPD} and are denoted WWPD~(1)--(4). In particular, WWPD~(2) says that the $G$-orbit of the ordered pair $(\bdy_- \gamma, \bdy_+ \gamma)$ is a discrete subset of the space of distinct ordered pairs $\bdy \S \times \bdy \S - \Delta$. 

Let $\IAThree \subgroup \Out(F_n)$ denote the finite index normal subgroup consisting of all outer automorphisms whose induced action on $H_1(F_n;\Z/3)$ is trivial. A subgroup $\Gamma \subgroup \IAThree$ has \emph{(virtually) abelian restrictions}  (Definition~\ref{DefVirAbRestr}) if for each proper free factor $A \subgroup F_n$ whose conjugacy class $[A]$ is fixed by each element of~$\Gamma$, the natural restriction homomorphism $\Gamma \mapsto \Out(A)$ has (virtually) abelian image. As explained in Part~I~\HTwoBOne, the property of $\Gamma \subgroup \IAThree$ having (virtually) abelian restrictions plays a role in our theory analogous to the role played by irreducible subgroups of mapping class groups in the theory of Bestvina and Fujiwara \BeFuji.

The decomposition theory of Bestvina, Feighn and Handel \BookOne\ associates to each $\phi \in \Out(F_n)$ a finite set $\L(\phi)$ of \emph{attracting laminations}. Associated to a subgroup $\Gamma \subgroup \Out(F_n)$ is the set $\L(\Gamma) = \union_{\phi \in \Gamma} \, \L(\Gamma)$. If $\L(\Gamma)$ is finite then $\Gamma$ is a \emph{finite lamination subgroup}, otherwise $\Gamma$ is an \emph{infinite lamination subgroup}. In Part I we reduced Theorem~A to two results about subgroups of $\IAThree$ with (virtually) abelian restrictions: Theorem~B concerning infinite lamination subgroups which was proved there in Part~I; and Theorem~C concerning finite lamination subgroups, which is proved here in Part~II. Each of Theorems~B and~C concludes with the existence of a hyperbolic action of a finite index normal subgroup possessing a sufficiently rich collection of WWPD elements. The reduction argument in Part~I combines those two conclusions with the Global WWPD Theorem of~\cite{HandelMosher:WWPD} to prove Theorem~A. 

Here is the main theorem of this paper:

\begin{theoremC}
For any finitely generated, non virtually abelian, finite lamination subgroup $\Gamma \subgroup \IAThree$ such that $\Gamma$ has virtually abelian restrictions, there exists a finite index normal subgroup $N \normal \Gamma$ and an action $N \act \S$ on a hyperbolic space, such that the following hold:
\begin{enumerate}
\item\label{ItemThmC_EllOrLox}
Every element of $N$ acts either elliptically or loxodromically on $\S$;
\item\label{ItemThmC_Nonelem}
The action $N \act \S$ is nonelementary;
\item\label{ItemThmC_WWPD}
Every loxodromic element of the commutator subgroup $[N,N]$ is a strongly axial, WWPD element with respect to the action $N \act \S$.
\end{enumerate}
\end{theoremC}
\noindent
See below, after the \emph{Hyperbolic Action Theorem}, for the definition of ``strongly axial''.


\paragraph{Remarks: WPD versus WWPD.} In lectures on this topic we stated a version of Theorem~C~\pref{ItemThmC_WWPD} with a stronger conclusion, saying that in the group $\image(N \mapsto \Isom(\S))$, every loxodromic element of the commutator subgroup satisfies WPD. That conclusion requires a stronger hypothesis, saying roughly that ``virtually abelian restrictions'' holds for a broader class of $\Gamma$-invariant subgroups of $F_n$ than just free factors. This makes the proof and the applications of the theorem considerably more intricate, which was not necessary for the application to Theorem~A in Part~I, and so we have settled for the version of Theorem~C presented here.

\subsection*{Methods of proof of Theorem C} 
The first step of Theorem~C will be to reduce it to a theorem which produces hyperbolic actions of certain subgroups of \emph{auto}morphism groups of free groups. The key step of the reduction argument, carried out in Proposition~\ref{LemmaAutLiftExists}, is the construction of ``automorphic lifts'': for each subgroup $\Gamma \subgroup \Out(F_n)$ that satisfies the hypotheses of Theorem~C, there exists a free factor $A < F_n$ of rank $2 \le k \le n-1$, such that its conjugacy class $[A]$ is $\Gamma$-invariant, and such that the natural homomorphism $\Gamma \mapsto \Out(A)$ lifts to a homomorphism $\Gamma \mapsto \Aut(A)$ whose image in $\Aut(A)$ is not virtually abelian. In Section~\ref{SectionFImpliesCProof} we shall show, using an automorphic lift for which the free factor factor~$A$ has minimal rank, how to reduce Theorem~C to the following statement, in which we have identified $A \approx F_k$ and then rewritten $k$ as~$n$. Recall the canonical isomorphism $F_n \approx \Inn(F_n)$ which associates to each $\gamma \in F_n$ the inner automorphism $i_\gamma(\delta)=\gamma\delta\gamma^\inv$.


\begin{theoremF}
Consider a subgroup $\wh\cH \subgroup \Aut(F_n)$ with $n \ge 2$, and denote $\cH = \image(\wh\cH \mapsto \Out(F_n))$ and $J = \kernel(\wh\cH \mapsto \cH) = \wh\cH \intersect \Inn(F_n)$, giving the following commutative diagram of short exact sequences:
\vspace{-.1in}
$$\xymatrix{
1 \ar[r] & \, J \,\, \ar[r]^{\subset} \ar[d]^-{\subset} & \wh\cH \ar@{->>}[r] \ar[d]^-{\subset} & \cH \ar[r] \ar[d]^-{\subset} & 1 \\
1 \ar[r] & F_n \approx \Inn(F_n) \ar[r]_{\subset}              & \Aut(F_n) \ar@{->>}[r] & \Out(F_n) \ar[r] & 1
}
$$
\vspace{-.1in}

\noindent
If $\cH$ is abelian, if $\wh\cH$ is not virtually abelian, and if no proper, nontrivial free factor of $F_n$ is fixed by the action of $\wh\cH$ on the set of subgroups of $F_n$, then there exists a finite index normal subgroup $\cN \subgroup \wh\cH$ and an action $\cN \act \S$ on a hyperbolic space such that the following properties hold:
\begin{enumerate}
\item\label{ItemThmF_EllOrLox}
Every element of $\cN$ acts either elliptically or loxodromically on $\S$;
\item\label{ItemThmF_Nonelem}
The action $\cN \act \S$ is nonelementary;
\item\label{ItemThmF_WWPD}
Every loxodromic element of $J \intersect \cN$ is a strongly axial, WWPD element with respect to the action $\cN \act \S$. 
\end{enumerate}
\end{theoremF}
To say that a loxodromic element $\phi \in G$ of a hyperbolic action $G \act X$ is \emph{strongly axial} means that it has a \emph{strong axis}, which is a quasi-isometric embedding $\ell \from \reals \to X$ for which there exists a homomorphism $\tau \from \Stab(\bdy_\pm\phi) \to \reals$ such that for all $\psi \in \Stab(\bdy_\pm\phi)$ and $s \in \reals$ we have $\psi(\ell(s)) = \ell(s+\tau(\psi))$.


The proof of the Hyperbolic Action Theorem begins in Section~\ref{SectionOneEdgeExtension} by reducing to the case $\cH \subgroup \IAThree$.
We then choose a maximal, proper, $\cH$-invariant free factor system~$\B$ in $F_n$, and break the proof into cases depending on the ``co-edge number'' of $\B$, which is the minimum integer $k \ge 1$ such that $\B$ is represented by a subgraph $H \subset G$ of a marked graph $G$ for which the complement $G \setminus H$ has $k$ edges. The ``one-edge'' case, where the co-edge number of $\B$ equals~$1$, is handled in Section~\ref{SectionOneEdgeExtension} using an action of $\wh\cH$ on a simplicial tree that is naturally associated to the free factor system~$\B$. 

The ``multi-edge'' case, where the co-edge number of $\B$ is~$\ge 2$, takes up the majority of the paper from Section~\ref{SectionMultiEdgeIntro} to the end. For a full introduction to the multi-edge case, see Section~\ref{SectionMultiEdgeIntro}. In brief, one finds $\phi \in \cH$ which is fully irreducible relative to the free factor system~$\B$, and one uses the top \eg\ stratum of a good relative train track representative of $\phi$
to produce a certain hyperbolic suspension space $\S$. The construction of~$\S$, the formulation and verification of flaring properties of $\S$, and the proof of hyperbolicity of $\S$ by applying the Mj-Sardar combination theorem \cite{MjSardar:combination}, are found in Sections~\ref{SectionFlaringOfLu} and~\ref{SectionPathFunctionsReT}. Then one constructs an isometric action $\wh\cH \act \S$ by applying the theory of abelian subgroups of $\Out(F_n)$ \cite{FeighnHandel:abelian}; see Sections~\ref{SectionAbelianSubgroups}, \ref{SectionTTSemiAction} and~\ref{SectionSuspConstr}. The pieces are put together, and the multi-edge case of the Hyperbolic Action Theorem is proved, in Section~\ref{SectionMultiEdgeExtension}.

\paragraph{Prerequisites from the theory of $\Out(F_n)$.} We will assume that the reader is familiar with certain basic topics of $\Out(F_n)$ that have already been reviewed in Part~I of this paper \PartOne, with detailed references found there. We list these some of these topics here with original citations:
\begin{description}
\item[\protect{\cite[Section 2.2]{\PartOneTag}}] Marked graphs, topological representatives, and relative train track maps \BH. Free factor systems and attracting laminations \BookOne.
\item[\protect{\cite[Section 3.1]{\PartOneTag}}] Properties of $\IAThree = \kernel\bigl(\Out(F_n) \mapsto GL(n,\Z/3)\bigr)$ (\cite{BaumslagTaylor:Center,Vogtmann:OuterSpaceSurvey} and \SubgroupsTwo).
\item[\protect{\cite[Section 3.1]{\PartOneTag}}] Elements and subgroups of $\Out(F_n)$ which are fully irreducible relative to a free factor system~$\A$ of $\Out(F_n)$. The co-edge number of a free factor system $\A$ in $F_n$.
\item[\protect{\cite[Section 3.1]{\PartOneTag}}] Weak attraction theory. The nonattracting subgroup system $\A_\na\Lambda^\pm_\phi$ associated to $\phi \in \Out(F_n)$ and one of its lamination pairs $\Lambda^\pm_\phi$  (\cite{\BookOneTag} and \SubgroupsThree).
\end{description}
Where needed in this paper, we will conduct reviews of other basic concepts.


\tableofcontents

\section{Lifting to an automorphism group}
\label{SectionAutLift}
In this section, the first thing we do is to study the structure of finitely generated, finite lamination subgroups $\Gamma \subgroup \IAThree$ which are not abelian but have virtually abelian restrictions. The motivating question of this study is: If $A \subgroup F_n$ is a proper, nontrivial free factor whose conjugacy class $[A]$ is fixed by $\Gamma$, can the natural restriction map $\Gamma \mapsto \Out(A)$ be lifted to a homomorphism $\Gamma \mapsto \Aut(A)$? And can this be done so that the image is still not virtually abelian? Sections~\ref{SectionDefAutLift}--\ref{SectionAutLiftExists} are devoted to constructions of such ``automorphic lifts''. Using this construction, in Section~\ref{SectionFImpliesCProof} we prove the implication (Hyperbolic Action Theorem)$\implies$(Theorem~C). 


After that, in Section~\ref{SectionFreeSplittingExtension}, we consider any free splitting $F_n \act T$, and we study a natural subgroup of $\Aut(F_n)$ which acts on $T$ in a manner that extends the free splitting action of $F_n \approx \Inn(F_n)$. That study is used in Section~\ref{SectionOneEdgeExtension} to prove one of the two major cases of the Hyperbolic Action Theorem.

\subsection{Definition of automorphic lifts}
\label{SectionDefAutLift}

Recall that for any group and subgroup $H \subgroup G$ which is its own normalizer (e.g.\ a free factor), letting $\Stab[H] \subgroup \Out(G)$ be the stabilizer of the conjugacy class of~$H$, the natural \emph{restriction homomorphism} $\Stab[H] \mapsto \Out(H)$, denoted $\phi \mapsto \phi \restrict H$, is well-defined by choosing $\Phi \in \Aut(G)$ representing~$\phi$ and preserving $H$, and then taking $\phi \restrict H$ to be the outer automophism class of $\Phi \restrict H \in \Aut(H)$ (see \hbox{e.g.\ \SubgroupsOne~Fact 1.4).}

Throughout the paper we use the theorem that virtually abelian subgroups of $\IAThree$ are abelian \cite{HandelMosher:VirtuallyAbelian}. We sometimes write ``(virtually) abelian'' as a reminder that one may freely include or ignore the adverb ``virtually'' in front of the adjective ``abelian'' in the context of a subgroup of $\IA_F(\Z/3)$ for any finite rank free group $F$. One has this freedom, for example, in the following definition (see \cite[Corollary 3.1]{\PartOneTag}):

\begin{definition}[(Virtually) Abelian restrictions]
\label{DefVirAbRestr}
A subgroup $\Gamma \subgroup \IAThree$ has (virtually) abelian restrictions if for any proper free factor $A \subgroup F_n$ such that $\Gamma \subgroup \Stab[A]$, the restriction homomorphism $\Gamma \mapsto \Out(A)$ has (virtually) abelian image.
\end{definition}

\begin{definition}\label{DefAutLift}
Let $\Gamma \subgroup \IAThree$ be a 
subgroup which is not (virtually) abelian and which has (virtually) abelian restrictions. An \emph{automorphic lift} of $\Gamma$ is a homomorphism $\rho \from \Gamma \mapsto \Aut(A)$, where $A \subgroup F_n$ is a proper free factor and $\Gamma \subgroup \Stab[A]$, such that the group $\wh\cH = \image(\rho)$ is not virtually abelian, and such that the following triangle commutes
$$\xymatrix{
 & \Aut(A) \ar[d] \\
\Gamma \ar[ur]^\rho \ar[r]  & \Out(A) 
}$$
In this diagram the horizontal arrow is the natural restriction homomorphism $\Stab[A] \mapsto \Out(A)$ with domain restricted to~$\Gamma$, and the vertical arrow is the natural quotient homomorphism. To emphasize the role of $A$ we will sometimes refer to an \emph{automorphic lift of $\Gamma$ rel~$A$}. Adopting the notation of the Hyperbolic Action Theorem, we set $\wh\cH = \image(\rho)$, $\cH = \image(\wh\cH \mapsto \Out(A)) = \image(\Gamma \mapsto \Out(A))$, and $J = \kernel(\wh\cH \mapsto \cH) = \wh\cH \intersect \Inn(A)$, thus obtaining the commutative diagram shown in Figure~\ref{FigureAutLift}.
\begin{figure}
$$\xymatrix{
& & \Gamma \ar@{->>}[d]_-\rho \ar[ddr] \\
1 \ar[r] & \, J \,\, \ar[r]^\subset \ar[d]^-{\subset} & \wh\cH \ar@{->>}[r] \ar[d]^-{\subset} & \cH \ar[r] \ar[d]^-{\subset} & 1 \\
1 \ar[r] & A \approx \Inn(A) \ar[r]_\subset              & \Aut(A) \ar@{->>}[r] & \Out(A) \ar[r] & 1
}
$$
\caption{Notation associated to an automorphic lift $\rho \from \Gamma \to \Aut(A)$ with image $\wh\cH$. The group $\wh\cH$ is not virtually abelian, the quotient $\cH$ is virtually abelian, and the kernel $J$ is free of rank~$\ge 2$ (possibly infinite). The horizontal rows are exact.}
\label{FigureAutLift}
\end{figure}
\noindent
We note two properties which follow from the definition:
\begin{itemize}
\item $\cH$ is abelian;
\item The free group $J$ has rank~$\ge 2$.
\end{itemize}
The first holds because $A < F_n$ is proper and $\Gamma$ has (virtually) abelian restrictions (see Definition~\ref{DefVirAbRestr}). The second is a consequence of the first combined with the defining requirement that $\wh\cH$ is not virtually abelian, for otherwise the free group $J$ is abelian and so $\wh\cH$ is virtually solvable, but $\Aut(A)$ injects into $\Out(F_n)$ and solvable subgroups of $\Out(F_n)$ are virtually abelian by \BookThree. We put no further conditions on the rank of~$J$, it could well be infinite. When referring to~$J$, each of its elements will be thought of ambiguously as an element of the free factor $A \subgroup F_n$ \emph{or} as the corresponding element of the inner automorphism group $\Inn(A)$; this ambiguity should cause little trouble, by using the canonical isomorphism $A \leftrightarrow \Inn(A)$ given by $\delta \leftrightarrow i_\delta$ where $i_\delta(\gamma)=\delta\gamma\delta^\inv$. 

The \emph{rank} of the automorphic lift $\Gamma \mapsto \Aut(A)$ is defined to be $\rank(A)$, and note that $\rank(A) \ge 2$ because otherwise $\Aut(A)$ is finite in which case each of its subgroups is virtually abelian. 

This completes Definition~\ref{DefAutLift}.
\end{definition}

Here is the first of two main results of Section~\ref{SectionAutLift}. 

\begin{proposition}
\label{LemmaAutLiftExists} If $\Gamma \subgroup \IAThree$ is a finitely generated, finite lamination subgroup which is not (virtually) abelian but which has (virtually) abelian restrictions, then there exists an automorphic lift $\Gamma \mapsto \Aut(A)$. 
\end{proposition}

The proof is found in Section~\ref{SectionAutLiftExists}, preceded by Lemma~\ref{LemmaUltimateEGCase} in Section~\ref{SectionVirtAbelSuffCond}. 

When $n=2$, one recovers from Proposition~\ref{LemmaAutLiftExists} the simple fact that every finite lamination subgroup of $\Out(F_2)$ is virtually abelian, for otherwise the intersection with $\IAThree$ would have an automorphic lift to $\Aut(A)$ for some proper free factor~$A$, from which it would follow that $2 \le \rank(A) \le n-1=1$. Of course this fact has a simple proof, expressed in terms the isomorphism $\Out(F_2) \mapsto \Aut(H_1(F_2;\Z)) \approx GL(2,\Z)$, which we leave to the reader.

\subsection{A sufficient condition to be abelian.}
\label{SectionVirtAbelSuffCond}
In this section we prove Lemma~\ref{LemmaUltimateEGCase} which gives a sufficient condition for a finitely generated, finite lamination subgroup $\Gamma \subgroup \IAThree$ to be abelian. The negation of this condition then becomes a property that must hold when $\Gamma$ is not abelian.

\smallskip


\begin{lemma} \label{LemmaUltimateEGCase} Let $\Gamma \subgroup \IAThree$ be a finitely generated, finite lamination subgroup. If $\A = \{[A_1],\ldots,[A_I]\}$ is a maximal proper $\Gamma$-invariant \ffs, if each restriction $\Gamma \restrict A_i \subgroup \Out(A_i)$ is abelian for $i=1,\ldots,I$, and if $\A$ has co-edge number $\ge 2$ in $F_n$, then $\Gamma$ is abelian.
\end{lemma}

\begin{proof} By Theorem~C of \SubgroupsFour, there exists $\eta \in \Gamma$ which is fully irreducible rel~$\A$, meaning that there does not exist any $\eta$-invariant proper free factor system that strictly contains~$\A$. By relative train track theory, there is a unique lamination pair $\Lambda^\pm_\eta \in \L^\pm(\eta)$ for $\eta$ which is not carried by~$\A$ (\cite{\BookOneTag} Section 3; also see \SubgroupsOne\ Fact~1.55). The nonattracting subgroup system of $\Lambda^\pm_\eta$ has one of two forms, either $\A_\na(\Lambda^\pm_\eta) = \A$ or $\A_\na(\Lambda^\pm_\eta) = \A \union \{[C]\}$ where $C \subgroup F_n$ is a certain maximal infinite cyclic subgroup that is not carried by $\A$ (see \cite[Section 6]{\BookOneTag}; also see \SubgroupsThree, Definitions 1.2 and Corollary 1.9). 

For every nonperiodic line $\ell$ that is not carried by~$\A$, evidently $\ell$ is not carried by $[C]$, and so $\ell$ is not carried by $\A_\na(\Lambda^\pm_\eta)$. By applying Theorem H of \hbox{\SubgroupsThree,} it then follows that $\ell$ is weakly attracted to either $\Lambda^+_\eta$ by iteration of $\eta$ or to $\Lambda^-_\eta$ by iteration of $\eta^\inv$. From this it follows that the two laminations $\Lambda^+_\eta$, $\Lambda^+_{\eta^\inv}=\Lambda^-_\eta$ are the unique elements of $\L(\Gamma)$ not carried by~$\A$, for if there existed $\psi \in \Gamma$ with attracting lamination $\Lambda^+_\psi \in \L(\Gamma) - \{\Lambda^\pm_\eta\}$ not supported by~$\A$ then the generic lines of $\Lambda^+_\psi$ would be weakly attracted either to $\Lambda^+_\eta$ by iteration of $\eta$ or to $\Lambda^-_\eta$ by iteration of $\eta^\inv$, and in either case the set of laminations $\eta^k(\Lambda^+_\psi) = \Lambda^+_{\eta^k \psi \eta^{-k}} \in \L(\Gamma)$, $k \in \Z$, would form an infinite set, contradicting that~$\Gamma$ is a finite lamination group.

Since $\L(\Gamma)$ is finite, for each $\phi \in \Gamma$ each element of $\L(\Gamma)$ has finite orbit under the action of of $\phi$. Since $\Gamma \subgroup \IAThree$, it follows by \cite[Lemma 3.8]{\PartOneTag} that each element of $\L(\Gamma)$ is fixed by $\phi$. In particular, $\Gamma \subgroup \Stab(\Lambda^+_\eta)$. By \cite[Corollary 3.3.1]{\BookOneTag}, there exists a homomorphism $\PF_{\Lambda^+_\eta} \from \Stab(\Lambda^+_\eta) \to \reals$ having the property that for each $\phi \in \Stab(\Lambda^+_\eta)$, the inequality $\PF_{\Lambda^+_\eta}(\phi) \ne 0$ holds if and only if $\Lambda^+_\eta \in \L(\phi)$. 

Consider the following homomorphism, the range of which is abelian:
\begin{align*}
\Omega \from \Gamma \,\, \to \qquad \R\qquad &\oplus \,\, \Gamma \restrict A_1 \,\,\oplus \cdots \oplus \,\,\Gamma\restrict A_I \\
\Omega(\phi) = PF_{\Lambda^+_\eta}(\phi) &\oplus\,\, \phi \restrict A_1 \,\,\oplus \cdots \oplus\,\, \phi \restrict A_I
\end{align*}
We claim that that the kernel of $\Omega$ is also abelian. This claim completes the proof of the lemma, because every solvable subgroup of $\Out(F_n)$ is virtually abelian \BookThree, and every virtually abelian subgroup of $\IAThree$ is abelian~\cite{HandelMosher:VirtuallyAbelian}. 

To prove the claim, by Proposition 5.5 of \FSLox, every subgroup of the group $\kernel\bigl(PF_{\Lambda^+_\eta}(\phi)\bigr)$ consisting entirely of \upg\ elements is abelian, and so we need only check that each element of $\kernel(\Omega)$ is \upg. By \cite[Corollary 5.7.6]{\BookOneTag}, every \pg\ element of $\IAThree$ is \upg, and so we need only check that each $\phi \in \kernel(\Omega)$ is \pg, equivalently $\L(\phi) = \emptyset$. Suppose to the contrary that there exists $\Lambda^+_\phi \in \L(\phi)$, with dual repelling lamination $\Lambda^-_\phi \in \L(\phi^\inv)$. Since $\phi \restrict A_i$ is trivial in $\Out(A_i)$ for each component $[A_i]$ of $\A$, neither of the laminations $\Lambda^\pm_\phi$ is supported by $\A$. Since $\L(\phi) \subset \L(\Gamma)$, it follows as shown above that $\{\Lambda^+_\phi,\Lambda^-_\phi\} = \{\Lambda^+_\eta,\Lambda^-_\eta\}$, and so $PF_{\Lambda^+_\eta}(\phi) \ne 0$ a contradiction.
\end{proof}

\subsection{Constructing automorphic lifts: proof of Proposition~\ref{LemmaAutLiftExists}}
\label{SectionAutLiftExists}

Let $\Gamma \subgroup \IAThree$ be a finitely generated, finite lamination subgroup that is not (virtually) abelian and has (virtually) abelian restrictions. Choose a maximal proper $\Gamma$-invariant free factor system $\A = \{[A_1],\ldots,[A_I]\}$, and so each restricted group denoted $\cH_i = \Gamma \restrict [A_i] \subgroup \Out(A_i)$ is abelian. Since $\Gamma \subgroup \IAThree$, it follows that each component $[A_i]$ of $\A$ is fixed by $\Gamma$ (\SubgroupsTwo, Lemma 4.2). 

The group $\Gamma$ is not abelian by~\cite{HandelMosher:VirtuallyAbelian}, and so by Lemma~\ref{LemmaUltimateEGCase} the extension $\A \sqsubset \{[F_n]\}$ is a one-edge extension. We may therefore choose a marked graph pair $(G,H)$ representing $\A$ so that $G \setminus H = E$ is a single edge, and we may choose $H$ so that its components are roses, with each endpoint of $E$ being the rose vertex of a component of $H$. The number of components of $\A$ equals the number of components of $H$, that number being either one or two, and we cover those cases separately. 

\medskip\noindent
\textbf{Case 1: $\A$ has two components,} say $\A = \{[A_1],[A_2]\}$ where $F_n = A_1 \ast A_2$. We construct a commutative diagram as follows:
$$
\xymatrix{
\Aut(A_1) \oplus \Aut(A_2) \ar[r]^q & \Out(A_1) \oplus \Out(A_2) \\
 & \Gamma \ar[ul]^{\alpha=\alpha_1 \oplus \alpha_2 \!\quad} \ar[u]_{\rho=\rho_1\oplus\rho_2} \ar[dl]_\omega \ar[d]^\subset \\
\Aut(F_n,A_1,A_2) \ar[uu]^r \ar[r] & \Out(F_n)}
$$
\noindent
$\Aut(F_n,A_1,A_2)$ is the subgroup of $\Aut(F_n)$ that preserves both $A_1$ and~$A_2$. The homomorphism $r$ is induced by restricting $\Aut(F_n,A_1,A_2)$ to $\Aut(A_1)$ and to $\Aut(A_2)$. Evidently $r$ is injective, since an automorphism of $F_n$ is determined by its restrictions to the complementary free factors $A_1,A_2$. The homomorphisms denoted by the top and bottom arrows of the diagram are induced by canonical homomorphisms from automorphism groups to outer automorphism groups. For $i=1,2$ the homomorphism $\rho_i$ is the composition $\Gamma \inject  \Stab[A_i] \mapsto \Out(A_i)$ where the latter map is the natural restriction homomorphism.

We must construct $\omega$, $\alpha_1$, $\alpha_2$. We may choose the marked graph pair $(G,H)$ representing the free factor system~$\A$ to have the following properties: the two rose components $H_1$, $H_2$ of the subgraph $H$ have ranks equal to $\rank(A_1)$, $\rank(A_2)$ respectively; the edge $E$ is oriented and decomposes into oriented half-edges $E=\overline E_1 E_2$; the common initial point of $E_1,E_2$ is denoted $w$; their respective terminal vertices are the rose vertices $v_i \in H_i$; and there is an isomorphism $\pi_1(G,w) \approx F_n$ which restricts to isomorphisms $\pi_1(E_i \union H_i, w) \approx A_i$ for $i=1,2$. Given $\phi \in \Gamma$, let $f \from G \to G$ be a homotopy equivalence that represents $\phi$, preserves $H_1$ and $H_2$, and restricts to a locally injective path on $E = \overline E_1 E_2$. By Corollary~3.2.2 of \BookOne\ we have $f(E) = \bar u_1 E^{\pm 1} u_2$ for possibly trivial paths closed paths $u_i$ in $H_i$, $i=1,2$, and in fact the plus sign occurs and so $f(E) = \bar u_1 E u_2$, because $\phi \in \Gamma \subgroup \IAThree$. After pre-composing $f$ with a homeomorphism of $G$ isotopic to the identity that restricts to the identity on $H_1 \union H_2$ and that moves the point $w \in E$ to $f^\inv(w) \in E$, we may also assume that $f(w)=w$, and so $f(E_1) = E_1 u_1$ and $f(E_2) = E_2 u_2$. Define $\alpha_i(\phi) \in \Aut(A_i) \approx \Aut(\pi_1(E_i \union H_i,v_i))$ to be the automorphism induced by $f \restrict E_i \union H_i$, and then use the isomorphism $F_n = A_1 * A_2$ to define $\omega(\phi)$. Note that $\omega(\phi)$ is the unique lift of $\phi \in \Out(F_n)$ to $\Aut(F_n)$ which preserves $A_1$ and $A_2$, because any two such lifts differ by an inner automorphism $i_c$ that preserves both $A_1$ and $A_2$, implying by malnormality that $c \in A_1 \intersect A_2$ and so is trivial. It follows from uniqueness that $\omega$, $\alpha_1$, and $\alpha_2$ are homomorphisms, and hence so is $\alpha$. Commutativity of the diagram is straightforward from the construction. The homomorphism $\omega$ is injective because it is a lift of the inclusion $\Gamma \inject \Out(F_n)$. Since $r$ and $\omega$ are injective, by commutativity of the diagram $\alpha$ is also injective. 

At least one of the two maps $\alpha_i \from \Gamma \to \Aut(A_i)$ is an automorphic lift because at least one of the corresponding images $\wh\cH_i = \image(\alpha_i) < \Aut(A_i)$ is not virtually abelian: if both were virtually abelian, then $\alpha(\Gamma) \subgroup \wh\cH_1 \oplus \wh\cH_2$ would be virtually abelian, but $\alpha$ is injective and so $\Gamma$ would be virtually abelian, a contradiction. 


\medskip\noindent
\textbf{Case 2: $\A$ has a single component,} say $\A=\{[A]\}$, and so $F_n = A * \<b\>$ for some $b \in F_n$. The proof in this case is similar to Case~1, the main differences being that in place of direct sum we use fiber sum, and the marked graph pair $(G,H)$ representing $\A$ will have connected subgraph~$H$.

We shall construct the following commutative diagram:
$$\xymatrix{
\Aut^2(A,A^b) \ar[rr]^{q} && \Out(A) \\
 && \cH \ar[u]_\rho \ar[ull]^{\alpha} \ar[dll]^\omega \ar[d]^\subset \\
\Aut(F_n,A,A^b) \ar[uu]^{r} \ar[rr] && \Out(F_n)
}
$$
\noindent
In this diagram we use the following notations: the conjugate $A^b = b A b^\inv$; the restricted inner automorphism $i_{b} \from A \to A^b$ where $i_{b}(a)=b a b^\inv$; the adjoint isomorphism $\Ad_b \from \Aut(A) \to \Aut(A^b)$ where $\Ad_b(\Phi) = i^{\vphantom\inv}_{b} \,\composed \, \Phi \, \composed \, i^\inv_{b}$; the canonical epimorphism $q_A \from \Aut(A) \to \Out(A)$; and the epimorphism $q_{A^b} = q_A \composed \Ad_b^\inv \from \Aut(A^b) \to \Out(A)$. Also, the fiber sum of $q_A$ and $q_{A^b}$ is the following subgroup of $\Aut(A) \oplus \Aut(A^b)$:
$$\Aut^2(A,A^b) = \{(\Phi,\Phi') \in \Aut(A) \oplus \Aut(A^b) \suchthat q_A(\Phi)=q_{A^b}(\Phi')\}
$$
Define the homomorphism $q$ by $q(\Phi,\Phi') = q_A(\Phi)=q_{A^b}(\Phi')$. Define $\Aut(F_n,A,A^b) \subgroup \Aut(F_n)$ to be the subgroup that preserves both $A$ and~$A^b$. The homomorphism $r$ is jointly induced by the two restriction homomorphisms $r_A$, $r_{A^b}$ from $\Aut(F_n,A,A^b)$ to $\Aut(A)$, $\Aut(A^b)$ because the two compositions $q_A \composed r_A$, $q_{A^b} \composed r_{A^b} \from \Aut(F_n,A,A^b) \to \Out(A)$ are evidently the same. Note that $r$ is injective, for if $\Phi \in \Aut(F_n,A,A^b)$ restricts to the identity on each of $A$ and $A^b$, then $\Phi(a)=a$ and $\Phi(a^b)=a^b$ for all $a \in A$, and it follows that $b^\inv \Phi(b)$ commutes with all $a \in A$; since $\rank(A) \ge 2$, we have $\Phi(b)=b$, and hence $\Phi$ is trivial.

To construct the homomorphism $\omega$, we may in this case choose the marked graph pair $(G,H)$ representing $\A$ so that: $H$ is a rose whose rank equals $\rank(A)=n-1$; as before $E=\overline E_1 E_2$ with $w$ the common initial point of $E_1,E_2$; the terminal endpoints of both $E_1$ and $E_2$ equal the rose vertex $v \in H$; and there is an isomorphism $\pi_1(G,w) \approx F_n$ which restricts to $\pi_1(E_1 \union H,w) \approx A$ and $E_2 \overline E_1 \approx b$. It follows that $\pi_1(E_2 \union H,w) \approx A^b$. For each $\phi \in \cH$, applying Corollary~3.2.2 of \BookOne\ as in the previous case, we find $\Phi = \omega(\phi) \in \Aut(F_n,A,A^b)$ that represents $\phi$ and that is represented by a homotopy equivalence $f_\phi \from G \to G$ that preserves $H$, fixes $w$, and takes $E_i \mapsto E_i u_i$ for possibly trivial paths $u_1,u_2$ in $H$ based at $v$. The map $\omega$ is a homomorphism because $\Phi$ is the unique element of $\Aut(F_n,A,A^b)$ that represents $\phi$, for if $c \in F_n$ and if the inner automorphism $i_c$ preserves both $A$ and $A^b$ then by malnormality of $A$ and $A^b$ we have $c \in A \cap A^b$, and again by malnormality we have that $c$ is trivial. It follows that $\omega$ is injective. Since $r$ is injective it follows that $\alpha = r \composed \omega$ is injective. Denote $\alpha(\phi) = (\alpha_A(\phi),\alpha_{A^b}(\phi)) \in \Aut^2(A,A^b)$. Obviously the compositions $q_A \, \composed \, \alpha_A$, $q_{A^b} \, \composed \, \alpha_{A^b} \from \cH \to \Out(A)$ are the same, and hence there is an induced homomorphism $\rho \from \cH \to \Out(A)$ which is topologically represented by $f_\phi \restrict H$. This completes the construction of the above diagram, and commutativity is evident. 

As in the previous case, we will be done if we can show that at least one of the two homomorphisms $\alpha_A \from \Gamma_0 \to \Aut(A)$ or $\alpha_{A^b} \from \Gamma_0 \to \Aut(A^b)$ has image that is not virtually abelian, but if both are virtually abelian then $\image(\alpha)$ is contained in a virtually abelian subgroup of $\Aut^2(A,A^b)$, which by injectivity of $\alpha$ implies that $\Gamma$ is virtually abelian, a contradiction.


\subsection{Proof that the Hyperbolic Action Theorem implies Theorem C}
\label{SectionFImpliesCProof}

Let $\Gamma \subgroup \IAThree$ be a finitely generated, finite lamination subgroup which is not (virtually) abelian and which has (virtually) abelian restrictions. 
By applying Proposition~\ref{LemmaAutLiftExists}, there exists a free factor $A \subgroup F_n$ such that $\Gamma \subgroup \Stab[A]$, and there exists an automorphic lift $\rho \from \Gamma \mapsto \Aut(A)$; we may assume that $\rank(A)$ is minimal amongst all choices of $A$ and $\rho$. We adopt the notation of Figure~\ref{FigureAutLift} in Section~\ref{SectionDefAutLift}, matching that notation with the Hyperbolic Action Theorem by choosing an isomorphism $A \approx F_k$ where $k = \rank(A)$. Most of the hypotheses of the Hyperbolic Action Theorem are now immediate: $\wh\cH = \image(\rho) \subgroup \Aut(A)$ is not virtually abelian by definition of automorphic lifts; also $\cH = \image(\wh\cH \mapsto \Out(A))$ is abelian.

We must check the one remaining hypothesis of the Hyperbolic Action Theorem, namely that no proper, nontrivial free factor $B \subgroup A$ is preserved by the action of~$\wh\cH$. Assuming by contradiction that $B$ is preserved by $\wh\cH$, consider the restriction homomorphism $\sigma \from \wh\cH \mapsto \Aut(B)$. We claim that the composition $\sigma\rho \from \Gamma \to \Aut(B)$ is an automorphic lift of $\Gamma$. Since $\rank(B)<\rank(A)$, once this claim is proved, it contradicts the assumption that $\rho \from \Gamma \to \Aut(A)$ is an automorphic lift of minimal rank. The canonical isomorphism $\Inn(F_k) \leftrightarrow F_k$, denoted $i_\delta \leftrightarrow \delta$, restricts to an isomorphism between $J = \wh\cH \intersect \Inn(F_k)$ and some subgroup of $F_k$. If $i_\delta \in J$ then $i_\delta$ preserves $B$, and since $B$ is malnormal in $A$ it follows that $\delta \in B$. Thus $\sigma$ restricts to an injection from $J$ to $\Inn(B)$. Also, the group $J$ is a free group of rank~$\ge 2$, for if it were trivial or cyclic then $\wh\cH$ would be virtually solvable and hence, by \BookThree, virtually abelian, a contradiction. Since the image of the map $\sigma\rho \from \Gamma \xrightarrow{\rho} \wh\cH \xrightarrow{\sigma} \Aut(B)$ contains $\sigma(J)$, it follows that $\image(\sigma\rho)$ is not virtually abelian. Since $\wh\cH$ preserves the $A$-conjugacy class of~$B$, and since $B$ is malnormal in~$F_n$, it follows $\Gamma$ preserves the $F_n$-conjugacy class of~$B$. Tracing through the definitions one easily sees that the composed homomorphism $\Gamma \xrightarrow{\sigma\rho} \Aut(B) \mapsto \Out(B)$ is equal to the composition of $\Gamma \inject \Stab_{\Out(F_n)}[B] \mapsto \Out(B)$, where the latter map is the natural restriction homomorphism. Thus $\sigma\rho \from \Gamma \to \Aut(B)$ is an automorphic lift of $\Gamma$, completing the proof of the claim.

Applying the conclusions of the Hyperbolic Action Theorem using the free group $A \approx F_k$ and $\wh\cH \subgroup \Aut(F_k)$, we obtain a finite index normal subgroup $\wh N \subgroup \wh\cH$ and a hyperbolic action $\wh N \act \S$ satisfying conclusions~\pref{ItemThmF_EllOrLox}, \pref{ItemThmF_Nonelem} and~\pref{ItemThmF_WWPD} of that theorem. The subgroup $N = \rho^\inv(\wh N) \subgroup \Gamma$ is a finite index normal subgroup of~$\Gamma$, and by composition we have a pullback action $N \mapsto \wh N \act \S$. By the Hyperbolic Action Theorem~\pref{ItemThmF_EllOrLox}, each element of~$\wh N$ acts elliptically or loxodromically on~$\S$, and so the same holds for each element of~$N$, which is Theorem~C~\pref{ItemThmC_EllOrLox}. By item~\pref{ItemThmF_Nonelem} of the Hyperbolic Action Theorem, the action $\wh N \act \S$ is nonelementary, and so the same holds for the action $N \act \S$, which is Theorem~C~\pref{ItemThmC_Nonelem}. Since the image of the homomorphism $N \mapsto \wh N \inject \wh\cH \mapsto \cH$ is abelian, it follows that the image in $\wh\cH$ of the commutator subgroup $[N,N]$ is contained in $J = \kernel(\wh\cH \mapsto \cH)$, and hence the image of $[N,N]$ in $\wh N$ is contained in $J \intersect \wh N$. By conclusion~\pref{ItemThmF_WWPD} of the Hyperbolic Action Theorem, each loxodromic element of $J \intersect \wh N$ is a strongly axial WWPD element with respect to the action $\wh N \act \S$. Applying \cite[Corollary 2.5]{HandelMosher:WWPD} which says that WWPD is preserved under pullback, and using the evident fact that the strongly axial property is preserved under pullback, it follows that each loxodromic element of $[N,N]$ is a strongly axial, WWPD element of the action $N \act \S$, which is the statement of Theorem~C~\pref{ItemThmC_WWPD}. \qed

\subsection{Automorphic extensions of free splitting actions}
\label{SectionFreeSplittingExtension}
From the hypothesis of the Hyperbolic Action Theorem, our interest is now transferred to the context of a finite rank free group $F_n$ --- perhaps identified isomorphically with some free factor of a higher rank free group --- and of a subgroup $\wh\cH \subgroup \Aut(F_n)$ that has the following irreducibility property: no proper, nontrivial free factor of $F_n$ is preserved by $\wh\cH$. To prove the Hyperbolic Action Theorem one needs actions of such groups $\wh\cH$ on hyperbolic spaces. In this section we focus on a natural situation which produces actions on trees.

\paragraph{Free splittings.} Recall that a \emph{free splitting} of $F_n$ is a minimal, simplicial action $F_n \act T$ on a simplicial tree~$T$ such that the stabilizer of each edge is trivial. Two free splittings $F_n \act S,T$ are \emph{simplicially equivalent} if there exists an $F_n$-equivariant simplicial isomorphism $S \mapsto T$; we sometimes use the notation $[T]$ for the simplicial equivalence class of a free splitting $F_n \act T$. Formally the action $F_n \act T$ is given by a homomorphism $\alpha \from F_n \mapsto \Isom(T)$, which we denote more briefly as $\alpha \from F_n \act T$. In this formal notation, $\Isom(T)$ refers to the group of simplicial self-isomorphisms of $T$, equivalently the self-isometry group of $T$ using the geodesic metric given by barycentric coordinates on simplices of $T$. We note that an element of $\Isom(T)$ is determined by its restriction to the vertex set, in fact it is determined by its restriction to the subset of vertices of valence~$\ge 3$. Two free splittings are \emph{equivalent} if there is an $F_n$-equivariant simplicial isomorphism between them, the equivalence class of a free splitting $F_n \act T$ is denoted~$[T]$, and the group $\Out(F_n)$ acts naturally on the set of equivalence classes of free splittings.

Given a free splitting $F_n \act T$, the set of conjugacy classes of nontrivial vertex stabilizers of a free splitting is a free factor system of $F_n$ called the \emph{vertex group system} of~$T$ denoted as $\A(T)$. The function which assigns to each free splitting $T$ its vertex group system $\A(T)$ induces a well-defined, $\Out(F_n)$-equivariant function 
$$[T] \mapsto \A(T)
$$
from the set of simplicial equivalence classes of free splittings to the set of free factor systems. 

Every free splitting $F_n \act T$ can be realized by some marked graph pair $(G,H)$ in the sense that $T$ is the $F_n$-equivariant quotient of the universal cover $\wt G$ where each component of the total lift $\wt H \subset \wt G$ is collapsed to a point. One may also assume that each component of $H$ is noncontractible, in which case the same marked graph pair $(G,H)$ topologically represents the vertex group system of $T$.

\paragraph{Twisted equivariance (functional notation).} Given two free splittings $F_n \act S,T$ and an automorphism $\Phi \in \Aut(F_n)$, a map $h \from S \to T$ is said to be \emph{$\Phi$-twisted equivariant} if
$$h(\gamma \cdot x) = \Phi(\gamma) \cdot h(x) \quad\text{for all $x \in S$, $\gamma \in F_n$.}
$$
The special case when $\Phi = \Id$ is simply called \emph{equivariance}. 

A twisted equivariant map behaves well with respect to stabilizers, as shown in the following simple fact:

\begin{lemma}\label{LemmaActionTwisted}
For any free splittings $F_n \act S,T$, for each $\Phi \in \Aut(F_n)$, and for each $\Phi$-twisted equivariant map $f \from S \to T$, we have
\begin{enumerate}
\item\label{ItemSublemmaPoints}
$\Phi(\Stab(x)) \subgroup \Stab(f(x))$ for all $x \in S$.
\item\label{ItemSublemmaAxes}
If in addition the map $f \from S \to T$ is a simplicial isomorphism, then the inclusion of item~\pref{ItemSublemmaPoints} is an equality: $\Phi(\Stab(x)) = \Stab(f(x))$. Furthermore, $\gamma \in F_n$ acts loxodromically on $S$ if and only if $\Phi(\gamma)$ acts loxodromically on $T$, in which case their axes $A^S_\gamma \subset S$ and $A^T_\Phi(\gamma) \subset T$ satisfy $A^T_{\Phi(\gamma)} = f(A^S_\gamma)$.
\end{enumerate}
\end{lemma}

\paragraph{Remark.} One could approach this proof by first working out the equivariant case ($\Phi = \Id$), and then reducing the twisted case to the equivariant case by conjugating the action $F_n \act T$ using $\Phi$. We instead give a direct proof.

\begin{proof} To prove~\pref{ItemSublemmaPoints}, for each $x \in S$ and $\gamma \in F_n$ we have
\begin{align*}
\gamma \in \Phi(\Stab(x)) &\iff \Phi^\inv(\gamma) \in \Stab(x) \iff \Phi^\inv(\gamma) \cdot x = x \\
  &\implies f(\Phi^\inv(\gamma) \cdot x) = f(x) \\
  &\iff \Phi(\Phi^\inv(\gamma)) \cdot f(x) = f(x) \quad\text{(by twisted equivariance)}\\
  &\iff \gamma \cdot f(x) = f(x) \\
  &\iff \gamma \in \Stab(f(x))
\end{align*}
and so $\Phi(\Stab(x)) \subgroup \Stab(\Phi \cdot x)$. 

To prove \pref{ItemSublemmaAxes}, consider the inverse simplicial automorphism $f^\inv \from T \to S$. The implication in the second line may be inverted by applying $f^\inv$ to both sides of the equation in the second line. For the rest of \pref{ItemSublemmaAxes}, it suffices to prove the ``only if'' direction, because the ``if'' direction can then be proved using that $f^\inv$ is $\Phi^\inv$-twisted equivariant. Assuming $\gamma$ is loxodromic in $S$ with axis $A^S_\gamma$, consider the line $f(A^S_\gamma) \subset T$. Calculating exactly as above one shows that the equation $\Phi(\Stab(A^S_\gamma)) = \Stab(f(A^S_\gamma))$ holds. We may assume that $\gamma$ is a generator of the infinite cyclic group $\Stab(A^S_\gamma)$, and so $\<\Phi(\gamma)\> = \Phi\<\gamma\> = \Stab(f(A^S_\gamma))$. Since the stabilizer of the line $f(A^S_\gamma)$ is the infinite cyclic group $\<\Phi(\gamma)\>$, it follows that $\Phi(\gamma)$ is loxodromic and its axis $A^T_{\Phi(\gamma)}$ is equal to~$f(A^S_\gamma)$.
\end{proof}

\paragraph{Free splittings of co-edge number~$1$.}
Recall the \emph{co-edge number} of a free factor system $\A$ of~$F_n$ (see for example \cite[Section 3.1]{\PartOneTag}), which is the minimal number of edges of $G \setminus H$ amongst all marked graph pairs $(G,H)$ such that $H$ is a representative of $\A$. There is a tight relationship between free splittings with a single edge orbit and free factor systems with co-edge number~$1$. To be precise:

\begin{fact}\label{FactOneEdge} \cite[Section 4.1]{HandelMosher:distortion}
When the $\Out(F_n)$-equivariant function $[T] \mapsto \A = \A(T)$ is restricted to free splittings $T$ with one edge orbit and free factor systems $\A$ with co-edge number~$1$, the result is a bijection, and hence 
$$\Stab[T] = \Stab(\A)
$$
whenever $T$ and $\A$ correspond under this bijection.
\qed\end{fact}
\noindent
Under the bijection in Fact~\ref{FactOneEdge}, the number of components of $\A$ equals the number of vertex orbits of $T$ which equals $1$ or $2$.  If $\A=\{[A]\}$ has a single component then $\rank(A)=n-1$, there is a free factorization $F_n = A * B$ where $\rank(B)=1$, and the quotient graph of groups $T / F_n$ is a circle with one edge and one vertex. On the other hand if $\A = \{[A_1],[A_2]\}$ has two components then $\rank(A_1) + \rank(A_2)=n$, $A_1,A_2$ can be chosen in their conjugacy classes so that there is a free factorization $F_n = A_1 * A_2$, and the quotient graph of groups $T / F_n$ is an arc with one edge and two vertices.

\paragraph{The stabilizer of a free splitting and its automorphic extension}
Consider a free splitting $\alpha \from F_n \act T$ and its stabilizer subgroup $\Stab[T] \subgroup \Out(F_n)$. Let $\wtStab[T] \subgroup \Aut(F_n)$ be the preimage of $\Stab[T]$ under the standard projection homomorphism $\Aut(F_n) \mapsto \Out(F_n)$, and so we have a short exact sequence
$$1 \mapsto F_n \approx \Inn(F_n) \inject \wtStab[T] \mapsto \Stab[T] \mapsto 1
$$
From this setup we shall define in a natural way an action $\wtStab[T] \act T$ which extends the given free splitting action $\alpha \from F_n \act T$. We proceed as follows. 

For each $\Phi \in \Aut(F_n)$ we have the composed action $\alpha \circ \Phi \from F_n \act T$. Assuming in addition that $\Phi \in \wtStab[T]$, in other words that $\Phi$ is a representative of some element of $\Stab[T]$, it follows the actions $\alpha$ and $\alpha \circ \Phi \from F_n \act T$ are equivalent, meaning that there exists a simplicial automorphism $h \in \Isom(T)$ such that $h \composed \alpha(\gamma) = \alpha(\Phi(\gamma)) \composed h$. When this equation is rewritten in action notation it simply says that $h$ satisfies $\Phi$-twisted equivariance: $h(\gamma \cdot x) = \Phi(\gamma) \cdot h(x)$ for all $\gamma \in F_n$, $x \in S$. Suppose conversely that for some $\Phi \in \Aut(F_n)$ there exists a $\Phi$-twisted equivariant $h \in \Isom(T)$. Since $h$ conjugates the action $\alpha \from F_n \act T$ to the action $\alpha \circ\Phi \from F_n \act T$, it follows that $\phi \in \Stab[T]$ and $\Phi \in \wtStab[T]$. This proves the equivalence of \pref{ItemLTUInStabT}, \pref{ItemLTUForall} and~\pref{ItemLTUExists} in the following lemma, which also contains some uniqueness information regarding the conjugating maps $h$.


\begin{lemma}
\label{LemmaTwistedUniqueness}
For each free splitting $\alpha \from F_n \act T$ and each $\phi \in \Out(F_n)$ the following are equivalent:
\begin{enumerate}
\item\label{ItemLTUInStabT}
$\phi \in \Stab[T]$
\item\label{ItemLTUForall}
For each $\Phi \in \Aut(F_n)$ representing $\phi$, there exists a $\Phi$-twisted equivariant isomorphism $h \from T \to T$.
\item\label{ItemLTUExists}
For some $\Phi \in \Aut(F_n)$ representing $\phi$, there exists a $\Phi$-twisted isomorphism $h \from T \to T$.
\end{enumerate}
Furthermore, 
\begin{enumeratecontinue}
\item\label{ItemLTUUnique}
If $\phi$ satisfies the equivalent conditions \pref{ItemLTUInStabT}, \pref{ItemLTUForall},~\pref{ItemLTUExists} then for each $\Phi \in \wtStab[T]$ representing $\phi$, the $\Phi$-twisted equivariant isomorphism of $h_\Phi$ is uniquely determined by~$\Phi$, and is denoted
$$h_\Phi \from T \to T
$$
\item\label{ItemLTUInner}
For each $\gamma \in F_n$ with corresponding inner automorphism $i_\gamma(\delta)=\gamma\delta\gamma^\inv$, the two maps $h_{i_\gamma} \from T \to T$ and $\alpha(\gamma) \from T \to T$ are equal.
\end{enumeratecontinue}
\end{lemma} 

\textbf{Remark:} In item~\pref{ItemLTUInner}, note that $h_{i_\gamma}$ is defined because $i_\gamma \in \Inn(F_n) \subgroup \wtStab[T]$.  

\medskip

The uniqueness statement~\pref{ItemLTUUnique} is a special case of a more general uniqueness statement that we will make use of later:

\begin{lemma}\label{LemmaTwistedUniqueGeneral}
For any two free splittings $F_n \act S,T$ and any $\Phi \in \Aut(F_n)$, there exists at most one $\Phi$-equivariant simplicial isomorphism $h \from S \mapsto T$. In particular, taking $\Phi = \Id$, there exists at most one equivariant simplicial isomorphism $h \from S \mapsto T$.
\end{lemma}

\begin{proof} Suppose that a $\Phi$-twisted equivariant simplicial isomorphism $h \from S \mapsto T$ exists. For each $\gamma \in F_n$, it follows by $\Phi$-twisted equivariance that $\gamma$ acts loxodromically on~$S$ with axis $A^S_\gamma$ if and only if $\Phi(\gamma)$ acts loxodromically on $T$ with axis $A^T_\gamma = f(A^S\gamma)$ (by Lemma~\ref{LemmaActionTwisted}). The map that $h$ induces on the set of axes of loxodromic elements is therefore uniquely determined. It follows that the restriction of $h$ to the set of vertices $v \in T$ of valence~$\ge 3$ is uniquely determined by~$\Phi$, because $v$ may be expressed in the form $\{v\} = A^S_\beta \intersect A^S_\gamma \intersect A^S_\delta$ for a certain choice of $\beta,\gamma,\delta \in F_n$, and hence 
\begin{align*}
\{h(v)\} &= h(A^S_\beta) \intersect h(A^S_\gamma) \intersect h(A^S_\delta) \\
&= A^T_{\Phi(\beta)} \intersect A^T_{\Phi(\gamma)} \intersect A^T_{\Phi(\delta)}
\end{align*}
Since $h$ is uniquely determined by its restriction to the vertices of valence~$\ge 3$, it follows that $h$ is uniquely determined amongst simplicial isomorphisms.
\end{proof}

\begin{proof}[Proof of Lemma~\ref{LemmaTwistedUniqueness}] \quad The uniqueness clause~\pref{ItemLTUUnique} is an immediate consequence of Lemma~\ref{LemmaTwistedUniqueGeneral}. Item~\pref{ItemLTUInner} follows from the uniqueness clause~\pref{ItemLTUUnique}, because for each $\gamma \in F_n$ the map $h = \alpha(\gamma)$ clearly satisfies $i_\gamma$ twisted equivariant: $\alpha(\gamma) \composed \alpha(\delta) = \alpha(i_\gamma(\delta)) \composed \alpha(\gamma)$ for all $\delta \in F_n$.
\end{proof}

%
%

Consider now the function $\ti\alpha \from \wtStab[T] \mapsto \Isom[T]$ defined by $\ti\alpha(\Phi) = h_\Phi$ as given by Lemma~\ref{LemmaTwistedUniqueness}. This defines an action $\ti\alpha \from \wtStab[T] \act T$, because for each $\Phi,\Psi \in \wtStab[T]$ both sides of the action equation $h_\Phi \circ h_\Psi = h_{\Phi\Psi}$
clearly satisfy $\Phi\Psi$-twisted equivariance, and hence the equation holds by application of the uniqueness clause \pref{ItemLTUUnique} of Lemma~\ref{LemmaTwistedUniqueness}. 

The following lemma summarizes this discussion together with the evident generalization to subgroups of $\wtStab[T]$, and rewrites the twisted equivariance property using action notation instead of functional notation.


\begin{lemma}
\label{LemmaAutTreeAction}
Associated to each free splitting $F_k \act T$ there is a unique isometric action $\wtStab[T] \act T$ which assigns to each $\Phi \in \wtStab[T]$ the unique simplicial isomorphism $T \mapsto T$ as stated in Lemma~\ref{LemmaTwistedUniqueness}, denoted in action notation as $x \mapsto \Phi\cdot x$, satisfying the following:
\begin{description}
\item[Twisted equivariance (action notation):] For all $\Phi \in \Aut(F_n)$, $\gamma \in F_n$, $x \in T$,

\smallskip
\centerline{$\Phi \cdot (\gamma \cdot x) = \Phi(\gamma) \cdot (\Phi \cdot x)$}
\end{description}
More generally, the restriction to any subgroup $\cK \subgroup \wtStab[T]$ of the action $\wtStab[T] \act T$ is the unique isometric action $\cK \act T$ such that satisfies twisted equivariance.\qed
\end{lemma}

\subparagraph{Remark:} In the twisted equivariance equation $\Phi \cdot (\gamma \cdot x) = \Phi(\gamma) \cdot (\Phi \cdot x)$, the action dot ``$\cdot$'' is used ambiguously for both the action of $F_n$ on $T$ and the action of $\wtStab[T]$ on $T$. The meaning of any particular action dot should be clear by context. Furthermore, in contexts where the two meanings overlap they will always agree. For example, Lemma~\ref{LemmaTwistedUniqueness}~\pref{ItemLTUInner} says that the action dots always respect the standard isomorphism $F_n \approx \Inn(F_n)$ given by $\delta \approx i_\delta$.

\bigskip

The next lemma will be a key step of the proof of the loxodromic and WWPD portions of the Hyperbolic Action Theorem (see remarks after the statement).  In brief, given a free splitting $F_n \act T$ and a certain subgroup $J \subset F_n$, the lemma gives a criterion for verifying that the restriction of the free splitting action $F_n \act T$ to a certain group $J \subset F_n$ is nonelementary. To understand the statement of the lemma, the reader may note that by applying Lemma~\ref{LemmaAutTreeAction}, the entire setup of the lemma --- including hypotheses~\pref{ItemJActionsAgree} and~\pref{ItemHatHTwistedEquiv} --- is satisfied for any free splitting $F_n \act T$ and any subgroup $\wh\cH \subgroup \wtStab[T]$.

\newcommand\nt{{\text{nt}}}


\begin{lemma}
\label{LemmaJNonelementary}
Let $F_n \act T$ be a free splitting with vertex set $V$, and let $V^\nt$ be the subset of all $v \in V$ such that $\Stab(v)$ is nontrivial. Let $\wh\cH \subgroup \Aut(F_n)$ be a subgroup with normal subgroup $J = \wh\cH \intersect \Inn(F_n)$. Let $J \act T$ denote the restriction to $J$ of the given free splitting action $\Inn(F_n) \approx F_n \act T$. Let $\wh\cH \act V^\nt$ be an action that satisfies the following:
\begin{enumerate}
\item\label{ItemJActionsAgree}
The two actions $J \act V^\nt$, one obtained by restricting to $J$ the given action \hbox{$\wh\cH \act V^\nt$,} the other by restricting the action $J \act T$ to the subset $V^\nt$, \break are~identical. 
\item\label{ItemHatHTwistedEquiv}
The following twisted equivariance condition holds:
$$\Phi \cdot (\gamma \cdot x) = \Phi(\gamma) \cdot (\Phi \cdot x), \quad\text{for all $\Phi \in \wh\cH$, $\gamma \in F_n$, $x \in V^\nt$}
$$
\end{enumerate}
%
If no subgroup $\Stab_{F_n}(v)$ $(v \in V^\nt(T))$ is fixed by the whole group $\wh\cH$, and if the free group $J$ has rank~$\ge 2$, then the action $J \act T$ is nonelementary.
\end{lemma}

\textbf{Remarks.} Lemma~\ref{LemmaJNonelementary} is applied in Section~\ref{SectionOneEdgeExtension} when proving the one-edge case of the Hyperbolic Action Theorem, in which case we have a group $\wh\cH \subgroup \wtStab[T]$ for which the hypotheses \pref{ItemJActionsAgree} and \pref{ItemHatHTwistedEquiv} hold automatically (by Lemma~\ref{LemmaAutTreeAction}).

Lemma~\ref{LemmaJNonelementary} is also applied in Section~\ref{SectionSuspension} when proving the multi-edge case of the Hyperbolic Action Theorem. When Lemma~\ref{LemmaJNonelementary} is applied in that case, the group $\wh\cH \subgroup \Aut(F_n)$ is not contained in $\wtStab[T]$, and the given action $\wh\cH \act V^\nt$ does not extend to an action of $\wh\cH$ on~$T$ itself. Nonetheless $\wh\cH$ will have a kind of ``semi-action'' on $T$ which extends the given free splitting action $J \act T$, and this will be enough to give an action $\wh\cH \act V^\nt(T)$ that also satisfies~\pref{ItemJActionsAgree} and~\pref{ItemHatHTwistedEquiv}.

\begin{proof} By hypothesis~\pref{ItemJActionsAgree} the action $J \act T$ has trivial edge stabilizers and each point of $V-V^\nt$ has trivial stabilizer. It follows that for each nontrivial $\alpha \in J$, either $\alpha$ is elliptic and its fixed point set $\Fix(\alpha) \subset T$ is a single point of $V^\nt$, or $\alpha$ is loxodromic with repeller--attractor pair $(\bdy_-\alpha,\bdy_+\alpha) \subset \bdy T \times \bdy T$ and $\Stab_{J}(\bdy_-\alpha) = \Stab_{J}(\bdy_+\alpha) = \Stab_{J}\{\bdy_-\alpha,\bdy_+\alpha\}$ is an infinite cyclic group. 

We claim there is no point $v \in V^\nt(T)$ such that $\Fix(\alpha)=\{v\}$ for each nontrivial $\alpha \in J$. Otherwise that point $v$ is unique, and its stabilizer $B_v = \Stab_{F_n}(v)$ is the unique nontrivial vertex stabilizer fixed by the action of $J$ on subgroups of $F_n$, because the bijection $v \leftrightarrow B_v$ is $J$-equivariant (by Lemma~\ref{LemmaActionTwisted}). It follows that each $\Phi \in \wh\cH$ also fixes~$B_v$, because $J = \Phi J \Phi^\inv$ fixes the subgroup $\Phi(B_v) \subgroup F_n$ which is also a nontrivial vertex stabilizer, namely $\Phi(B_v)$ is equal to the stabilizer of $\Phi \cdot v$ by hypothesis~\pref{ItemHatHTwistedEquiv}; by~uniqueness of $B_v$ we therefore have $\Phi(B_v)=B_v$. Since this holds for all $\Phi \in \wh\cH$, we have contradicted the hypothesis of the lemma, thus proving the claim.

The proof that the action $J \act T$ is nonelementary now follows a standard argument. Some $\gamma \in J$ is loxodromic, for otherwise by applying the claim it follows that $J$ has nontrivial elliptic elements $\alpha,\beta$ with $\Fix(\alpha) \ne \Fix(\beta) \in T$ but in that case $\gamma=\alpha\beta$ is loxodromic (see for example \cite[Proposition 1.5]{CullerMorgan:Rtrees}). Since $\rank(J) \ge 2$, there exists $\delta \in J - \Stab_J\{\bdy_-\gamma,\bdy_+\gamma\}$. It follows that $\gamma$ and $\delta\gamma\delta^\inv$ are independent loxodromic elements of~$J$.
\end{proof}

\subsection{Hyperbolic Action Theorem, One-edge case: Proof.}
\label{SectionOneEdgeExtension}

Adopting the notation of the Hyperbolic Action Theorem, we may assume that $\cH \subgroup \IAThree$. Choose $\B$ to be a maximal, proper, $\cH$-invariant free factor system of $F_n$, and so $\cH$ is fully irreducible relative to the extension $\B \sqsubset \{[F_n]\}$, meaning that there is no free factor system $\A$ with strict nesting $\B \sqsubset \A \sqsubset \{[F_n\}$ such that $\A$ is invariant under any finite index subgroup of $\cH$, equivalently (since $\cH \subgroup \IAThree$) such that $\A$ is invariant under $\cH$. 

The proof of the Hyperbolic Action Theorem breaks into two cases: 

\textbf{The one-edge case:} The co-edge number of $\B$ equals~$1$.

\textbf{The multi-edge case:} The co-edge number of $\B$ is $\ge 2$.

\medskip\noindent\textbf{Proof of the Hyperbolic Action Theorem in the one-edge case.} Assuming that $\B \sqsubset \{[F_n]\}$ is a one-edge extension, using that $\cH \subgroup \Stab(\B)$, and applying Fact~\ref{FactOneEdge}, there exists a free splitting $F_n \act T$ with one edge orbit whose vertex stabilizer system forms the free factor system~$\B$, and we have equality of stabilizer subgroups $\Stab(\B) = \Stab[T]$. It follows that $\cH \subgroup \Stab[T]$. Applying Lemma~\ref{LemmaAutTreeAction}, consider the resulting action $\wtStab[T] \act T$. We show that the restricted action $\wh\cH \act T$ satisfies the conclusions of the Hyperbolic Action Theorem (using~$\S=T$ and~$\cN = \wh\cH$). Since every isometry of a tree is either elliptic or loxodromic, Conclusion~\pref{ItemThmF_EllOrLox} of the Hyperbolic Action Theorem holds. 

For proving Conclusion~\pref{ItemThmF_Nonelem} of the Hyperbolic Action Theorem we shall apply Lemma~\ref{LemmaJNonelementary} to the action $\wh\cH \act T$, so we must check its hypotheses. Hypotheses~\pref{ItemJActionsAgree} and~\pref{ItemHatHTwistedEquiv} of Lemma~\ref{LemmaJNonelementary} hold by applying Lemma~\ref{LemmaAutTreeAction} to the subgroup $\wh\cH \subgroup \wtStab[T]$. Also, by the hypothesis of the Hyperbolic Action Theorem, the subgroup $\wh\cH \subgroup \Aut(F_n)$ does not preserve any proper, nontrivial free factor of $F_n$, in particular it does not preserve any nontrivial vertex stabilizer of the free splitting $F_n \act T$. Finally, by hypothesis of the Hyperbolic Action Theorem, $\cH$ is abelian and $\wh\cH$ is not virtually abelian, and so it follows that $J$ has rank~$\ge 2$, for otherwise the group $\wh\cH \subgroup \Aut(F_n) \subgroup \Out(F_{k+1})$ would be solvable and hence virtually abelian (by \SubgroupsThree), a contradiction. The conclusion of Lemma~\ref{LemmaJNonelementary} therefore holds, saying that the restricted action $J \act T$ is nonelementary. The action $\wh\cH \act T$ is therefore nonelementary, which verifies conclusion~\pref{ItemThmF_Nonelem} of the Hyperbolic Action Theorem.  

Conclusion~\pref{ItemThmF_WWPD} of the Hyperbolic Action Theorem says that each loxodromic element of $J$ is a strongly axial, WWPD element for the action $\wh\cH \act T$, and this is an immediate consequence of the next lemma (which will also be used in the multi-edge case of the Hyperbolic Action Theorem). 

In formulating this lemma, we were inspired by results of Minasyan and Osin \cite[Section 4]{MinasyanOsin:TreeAcylindrical} producing WPD elements of certain group actions on trees.

\begin{lemma} 
\label{LemmaTreeWWPD}
Let $G \act T$ be a group action on a simplicial tree equipped with the simplicial metric on $T$ that assigns length~$1$ to each edge. If $J \normal G$ is a normal subgroup such that the restricted action $J \act T$ has trivial edge stabilizers, 
then each loxodromic element of $J$ is a strongly axial, WWPD element of the action $G \act T$.
\end{lemma}

\begin{proof} Given an oriented line $A \subset T$, let the stabilizers of $A$ under the actions of $G$ and of $J$ be denoted
\begin{align*}
\Stab_G(A) &= \{\gamma \in G \suchthat \gamma(A)=A\} \\
\Stab_J(A) &= J \intersect \Stab_G(A)
\end{align*}
Consider a loxodromic $\mu \in J$. Clearly the axis $A_\mu \subset T$ of $\mu$ is a strong axis of~$\mu$. Orient $A_\mu$ so that $\mu$ translates in the positive direction, with repelling/attracting endpoints $\bdy_- A_\mu, \bdy_+ A_\mu \in \bdy T$. Since $J \act T$ has trivial edge stabilizers, the group $\Stab_J(\bdy_- A_\mu) = \Stab_J(\bdy_+ A_\mu) = \Stab_J(\bdy_- A_\mu,\bdy_+A_\mu)$ is infinite cyclic.

By \cite[Section 2.2]{HandelMosher:WWPD}, for $\mu$ to satisfy WWPD with respect to the action $G \act T$ is equivalent to saying that the ordered pair $\bdy_\pm A_\mu = (\bdy_-A_\mu,\bdy_+A_\mu)$ is an isolated point in its orbit $G \cdot \bdy_\pm A_\mu \subset \bdy T \times \bdy T - \Delta$ (this is the property denoted ``WWPD~(2)'' in \cite[Section 2.2]{HandelMosher:WWPD}). We may assume that $\mu$ is a generator for the infinite cyclic group $\Stab_J(\bdy_\pm A_\mu)$; if this is not already true then, without affecting the WWPD property for $\mu$, we may replace $\mu$ with a generator. Letting $\ell_\mu>0$ denote the integer valued length of a fundamental domain for the action of $\mu$ on $A_\mu$, it follows that any edge path in $A_\mu$ of length $\ell_\mu$ is a fundamental domain for the $\Stab_J(\bdy_\pm A_\mu)$ on $A_\mu$. Choose a subsegment $\alpha \subset A_\mu$ of length $\ell_\mu+1$. There is a corresponding neighborhood $U_\alpha \subset \bdy T \times \bdy T - \Delta$ consisting of all endpoint pairs of oriented lines in $T$ containing $\alpha$ as an oriented subsegment. Consider $\gamma \in G$ such that $\gamma(\bdy_\pm A_\mu) \subset U_\alpha$. It follows that $\mu' = \gamma \mu \gamma^\inv \in J$ has axis $A_{\mu'} = \gamma(A_\mu)$ and that $A_\mu \intersect A_{\mu'}$ contains $\alpha$, and hence $A_\mu \intersect A_{\mu'}$ has length $\ge \ell_\mu+1$. Also, the map $\gamma \from A_{\mu} \to A_{\mu'}$ takes the restricted orientation of the subsegment $A_\mu \intersect A_{\mu'} \subset A_{\mu}$ to the restricted orientation on the subsegment $\gamma(A_\mu \intersect A_{\mu'}) \subset A_{\mu'}$. Let the edges of the oriented segment $A_\mu \intersect A_{\mu'}$ be parameterized as $E_0 E_1 E_2 \ldots E_J$, $J \ge \ell_\mu$. Since $\mu$ and $\mu'$ both have translation number $\ell_\mu$ it follows that $\mu(E_0)=\mu'(E_0)=E_{\ell_\mu}$, and so $\mu^\inv \mu' \in \Stab_J(E_0)$. Since $J$ has trivial edge stabilizers it follows that $\mu=\mu'$ and so $\gamma \in \Stab_J\{\bdy_- A_\mu,\bdy_+ A_\mu\} = \<\mu\>$. And having shown that $\gamma$ preserves orientation, it follows that $\gamma(\bdy_\pm A_\mu) = \bdy_\pm A_\mu$. This shows that $\bdy_\pm A_\mu$ is isolated in its orbit $G \cdot \bdy_\pm A_\mu$, being the unique element in the intersection $U_\alpha \intersect (G \cdot \bdy_\pm A_\mu)$.
\end{proof}

This completes the proof of the Hyperbolic Action Theorem in the one-edge case.

\section{Hyperbolic Action Theorem, Multi-edge case: Introduction.} 
\label{SectionMultiEdgeIntro}

The proof of the Hyperbolic Action Theorem in the multi-edge case will take up the rest of the paper. In this section, we give a broad outline of the methods of proof, followed by some motivation coming from well-known constructions in geometric group theory. 

\subsection{Outline of the multi-edge case.}
\label{SectionMultiEdgeOutline}

We will make heavy use of the theory of abelian subgroups of $\Out(F_n)$ developed by Feighn and Handel \abelian. In very brief outline here are the main features of that theory we will need. 

\smallskip

\textbf{Disintegration groups.} (\abelian, and see Section \ref{SectionDisintegration}) Any element of $\Out(F_n)$ has a uniformly bounded power which is \emph{rotationless}, meaning roughly that each of various natural finite permutations induced by that element are trivial. Any rotationless $\phi \in \Out(F_n)$ has a particularly nice relative train track representative called a~\ct. For any \ct\ $f \from G \to G$ there is an associated abelian subgroup $\D(f) \subgroup \Out(F_n)$ that contains $\phi$ and is called the ``disintegration subgroup'' of~$f$. The idea of the disintegration group is to first disintegrate or decompose $f$ into pieces, one piece for each non-fixed stratum $H_r$, equal to $f$ on $H_r$ and to the identity elsewhere. Then one re-integrates those pieces to form generators of the group $\D(f)$, by choosing a list of non-negative exponents, one per non-fixed stratum, and composing the associated powers of the pieces of $f$. However, in order for this composition to be continuous and a homotopy equivalence, and for $\D(f)$ to be abelian, the exponents in that list cannot be chosen independently. Instead two constraints are imposed: the non-fixed strata are partitioned into a collection of ``almost invariant subgraphs'' on each of which the exponent must be constant; and certain linear relations are required amongst strata that wrap around a common twist path of~$f$.

The following key theorem about disintegration groups lets us study an abelian subgroup of $\Out(F_n)$ (up to finite index) by working entirely in an appropriate disintegration group:

\begin{disintegrationtheorem}[\protect{\cite[Theorem 7.2]{\abelianTag}}]
For every rotationless abelian subgroup $\cH \subgroup \Out(F_n)$ there exists $\phi \in \cH$ such that for every \ct\ $f \from G \to G$ representing $\phi$, the intersection $\cH \intersect \D(f)$ has finite index in $\cH$.
\end{disintegrationtheorem}

\textbf{The proof of the Hyperbolic Action Theorem in the multi-edge case.} The detailed proof is carried out in Section~\ref{SectionMultiEdgeExtension}, based on material whose development we soon commence. Here is a sketch.

By a finite index argument we may assume that every element of the abelian group $\cH$ is rotationless. Let $\B$ be a maximal, proper, $\cH$-invariant free factor system of $F_n$. Being in the multi-edge case means that $\B$ has co-edge number~$\ge 2$. From the Disintegration Theorem we obtain $\phi \in \cH$, and we apply the conclusions of that theorem to a \ct\ representative $f \from G \to G$ of $\phi$ having a proper core filtration element $G_t$ representing $\B$. We may assume that $\cH \subgroup \D(f)$, by replacing $\cH$ with its finite index subgroup $\cH \intersect \D(f)$. From the construction of the disintegration group $\D(f)$, any core filtration element properly contained between $G_t$ and $G$ would represent a free factor system which is $\D(f)$-invariant, hence $\cH$-invariant, contradicting maximality of~$\B$. Thus $G_t$ is the \emph{maximal} proper core filtration element, and $\phi$ is fully irreducible relative to~$\B$. Since $\B$ has co-edge number~$\ge 2$, the top stratum $H_u$ is an \eg\ stratum. By maximality of $G_t$, every stratum strictly between $G_t$ and $G$ is either an \neg-linear edge with terminal endpoint attached to $G_t$ or a zero stratum enveloped by~$H_u$.

The hard work of the proof breaks into two major phases, the first of which is:

\begin{description}
\item[Sections~\ref{SectionFlaringOfLu}, 
\ref{SectionPathFunctionsReT}:] Construction and hyperbolicity of $\S$. 
\end{description}

The hyperbolic metric space $\S$ needed for verifying the conclusions of the Hyperbolic Action Theorem is constructed in terms of the \ct\ $f \from G \to G$ (see Section~\ref{SectionMultiEdgeMotivate} for further motivation of the construction). First we describe $\S$ in the simpler case that $f$ has no height~$u$ indivisible Nielsen path. Starting with $f \from G \to G$, lift to the universal cover to obtain $\ti f \from \wt G \to \wt G$. Let $\wt G_{u-1} \subset \wt G$ be the total lift of $G_{u-1}$. Collapse to a point each component of $\wt G_{u-1}$, obtaining a simplicial tree $T$. The map $\ti f \from \wt G \to \wt G$ induces a map $f_T \from T \to T$. Let $\S$ be the bi-infinite mapping cylinder of $f_T$, obtained from $T \cross \Z \cross [0,1]$ by identifying $(x,n,1) \sim (f_T(x),n+1,0)$ for each $x \in T$ and $n \in \Z$. 

The construction of $\S$ is more complex when $H_u$ is geometric, meaning that $f$ has a unique (up to inversion) height~$u$ indivisible Nielsen path $\rho$, and $\rho$ is closed, forming a circuit $c$ in~$G$ (\BookOne, and see \SubgroupsOne). Each line $\ti c \subset \wt G$ obtained by lifting $c$ projects to a line in $T$ called a \emph{geometric axis}, and the projections to $T$ of the lifts of $\rho$ in $\ti c$ are fundamental domains for that axis. In the bi-infinite mapping cylinder as defined above, the portion of the mapping cylinder corresponding to each geometric axis is a quasiflat, contradicting hyperbolicity. Thus we do not take $\S$ to be the bi-infinite mapping cylinder itself, instead $\S$ is constructed from the bi-infinite mapping cylinder by coning off each geometric axis in $T \cross n$ for each $n \in \Z$, using one cone point per geometric axis, attaching arcs that connect the cone point to the endpoints of the fundamental domains of that axis.

In Section \ref{SectionHypOfS}, we use the Mj-Sardar combination theorem \cite{MjSardar:combination} to prove hyperbolicity of $\S$. 
The Mj-Sardar theorem, a generalization of the Bestvina-Feighn Combination Theorem \cite{BestvinaFeighn:combination}, requires us to verify a flaring hypothesis. To do this, in Section~\ref{SectionFlaringOfLu} we study relative flaring properties of \ct\ $f \from G \to G$, specifically: how $f$ flares relative to the lower filtration element $G_{u-1}$; and in the geometric case, how $f$ flares relative to the Nielsen path~$\rho$. Then in Section~\ref{SectionPathFunctionsReT}, we study flaring properties of the induced map $f_T \from T \to T$ (and, in the geometric case, flaring properties of the induced map obtained by coning off each geometric axis of $T$). The required flaring hypothesis on $\S$ itself can then be verified in Section~\ref{SectionHypOfS}, allowing application of the Mj-Sardar theorem to deduce hyperbolicity of $\S$.

\medskip

The other major phase of the proof is:
\begin{description}
\item[Sections~\ref{SectionTTSemiAction}, \ref{SectionSuspConstr}:] Use the theory of disintegration groups to obtain an isometric action $\wh\cH \act \S$ with appropriate WWPD elements.
\end{description}
To do this, one uses that there is a number $\lambda > 0$ and a homomorphism 
$$PF_\Lambda \from \D(f) \to \Z
$$
such that for each $\psi \in \D(f)$, the lamination $\Lambda$ is an attracting lamination for $\psi$ if and only if $PF(\psi)>0$, in which case the stretch factor is $\lambda^{\PF(\psi)}$. One would like to think of $H_u$ as an \eg\ stratum for $\psi$ having Perron-Frobenius eigenvalue~$\lambda^{\PF(\psi)}$, but this does not yet make sense because we do not yet have an appropriate topological representative of $\psi$ on $G$. We define a subgroup and sub-semigroup 
\begin{align*}
\D_0(f) &= \kernel(PF_\Lambda) \\
\D_+(f) &= PF_\Lambda^\inv[0,\infty)
\end{align*}
and we then lift the sequence of inclusions $\D_0(f) \subset \D_+(f) \subset \D(f) \subgroup \Out(F_n)$ to a sequence of inclusions
to 
$$\wh D_0(f) \subset \wh D_+(f) \subset \wh D(f) \subgroup \Aut(F_n)
$$
The hard work in Section~\ref{SectionTTSemiAction} is to use the theory of disintegration groups to construct a natural action of the semigroup $\wh\D_+(f)$ on $T$, in which each element $\Psi \in \wh\D_+(f)$ acts on $T$ by stretching each edge by a uniform factor equal to $\lambda^{\PF(\psi)}$, and such that for each $\Psi \in \wh\D_+(f)$ the resulting map $x \mapsto \Psi \cdot x$ of $T$ is $\Psi$-twisted equivariant. When restricted to the subgroup $\wh\D_0$ we obtain an action on $T$ by twisted equivariant isometries, which allows us to identify the action $\wh\D_0 \act T$ with the one described in Lemma~\ref{LemmaAutTreeAction}.

Then what we do in Section~\ref{SectionSuspConstr} is to suspend the semigroup action of $\wh\D_+(f)$ on~$T$ (and, in the geometric case, on the graph obtained by coning off geometric axes), obtaining an isometric action $\wh\D(f) \act \S$. By restriction we obtain the required action $\wh\cH \act \S$.

Finally,
\begin{description}
\item[Section~\ref{SectionMultiEdgeExtension}:] Put the pieces together and verify the conclusions of the Hyperbolic Action Theorem.
\end{description}
The basis of the WWPD conclusions in the multi-edge case of the Hyperbolic Action Theorem is Lemma~\ref{LemmaTreeWWPD}, which has already played the same role for the one-edge case.

\subsection{Motivation: Suspension actions and combination theorems.} 
\label{SectionMultiEdgeMotivate}
The construction of the hyperbolic suspension action $\wh\D(f) \act \S$ may be motivated by looking at some familiar examples in a somewhat unfamiliar way. 

\medskip\textbf{Example: Mapping torus hyperbolization after Thurston.}
Consider a pseudo-Anosov homeomorphism $f \from S \to S$ of a closed, oriented hyperbolic surface $S$ of genus $\ge 2$ with associated deck action $\pi_1 S \act \wt S = \hyp^2$. The map $f$ uniquely determines an outer automorphism $\phi \in \Out(\pi_1 S)$. The associated extension group $\Gamma \subgroup \Aut(\pi_1 S)$ is the inverse image of the infinite cyclic group $\<\phi\> \subgroup \Out(\pi_1 S)$ under the natural homomorphism $\Aut(\pi_1 S) \mapsto \Out(\pi_1 S)$. A choice of $\Phi \in \Aut(\pi_1 S)$ representing $\phi$ naturally determines a semidirect product structure $\Gamma \approx \pi_1(S) \semidirect_\Phi \Z$. 

The deck transformation action $\pi_1 S \act \hyp^2$ extends naturally to an action $\Gamma \act \hyp^2$, whereby if $\Psi \in \Gamma$ projects to $\psi = \phi^k$ then the action of $\Psi$ on $\hyp^2$, denoted $F_\Psi \from \hyp^2 \to \hyp^2$, is the lift of $f^k$ whose induced action on the circle at infinity $\bdy\hyp^2$ agrees with the induced action of $\gamma$. Equivalently, $F_\Psi$ is the unique $\Psi$-twisted equivariant lift of $f^k$. Although the deck action $\pi_1 S \act \hyp^2$ is by isometries, the extended action $\Gamma \act \hyp^2$ is \emph{not} by isometries. 

However, there is a way to suspend the action $\Gamma \act \hyp^2$ obtaining an isometric action $\Gamma \act \S$, as follows. The suspension space $\S$ is the bi-infinite mapping cylinder obtained as the quotient of $\hyp^2 \times \Z \times [0,1]$ under the identifications $(x,n,1) \sim (F_\Phi(x),n+1,0)$. One may check (and we shall do in Section~\ref{SectionSuspConstr} in a different context) that there is an action $\Gamma \act \S$ which is generated by letting $\pi_1 S$ act as the deck group on $\hyp^2 \approx \hyp^2 \cross \{0\} \subset \S$ and extending naturally over the rest of $\S$, and by letting $\Phi$ according to the formula $\Phi \cdot (x,n,t)=(x,n-1,t)$. One may also check that the hyperbolic metrics on the slices $\hyp^2 \times n \times 0$ extend to a path metric on $\S$, uniquely up to quasi-isometry, such that the action $\Gamma \act \S$ is by isometries.

Returning to the land of the familiar, the group $\Gamma$ may be identified with the fundamental group of the 3-dimensional mapping torus $M_f$ of the surface homeomorphism $f \from S \to S$. By Thurston's hyperbolization theorem applied to $M_f$, the suspension space $\S$ may be identified up to $\Gamma$-equivariant quasi-isometry with the universal cover $\wt M_f \approx \hyp^3$ equipped with its deck transformation action by the group~$\Gamma \approx \pi_1 M_f$. 

\medskip\textbf{Example: Mapping torus hyperbolization after \BookZero.} Consider now an irreducible train track representative $f \from G \to G$ of a nongeometric, fully irreducible outer automorphism $\phi \in \Out(F_n)$. Again we have a natural extension group $\Gamma \subgroup \Aut(F_n)$ of $\<\phi\> \in \Out(F_n)$, with a semidirect product structure $\Gamma = F_n \semidirect_\Phi \Z$ determined by a choice of $\Phi \in \Aut(F_n)$ representing~$\phi$.

Unlike in the previous situation where we started with a surface homeomorphism, here we have started with a non-homeomorphic topological representative $f \from G \to G$ of $\phi$, and so the deck action $F_n \act \wt G$ does not extend to an action $\Gamma \act \wt G$. But it \emph{does} extend to an action of the semigroup $\Gamma_+$ which is the inverse image under $\Aut(F_n) \mapsto \Out(F_n)$ of the semigroup $\<\phi\>_+ = \<\phi^i \suchthat i \in \{0,1,2,\ldots\}\>$: for each $\Psi \in \Gamma_+$ mapping to $\phi^i$ with $i \ge 0$, the associated map $F_\Psi \from \wt G \to \wt G$ is the unique $\Psi$-twisted equivariant lift of $f^i$. Although the deck action $F_n \act \wt G$ is by isometries, the semigroup action $\Gamma_+ \act \wt G$ is \emph{not} by isometries.

But there is a way to suspend the semigroup action $\Gamma_+ \act \wt G$ to an isometric action $\Gamma \act \S$ where $\S$ is the bi-infinite mapping cyclinder of $F_\Phi \from \wt G \to \wt G$, defined exactly as above, namely the quotient space $\S$ of $\wt G \times \Z \times [0,1]$ where $(x,n,1) \sim (F_\Phi(x),n+1,0)$ with an appropriate metric. What is done in \cite[Theorem 5.1]{\BookZeroTag} is to apply properties of the train track map $f \from G \to G$ to prove a flaring hypothesis, allowing one to apply Bestvina--Feighn combination theorem \cite{BestvinaFeighn:combination} to conclude that $\S$ is Gromov hyperbolic, thus proving that $\Gamma$ is a hyperbolic group.

\medskip\textbf{Flaring methods.} Our proof of the multi-edge case will use the combination theorem of Mj and Sarder \cite{MjSardar:combination}, a generalization of the Bestvina and Feighn combination theorem \cite{BestvinaFeighn:combination}. A common feature of the above examples that is shared in the construction of this paper is that $\S$ is a ``metric bundle over $\reals$'': there is a Lipschitz projection map $\pi \from \S \mapsto \reals$ such that the minimum distance from $\pi^\inv(s)$ to $\pi^\inv(t)$ equals $\abs{s-t}$ for all $s,t \in \reals$, and each point $x \in \S$ is contained in the image of a geodesic section $\sigma \from \reals \to \S$ of the map $\pi$. In studying the large scale geometry of $\S$ it is important to study \emph{quasigeodesic} sections $\sigma \from \reals \to \S$ and their flaring properties. In our context it is convenient to translate these concepts into dynamical systems parlance: each such quasigeodesic section turns out to be a ``pseudo-orbit'' of a suspension semiflow on~$\S$ whose true orbits are the geodesic sections (see the closing paragraphs of Section~\ref{SectionMultiEdgeExtension}). The combination theorems of \cite{MjSardar:combination} and its predecessor \cite{BestvinaFeighn:combination} share key hypotheses regarding the asymptotic ``flaring'' behavior of such pseudo-orbits (see Definition~\ref{DefPathFlaring}). We remark, though, that those combination theorems hold in much more general settings, e.g.\ certain kinds of metric bundles over more general metric base spaces.

In Section~\ref{SectionFlaringOfLu} we study flaring of pseudo-orbits in the context of a relative train track map, which is then used in Section~\ref{SectionPathFunctionsReT} to extend to further contexts building up to the suspension space $\S$, its metric bundle $\S \mapsto \reals$, and its pseudo-orbits, allowing application of the Mj-Sardar combination theorem.

\section{Flaring in a top \eg\ stratum}
\label{SectionFlaringOfLu}

We assume the reader is familiar with the theory of \ct\ representatives of elements of $\Out(F_n)$, developed originally in \recognition. More narrowly, we shall focus on the terminology and notation of a \ct\ having a top \eg\ stratum, as laid out in Section~2.2 of Part I \PartOne, under the heading \textbf{\eg\ properties of \cts}. Here is a brief review, in order to fix notations for what follows in the rest of this paper; see \cite[Section 3.1]{\PartOneTag} for detailed citations drawn primarily from \recognition. At present we need no more about \cts\ than is described here; when needed in Section~\ref{SectionMoreCTs} we will give a more thorough review of~\cts.


\begin{notations}
\label{NotationsFlaring}
We fix $\phi \in \Out(F_n)$ and a relative train track representative \hbox{$f \from G \to G$} with associated $f$-invariant filtration $\emptyset = G_0 \subset G_1 \subset \cdots \subset G_u = G$ satisfying the following:
\begin{enumerate}
\item\label{ItemCTEG}
 The top stratum $H_u$ is \eg\ with Perron-Frobenius transition matrix $M_u$ having top eigenvalue $\lambda > 1$. The attracting lamination of $\phi$ corresponding to $H_u$ is denoted $\Lambda^+$ or just $\Lambda$, and its dual repelling lamination is $\Lambda^-$.
\item\label{ItemCTiNP}
There exists (up to reversal) at most one indivisible periodic Nielsen path $\rho$ of height $u$, meaning that $\rho$ is not contained in $G_{u-1}$. In this case $\rho$ and its inverse $\bar\rho$ are Nielsen paths, there is a decomposition $\rho = \alpha\beta$ where $\alpha,\beta$ are $u$-legal paths with endpoints at vertices, and $(\overline\alpha,\beta)$ is the unique illegal turn in $H_u$. At least one endpoint of $\rho$ is disjoint from $G_{u-1}$; we assume that $\rho$ is oriented with initial point $p \not\in G_{u-1}$ and terminal point $q$. The initial and terminal directions are distinct fixed directions in $H_u$. The stratum $H_u$ is geometric if and only if $\rho$ exists and is a closed path in which case $p=q$.
\item\label{ItemCTEigen}
From the matrix $M_u$ one obtains an \emph{eigenlength} function $l_\PF(\sigma)$ defined on all paths $\sigma$ in $G$ having endpoints at vertices, and having the following properties: 
\begin{enumerate}
\item for each edge $E \subset G$ we have $l_\PF(E)=0$ if and only if $E \subset G_{u-1}$; 
\item in general if $\sigma = E_1 \ldots E_K$ then $l_\PF(\sigma) = l_\PF(E_1) + \cdots + l_\PF(E_K)$; 
\item for all edges $E \subset G$ we have $l_\PF(f(E)) = \lambda \, l_\PF(E)$. 
\item If $\rho = \alpha\beta$ exists as in \pref{ItemCTiNP} then $l_\PF(\alpha)=l_\PF(\beta) = \frac{1}{2} l_\PF(\rho)$. 
\end{enumerate}
\item\label{ItemCTSplitSimple}
If $\gamma$ is a path with endpoints at vertices of $H_u$ or a circuit crossing an edge of $H_u$ then for all sufficiently large $i$ the path $f^i_\#(\gamma)$ has a splitting with terms in the set $\{$edges of~$H_u\} \union \{\rho,\bar\rho\}\union\{$paths in $G_{u-1}$ with endpoints in~$H_u\}$. 
\end{enumerate}
\end{notations}
\noindent
In what follows, there are three cases to consider: the \emph{ageometric} case where $\rho$ does not exist; the \emph{parageometric} case where $\rho$ exists but is not closed; and the \emph{geometric} case where $\rho$ exists and is closed. Occasionally we will need to consider these cases separately. But for the most part we will attempt to smoothly handle all three cases simultaneously, using the following conventions:
\begin{itemize}
\item In the ageometric case, where $\rho$ does not exist, the notation $\rho$ should simply be ignored (for an exception see the ``Notational Convention'' following Corollary~\ref{CorollarySomeFormulas});
\item In the parageometric case, where $\rho$ exists but is not closed, the notations $\rho^i$ for $\abs{i} \ge 2$ should be ignored.
\end{itemize}
This completes Notations~\ref{NotationsFlaring}.

\medskip

The main result of this section is Proposition~\ref{PropFlaringGeneral}, a uniform flaring property for the top \eg\ stratum $H_u$ of $f$. This result generalizes \cite[Theorem 5.1 and Lemma 5.6]{\BookZeroTag} which covers the special case that $f$ is a nongeometric train track map, that is, $u=1$ and if $\rho$ exists then it is not closed. There are two features which distinguish our present general situation from that special case. First, our flaring property is formulated relative to the penultimate filtration element $G_{u-1}$ of~$G$. Second, if the height~$u$ indivisible Nielsen path $\rho$ exists then our flaring property is formulated relative to~$\rho$. Taken together, these two features require very careful isolation and annihilation of the metric effects of paths in $G_{u-1}$ and of the path $\rho$. This task is particularly tricky in the geometric case where $\rho$ exists and is closed, in which situation we must also annihilate the metric effects of all iterates~$\rho^n$.

\subsection{The path functions $L_u$ and $L_\PF$.}
\label{SectionFlaringBasics}

Throughout this paper, and following standard terminology in the literature of relative train track maps, a \emph{(finite) path} in a graph $G$ is by definition a locally injective continuous function $\sigma \from [0,1] \to G$. We often restrict further by requiring that the endpoints $\sigma(0),\sigma(1)$ be vertices in which case $\sigma$ is a concatenation of edges without backtracking. A more general concatenation of edges in which locally injectivity may fail, hence backtracking may occur, will be called an ``edge~path''.

Generalizing the eigenlength function $l_\PF$ of Notations~\ref{NotationsFlaring}~\pref{ItemCTEigen}, a \emph{path function} on a graph is simply a function $l(\cdot)$ which assigns, to each finite path $\sigma$ having endpoints at vertices, a number $l(\sigma) \in [0,\infinity)$, subject to two metric-like properties:
\begin{itemize}
\item $l(\cdot)$ is symmetric, meaning that $l(\sigma)=l(\bar\sigma)$. 
\item $l(\cdot)$ assigns length zero to any trivial path. 
\end{itemize}
We do not require that $l(\sigma) > 0$ for all nontrivial paths $\sigma$. We also do not require additivity: the value of $l(\cdot)$ on a path $\sigma$ need not equal the sum of its values over the 1-edge subpaths of $\sigma$. We also do not require any version of the triangle inequality, but see Lemma~\ref{LemmaCTI} and the preceding discussion regarding coarse triangle inequalities for the path functions of most interest to~us.

Henceforth in this section we adopt adopt Notations~\ref{NotationsFlaring}.

The ``big $L$'' path functions $L_u$, $L_\PF$ are defined in terms of auxiliary ``little $l$'' path functions $l_u$, $l_\PF$ by omitting certain terms. The function $l_\PF$ is already defined in Notations~\ref{NotationsFlaring}~\pref{ItemCTEigen}. 

Consider a path $\sigma$ in $G$ with edge path decomposition $\sigma = E_1 \ldots E_K$. Define $l_u(\sigma)$ to be the number of terms $E_k$ contained in $H_u$. Now define $L_u(\sigma)$ and $L_\PF(\sigma)$ by summing $l_u(E_k)$ and $l_\PF(E_k)$ (respectively) only over those terms $E_k$ which do not occur in any $\rho$ or $\bar\rho$ subpath of~$\sigma$. Note that for each edge $E \subset G$ the ``little $l$'' eigenvector equation $l_\PF(f(E)) = \lambda \, l_\PF(E)$ implies:
\begin{description}
\item[``Big $L$'' eigenvector equation:]
$$L_\PF(f(E)) = \lambda \, L_\PF(E) \quad\text{for all edges $E \subset G$.}
$$
\end{description}
To see why this holds, if $E \subset G_{u-1}$ both sides are zero. For $E \subset H_u$ this follows from the fact that $f(E)$ is $H_u$-legal and hence has no $\rho$ or $\bar\rho$ subpath.

The following ``quasicomparability'' result is an obvious consequence of the fact that the Perron-Frobenius eigenvector of the transition matrix $M_u$ has positive entries:

\begin{lemma}
\label{LemmaUPFQI}
There is a constant $K = K_{\ref{LemmaUPFQI}}(f) \ge 1$ such that for all edges $E$ of $H_u$ we have $\frac{1}{K} L_\PF(E) \le L_u(E) \le K \cdot L_\PF(E)$. \hfill\qed
\end{lemma}

The next lemma says in part that $\rho$ and $\bar\rho$ subpaths never overlap. 

\begin{lemma}[Isolation of $\rho$]
\label{LemmaNoOverlap} 
For any path $\sigma$ in $G$ with endpoints, if any, at vertices, there is a unique decomposition $\sigma = \ldots \sigma_i \, \ldots \, \sigma_j \, \ldots$ with the following properties:
\begin{enumerate}
\item Each term $\sigma_i$ is an edge or a copy of $\rho$ or $\bar\rho$.
\item Every subpath of $\sigma$ which is a copy of $\rho$ or $\bar\rho$ is a term in the decomposition. 
\end{enumerate}
\end{lemma} 
\noindent Note that we do not assume $\sigma$ is finite in this statement. 

\medskip

Leaving the proof of Lemma~\ref{LemmaNoOverlap} for the end of Section~\ref{SectionFlaringBasics}, we give applications. For a path $\sigma$ in $G$ with endpoints at vertices, let $K_\rho(\sigma)$ denote the number of $\rho$ or $\bar\rho$ subpaths in the Lemma~\ref{LemmaNoOverlap} decomposition of $\sigma$.

\begin{corollary}[Some formulas for $L_u$ and $L_\PF$]
\label{CorollarySomeFormulas}
For any finite path $\sigma$ with endpoints at vertices, letting its Lemma~\ref{LemmaNoOverlap} decomposition be $\sigma = \sigma_1 \ldots \sigma_B$, we have
\begin{align*}
L_u(\sigma) &= 
\begin{cases}
l_u(\sigma) - K_\rho(\sigma) \cdot l_u(\rho) &\text{in general} \\
l_u(\sigma) &\text{if $\rho$ does not exist}
\end{cases} \\
\intertext{and similarly}
L_\PF(\sigma) &= 
\begin{cases}
l_\PF(\sigma) - K_\rho(\sigma) \cdot l_\PF(\rho) &\text{in general} \\
l_\PF(\sigma) &\text{if $\rho$ does not exist}
\end{cases}
\end{align*}\qed
\end{corollary}

\bigskip\noindent
\textbf{Notational convention when $\rho$ does not exist:}
Motivated by Corollary~\ref{CorollarySomeFormulas}, we extend the meaning of the notations $l_u(\rho)$ and $l_\PF(\rho)$ to the case that $\rho$ does not exist by defining  
$$l_u(\rho)=l_\PF(\rho)=0
$$

\begin{corollary}
\label{CorollaryAltDecomp}
For any finite path $\sigma$ in $G$ with endpoints at vertices, there is a unique decomposition $\sigma = \mu_0 \, \nu_1 \, \mu_1 \, \ldots \, \nu_A \, \mu_A$ with the following properties:
\begin{enumerate}
\item If $\rho$ does not exist then $A=0$ and $\sigma=\mu_0$.
\item If $\rho$ exists then each $\nu_a$ is an iterate of $\rho$ or $\bar\rho$ (the iteration exponent equals $1$ if $\rho$ is not closed), and each $\mu_a$ contains no $\rho$ or $\bar\rho$ subpath.
\item\label{ItemEveryOtherHasHu}
If $\rho$ exists, and if $1 \le a < a+1 \le A-1$, then at least one of the subpaths $\mu_a,\mu_{a+1}$ contains an edge of $H_u$; if in addition $\rho$ is closed then each $\mu_a$ subpath contains an edge of $H_u$.
\end{enumerate}
\end{corollary}
\noindent
Note that in the context of Corollary~\ref{CorollaryAltDecomp} the subpaths $\mu_a$ are nondegenerate for $1 \le a \le A-1$, but the subpaths $\mu_0$ and $\mu_A$ are allowed to be degenerate.

\begin{proof} \quad Everything is an immediate consequence of Lemma~\ref{LemmaNoOverlap} except perhaps item~\pref{ItemEveryOtherHasHu}, which follows from the fact that at least one endpoint of $\rho$ is disjoint from~$G_{u-1}$ (Notations~\ref{NotationsFlaring}~\pref{ItemCTiNP}). 
\end{proof}

The following is an immediate consequence of Lemma~\ref{LemmaNoOverlap} and Corollary~\ref{CorollaryAltDecomp}, with Lemma~\ref{LemmaUPFQI} applied term-by-term for the ``Furthermore'' part:

\begin{corollary}[More formulas for $L_u$ and $L_\PF$] 
\label{CorollaryLFormulas}
For any finite path $\sigma$ with endpoints at vertices, and using the Corollary~\ref{CorollaryAltDecomp} decomposition of $\sigma$, we have
$$
L_u(\sigma) = \sum_{a=1}^A L_u(\mu_a) = \sum_{a=1}^A l_u(\mu_a), \qquad\qquad
L_\PF(\sigma) = \sum_{a=1}^A L_\PF(\mu_a) = \sum_{a=1}^A l_\PF(\mu_a) 
$$
Furthermore, letting $K = K_{\ref{LemmaUPFQI}}(f) \ge 1$, we have
$$\hphantom{\qed}\quad\quad\quad\quad\quad\quad\quad\quad\quad\frac{1}{K} L_\PF(\sigma) \le L_u(\sigma) \le K \cdot L_\PF(\sigma)\quad\quad\quad\quad\quad\quad\quad\quad\quad\qed
$$
\end{corollary}
\noindent 
The latter inequality of Corollary~\ref{CorollaryLFormulas} gives us the freedom to switch back and forth between $L_\PF$ and $L_u$, using $L_u$ in more combinatorial situations and $L_\PF$ in more geometric situations through much of the rest of the paper.

\paragraph{Coarse triangle inequalities for $L_u$ and $L_\PF$.} The path functions $l_u$ and $l_\PF$ satisfy a version of the triangle inequality: for any two paths $\gamma,\delta$ with endpoints at vertices such that the terminal vertex of $\gamma$ equals the initial vertex of $\delta$ we have
$$l_u[\gamma \delta] \le l_u(\gamma)+l_u(\delta), \quad l_\PF[\gamma\delta] \le l_\PF(\gamma) + l_\PF(\delta)
$$
where $[\cdot]$ denotes the operation that straightens an edge path rel endpoints to obtain a path. For $L_u$ and $L_\PF$ the best we can get are coarse triangle inequalities:
\begin{lemma}
\label{LemmaCTI}
There exists a constant $C=C_{\ref{LemmaCTI}}$ such that for any finite paths $\gamma,\delta$ in $G$ with endpoints at vertices such that the terminal point of $\gamma$ coincides with the initial point of $\delta$ we have
\begin{align*}
L_u[\gamma \delta] &\le L_u(\gamma) + L_u(\delta) + C \\
L_\PF[\gamma \delta] &\le L_\PF(\gamma) + L_\PF(\delta) + C
\end{align*} 
\end{lemma}

\begin{proof} Consider the Lemma~\ref{LemmaNoOverlap} decompositions of $\gamma$ and $\delta$. We may write $\gamma=\gamma_1\gamma_2\gamma_3$, $\delta=\delta_1\delta_2\delta_3$ so that $[\gamma\delta] = \gamma_1 \gamma_2 \delta_2 \delta_3$ and so that the following hold: $\gamma_1$ is a concatenation of terms of the Lemma~\ref{LemmaNoOverlap} decomposition of $\gamma$, and $\gamma_2$ is either degenerate or an initial subpath of a single $\rho$ or $\bar\rho$ term of that decomposition; $\delta_2$ is either degenerate or a terminal subpath of a single $\rho$ or $\bar\rho$ term of the Lemma~\ref{LemmaNoOverlap} decomposition of $\delta$, and $\delta_3$ is a concatenation of terms of that decomposition. It follows that
\begin{align*}
L_u[\gamma\delta] &\le L_u(\gamma_1) + l_u(\gamma_2) + l_u(\delta_2) + L_u(\delta_3) \\
 &\le L_u(\gamma) + L_u(\delta) + 2 l_u(\rho) \\
\intertext{and similarly}
L_\PF[\gamma\delta] &\le L_\PF(\gamma) + L_\PF(\delta) + 2 l_\PF(\rho)
\end{align*}
\end{proof}

\paragraph{Quasi-isometry properties of $f$.} The next lemma describes quasi-isometry properties of the relative train track map $f \from G \to G$ with respect to the path functions~$L_u$ and $L_\PF$.

\begin{lemma} \label{LemmaLuBasics}
There exist constants $D=D_{\ref{LemmaLuBasics}} > 1$, $E=E_{\ref{LemmaLuBasics}} > 0$ such that for any finite path $\gamma$ in $G$ with endpoints at vertices, the following hold:
\begin{align}
L_u(f_\#(\gamma)) \le D \cdot L_u(\gamma) + E \quad&\text{and}\quad L_u(\gamma) \le D \cdot L_u(f_\#(\gamma)) + E \label{ItemQiU}\\
L_\PF(f_\#(\gamma)) \le D \cdot L_\PF(\gamma) + E \quad&\text{and}\quad L_\PF(\gamma) \le D \cdot L_\PF(f_\#(\gamma)) + E \label{ItemQiPF}
\end{align}
\end{lemma}

\begin{proof} Once $D,E$ are found satisfying \pref{ItemQiU}, by applying Corollary~\ref{CorollaryLFormulas} we find $D,E$ satisfying \pref{ItemQiPF}, and by maximizing we obtain $D,E$ satisfying both.

If $\rho$ exists let $P$ be its endpoint set, a single point if $H_u$ is geometric and two points otherwise; and if $\rho$ does not exist let $P=\emptyset$. Each point of $P$ is fixed by $f$. By elementary homotopy theory (see e.g.\ Lemma~\ref{LemmaGraphHomotopy}~\pref{ItemGPHomInvExists} in Section~\ref{SectionGraphHomotopy}), the map of pairs $f \from (G,P) \to (G,P)$ has a homotopy inverse $\bar f \from (G,P) \to (G,P)$ in the category of topological pairs. By homotoping $\bar f$ rel vertices, we may assume that $\bar f$ takes vertices to vertices and edges to (possibly degenerate) paths. If $\rho$ exists then, using that $\rho = f_\#(\rho)$, and applying $\bar f_\#$ to both sides, we obtain $\bar f_\#(\rho)=\bar f_\#(f_\#(\rho)) = \rho$. 

We prove the following general fact: if $g \from (G,P) \to (G,P)$ is a homotopy equivalence of pairs which takes vertices to vertices and takes edges to paths, and if $g_\#(\rho)=\rho$ assuming $\rho$ exists, then there exist $D \ge 1$, $E \ge 0$ such that for each path $\delta$ with endpoints at vertices we have
\begin{equation}
L_u(g_\#(\delta)) \le D \cdot L_u(\delta) + E \label{ItemQiG}
\end{equation}
This obviously suffices for the first inequality of \pref{ItemQiU} taking $g=f$ and $\delta=\gamma$. It also suffices for the second inequality of \pref{ItemQiU} taking $g = \bar f$ and $\delta=\gamma$ because in that case we have $g_\#(\rho) = \bar f_\#(\rho) = \bar f_\#(f_\#(\rho)) = \rho$.

To prove~\pref{ItemQiG}, consider the Corollary~\ref{CorollaryAltDecomp} decomposition $\gamma = \mu_0 \, \nu_1 \, \mu_1 \, \ldots \nu_A \, \mu_A$. Applying Corollary~\ref{CorollaryLFormulas} we have
$$L_u(\gamma) = l_u(\mu_0) + \cdots + l_u(\mu_A)
$$
We also have
$$g_\#(\gamma) = [g_\#(\mu_0) \, \nu_1 \, g_\#(\mu_1) \, \ldots \, \nu_A \, g_\#(\mu_A)]
$$
By inductive application of Lemma~\ref{LemmaCTI} combined with the fact that each $L_u(\nu_a)=0$, it follows that
\begin{align*}
L_u(g_\#(\gamma)) &\le L_u(g_\#(\mu_0)) + \cdots + L_u(g_\#(\mu_A)) + 2 \, A \, C_{\ref{LemmaCTI}} \\
 &\le l_u(g_\#(\mu_0)) + \cdots + l_u(g_\#(\mu_A)) + 2 \, A  \, C_{\ref{LemmaCTI}} \\
 &\le D'(l_u(\mu_0) + \cdots + l_u(\mu_A)) + 2 \, A  \, C_{\ref{LemmaCTI}}
\end{align*}
where $D'$ is the maximum number of $H_u$ edges crossed by $g(E)$ for any edge $E \subset G$. If $\rho$ does not exist then $A=0$ and we are done. If $\rho$ exists then, by Corollary~\ref{CorollaryLFormulas}~\pref{ItemEveryOtherHasHu}, if $1 \le a \le a+1 \le A-1$ then at least one of $l_u(\mu_a),l_u(\mu_{a+1})$ is $\ge 1$. It follows that $A \le 2(l_u(\mu_0)+\cdots+l_u(\mu_A)) + 3$ and therefore
$$L_u(g_\#(\gamma)) \le (D' + 4 \, C_{\ref{LemmaCTI}}) L_u(\gamma) + 6 \, C_{\ref{LemmaCTI}}
$$
\end{proof}


\paragraph{Proof of Lemma \ref{LemmaNoOverlap}.} The lemma is equivalent to saying that two distinct $\rho$ or $\bar\rho$ subpaths of~$\sigma$ cannot overlap in an edge. Lifting to the universal cover $\wt G$ this is equivalent to saying that for any path $\ti\sigma \subset \wt G$, if $\ti \mu$ and $\ti \mu'$ are subpaths of $\ti \sigma$ each of which projects to either $\rho$ or $\bar \rho$, and if $\ti\mu$ and $\ti\mu'$ have at least one edge in common, then $\ti \mu = \ti \mu'$.

As a first case, assume that $\ti \mu$ and $\ti \mu'$ both project to $\rho$ or both project to $\bar \rho$; up to reversal we may assume the former. Consider the decomposition $\rho = \alpha \beta$ where $\alpha,\beta$ are legal and the turn $(\bar \alpha, \beta)$ is illegal and contained in $H_u$. There are induced decompositions $\ti \mu= \ti \alpha \ti \beta$ and $\ti \mu' = \ti \alpha' \ti \beta'$. If the intersection $\ti \mu \cap \ti \mu'$ contains the height $u$ illegal turn of $\ti \mu$ or $\ti \mu'$ then $\ti \mu = \ti \mu'$ so assume that it does not. After interchanging $\ti \mu$ and $\ti \mu'$ we may assume that $\ti \mu \cap \ti \mu' \subset \ti\beta$. Projecting down to $\rho$ we see that $\beta = \beta_1 \beta_2$ and $\alpha = \alpha_1 \alpha_2$ where $\beta_2=\alpha_1$ is the projection of $\ti\mu \cap \ti\mu'$. This implies the initial directions of $\bar \alpha_1$ and of $\beta_2$ (which are respectively equal to the terminal and initial directions of $\rho$) are fixed, distinct, and in~$H_u$, and so for all $k \ge 0$ the initial directions of $f^k_\#(\bar\alpha_1)$ and of $f^k_\#(\beta_2)$ are also fixed, distinct, and in $H_u$. Using the eigenlength function $l_\PF$ (Notations~\ref{NotationsFlaring}~\pref{ItemCTEigen}) we have $l_\PF(\alpha)=l_\PF(\beta)$ and $l_\PF(\alpha_1)=l_\PF(\beta_2)$ and so $l_\PF(\alpha_2)=l_\PF(\beta_1)$. Thus when the path $f^k_\#(\alpha_1) f^k_\#(\alpha_2)f^k_\#( \beta_1)f^k_\#(\beta_2)$ is tightened to form $f^k_\#(\rho) = \rho$ for large $k$, the subpaths $f^k_\#(\alpha_2)$ and $f^k_\#( \beta_1)$ cancel each other out, and so the concatenation $f^k_\#(\alpha_1) f^k_\#(\beta_2)$ tightens to $\rho$. But this contradicts that the two terms of the concatenation are $u$-legal and that the turn $\{f^k_\#(\bar\alpha_1),f^k_\#(\beta_2)\}$ taken at the concatenation point is also $u$-legal as seen above.

The remaining case is that orientations on the projections of $\ti\mu$ and $\ti\mu'$ do not agree. In this case there is either an initial or terminal subpath of $\rho$ that is its own inverse, which is impossible. \qed

\subsection{Flaring of the path functions $L_u$ and $L_\PF$}
\label{SectionPathFlaring}

In this section we continue to adopt Notations~\ref{NotationsFlaring}. 


For each path function $l$ on $G$ and each $\eta \ge 0$ we define a relation 
$\beta \sim \gamma$
on paths $\beta,\gamma$ in $G$ with endpoints at vertices: this relation means that there exist paths $\alpha,\omega$ with endpoints at vertices such that $\gamma = [\alpha \beta \omega]$ and such that $l(\alpha),l(\gamma) \le \eta$. Note that this relation is symmetric because $\beta = [\bar\alpha \gamma \bar\omega]$. When we need to emphasize the dependence of the relation on $l$ and $\eta$ we will write more formally $\beta \, \coarsesim{\eta}{l} \, \gamma$.

\begin{definition}
\label{DefPathFlaring}
A path function $l$ on $G$ is said to satisfy the \emph{flaring condition} with respect to $f \from G \to G$ if for each $\mu > 1$, $\eta \ge 0$ there exist integers $R \ge 1$ and $A \ge 0$ such that for any sequence of paths $\beta_{-R},\beta_{-R+1},\ldots,\beta_0,\ldots,\beta_{R-1},\beta_R$ in $G$ with endpoints at vertices, the flaring inequality 
$$ \mu \cdot l(\beta_0) \le \max\{ l(\beta_{-R}), l(\beta_R) \}
$$ 
holds if the following two properties hold:

\medskip

\textbf{Threshold property:} $l(\beta_0) \ge A$;

\medskip

\textbf{Pseudo-orbit property:} $\beta_r \coarsesim{\eta}{l} f_\#(\beta_{r-1})$ for each $-R < r \le R$.

\end{definition}

\noindent\textbf{Remark.} These two properties are translations into our present context of two requirements in the original treatment of flaring by Bestvina and Feighn \cite{BestvinaFeighn:combination}: the threshold property corresponds to the requirement of large ``girth''; and the ``pseudo-orbit property'' corresponds to the ``$\rho$-thin'' requirement.

\smallskip

\begin{proposition}[General Flaring]
\label{PropFlaringGeneral}
The path functions $L_u$ and $L_\PF$ each satisfy the flaring condition with respect to $f \from G \to G$, with constants $R = R_{\ref{PropFlaringGeneral}}(\mu,\eta)$ and $A = A_{\ref{PropFlaringGeneral}}(\mu,\eta)$.
\end{proposition}

Much of the work is to prove the following, which is a version of Proposition~\ref{PropFlaringGeneral} under the special situation of Definition~\ref{DefPathFlaring} where $l=L_\PF$ and $\eta = 0$. Note that $\eta=0$ implies that $\beta_{r}=f_\#(\beta_{r-1})$, and so the choice of $\gamma = \beta_{-R}$ determines the whole sequence up to $\beta_R$.

\begin{lemma}[Special Flaring for $L_u$]
\label{LemmaSpecialFlaring}
For any $\nu>1$ there exist positive integers $N \ge 1$ and $A=A_{\ref{LemmaSpecialFlaring}} \ge 0$ so that if $\gamma$ is a finite path in $G$ with endpoints at vertices and if $L_u(f^N_\#(\gamma)) \ge A$ then 
$$\nu L_u(f^N_\#(\gamma)) \le \max\{ L_u(\gamma), L_u(f^{2N}_\#(\gamma)) \}
$$
\end{lemma}

\smallskip

The proof of this special flaring condition, which takes up Sections~\ref{SectionNegativeFlaring}--\ref{SectionFlaringProof}, is similar in spirit to \cite[Theorem 5.1]{\BookZeroTag}, and is based on two special cases: a ``negative flaring'' result, Lemma~\ref{flaring2} in Section~\ref{SectionNegativeFlaring}; and a ``positive flaring'' result, Lemma~\ref{LemmaBddBelow}~\pref{estimate3} in Section~\ref{SectionPositiveFlaring}. These two special cases are united in Section~\ref{SectionFlaringProof}, using the uniform splitting property given in Lemma~3.1 of \cite{HandelMosher:BddCohomologyI}, in order to prove the Special Flaring Condition.

Before embarking on all that, we apply Lemma~\ref{LemmaSpecialFlaring} to:

\begin{proof}[Proof of Proposition~\ref{PropFlaringGeneral}] Applying Corollary~\ref{CorollaryLFormulas} it is easy to see that $L_u$ satisfies a flaring condition if and only if $L_\PF$ satisfies a flaring condition. Thus it suffices to assume that the special flaring condition for $L_u$ holds, and use it to prove the general flaring condition for $L_u$. The idea of the proof is standard: the exponential growth given by Proposition~\ref{PropFlaringGeneral} swamps the constant error represented by the relation $\coarsesim{\eta}{L_u}$ which we will denote in shorthand as $\sim$. The hard work is to carefully keep track of various constants and other notations.

Fixing $\mu>1$ and $\eta \ge 0$, consider the relation $\sim$ given formally as $\coarsesim{\eta}{L_u}$. Fix an integer $R \ge 1$ whose value is to be determined --- once an application of Lemma~\ref{LemmaSpecialFlaring} has been arranged, $R$ will be set equal to the constant $N$ of that lemma.

Consider a sequence of paths 
$$\beta_{-R},\beta_{-R+1},\ldots,\beta_0,\ldots,\beta_{R-1},\beta_R
$$ 
in~$G$ with endpoints at vertices, such that we have
$$\beta_r \sim f_\#(\beta_{r-1}) \quad\text{for}\quad -R < r \le R
$$ 
Choose a vertex $V \in G$ which is $f$-periodic, say $f^K(V)=V$ for some integer $K \ge 1$. We may assume that $\beta_0$ has endpoints at $V$: if not then there are paths $\alpha'_0,\omega'_0$ of uniformly bounded length such that $\beta'_0 = [\alpha'_0\beta_0\omega'_0]$ has endpoints at $V$, and replacing $\beta_0$ with $\beta'_0$ reduces the flaring property as stated to the flaring property with uniform changes to the constants $\mu$, $\eta$, and~$A$.

We choose several paths with endpoints at vertices, denoted $\alpha_r,\omega_r,\gamma_r,\alpha'_r,\omega'_r$, as follows. First, using that $f_\#(\beta_{r-1}) \sim \beta_r$, there exist $\alpha_r,\omega_r$ so that $L_u(\alpha_r),L_u(\omega_r) \le \eta$ and $\beta_r = [\alpha_r \, f_\#(\beta_{r-1}) \, \omega_r]$ and hence $f_\#(\beta_{r-1}) = [\bar\alpha_r \beta_r \bar \omega_r]$. Next, anticipating an application of Lemma~\ref{LemmaSpecialFlaring}, choose $\gamma_{-R}$ so that $f^R_\#(\gamma_{-R})=\beta_0$, and then for $-R \le r \le R$ let $\gamma_r = f^{r+R}_\#(\gamma_{-R})$, hence $\gamma_0=\beta_0$. Finally, choose $\alpha'_r,\omega'_r$ so that $\gamma_r = [\alpha'_r \beta_r \omega'_r]$ and hence $\beta_r = [\bar\alpha'_r \gamma_r \bar\omega'_r]$ (in particular $\alpha'_0,\omega'_0$ are trivial paths). We also require $\alpha'_r,\omega'_r$ to be chosen so that if $r > -R$ then we have
\begin{align*}
f_\#(\beta_{r-1}) &= f_\#(\bar\alpha'_{r-1} \, \gamma_{r-1} \, \bar\omega'_{r-1}) \\
 &= [f_\#(\bar\alpha'_{r-1}) \, \gamma_r \, f_\#(\bar\omega'_{r-1})] \\
 &= [f_\#(\bar\alpha'_{r-1}) \, \alpha'_r \, \beta_r \, \omega'_r \, f_\#(\bar\omega'_{r-1})] \\
 &= \bigl[\, \underbrace{[f_\#(\bar\alpha'_{r-1}) \, \alpha'_r]}_{= \, \bar\alpha_r} \, \beta_r \, \underbrace{[\omega'_r \, f_\#(\bar\omega'_{r-1})]}_{= \, \bar\omega_r}\, \bigr] 
\end{align*}
from which we record the following identities:
$$(*) \qquad \bar\alpha_r = [f_\#(\bar\alpha'_{r-1}) \alpha'_r] \qquad \bar\omega_r = [\omega'_r f_\#(\bar\omega'_{r-1})]
$$
To see why these choices of $\gamma_r,\alpha'_r,\omega'_r$ are possible, we work with a lift to the universal cover $\ti f \from \wt G \to \wt G$. First choose any lift $\ti\beta_{-R}$. By induction for $-R< r \le R$ the lifts $\ti\alpha_r$, $\ti\omega_r$, $\ti\beta_r$ are determined by the equation $\ti\beta_r = [\ti\alpha_r \ti f_\#(\beta_{r-1}) \ti\omega_r]$. Using $f$-periodicity of $V$, the $\ti f^R$-preimages of the initial and terminal vertices of $\ti\beta_0$ contain vertices of $\wt G$ amongst which we choose the initial and terminal vertices of $\tilde\gamma_{-R}$, respectively; it follows that $\ti f^R_\#(\ti\gamma_{-R}) = \ti\beta_0$. Then define $\ti\gamma_r = \ti f^{r+R}_\#(\ti\gamma_{-R})$, define $\ti\alpha'_r$ to be the path from the initial vertex of $\ti\gamma_r$ to the initial vertex of $\ti\beta_r$, and define $\ti\omega'_r$ to be the path from the terminal vertex of $\ti\beta_r$ to the terminal vertex of $\ti\gamma_r$. Projecting down from $\wt G$ to $G$ we obtain the paths $\gamma_r,\alpha'_r,\omega'_r$. The identities $(*)$ follow from the evident fact that the paths $\ti\alpha'_r$, $\ti\alpha_r$ may be concatenated to form a path $\ti\alpha'_r \ti\alpha_r$, that the two paths $\ti\alpha'_r \ti\alpha_r$ and $f_\#(\ti\alpha'_{r-1})$ have the same initial endpoint as $f_\#(\gamma'_{r-1})=\gamma_r$, and that those two paths have the same terminal endpoint as the initial endpoint of $\ti f_\#(\ti\beta_{r-1})$; and similarly for the $\omega$'s.  

We need new upper bounds on the quantities $L_u(\alpha'_r)$ and $L_u(\omega'_r)$ which represent the difference between $L_u(\beta_r)$ and $L_u(\gamma_r)$. These bounds will be expressed in terms of the known upper bound $\eta$ on $L_u(\alpha_r)$ and $L_u(\omega_r)$, and will be derived by applying applying Lemmas~\ref{LemmaCTI} and~\ref{LemmaLuBasics} inductively, starting from $L_u(\alpha'_0)=L_u(\omega'_0)=0$ and applying $(*)$ in the induction step. These new upper bounds will then be used to derive an expression for the threshold constant $A_{\ref{PropFlaringGeneral}}$, which will be so large so that the differences $L_u(\alpha'_r)$, $L_u(\omega'_r)$ become insignificant.

The new upper bounds have the form 
$$L_u(\alpha'_{r}), L_u(\omega'_{r}) \le F_r(C,D,E,\eta) \quad\text{for $-R \le r \le R$}
$$
where for each $r$ the expression $F_r(C,D,E,\eta)$ represents a certain polynomial with integer coefficients in the variables $C,D,E,\eta$ and where we substitute 
$$C = C_{\ref{LemmaCTI}}, \quad D=D_{\ref{LemmaLuBasics}}, \quad E=E_{\ref{LemmaLuBasics}}
$$ 
The proofs of these inequalities are almost identical for $\alpha'$ and~for~$\omega'$; we carry out the proof for the latter. To start the induction, since $\gamma_0=\beta_0$ the path $\omega'_0$ degenerates to a point and so $L_u(\omega'_0)=0 \equiv F_0(C,D,E,\eta)$. Inducting in the forward direction on the interval $1 \le r \le R$, and using from $(*)$ that 
\begin{align*}
\omega'_{r} &= [\bar \omega_{r} f_\#(\omega'_{r-1})] \\
\intertext{we have}
L_u(\omega'_{r}) &\le L_u(\omega_{r}) + L_u(f_\#(\omega'_{r-1})) + C \\
 &\le \eta + \bigl(D \cdot F_{r-1}(C,D,E,\eta) + E \bigr) + C \\ &\equiv F_r(C,D,E,\eta)
\end{align*}
Inducting in the backward direction on the interval $-R \le r \le -1$, and using from $(*)$ that 
$$f_\#(\omega'_{r}) = [\omega_{r+1} \, \omega'_{r+1}]
$$
we have
\begin{align*}
L_u(\omega'_{r}) &\le D \cdot L_u(f_\#(\omega'_{r})) + E \\
 &\le D \cdot (L_u(\omega_{r+1}) + L_u(\omega'_{r+1})+C) + E \\
 &\le D (\eta + F_{r+1}(C,D,E,\eta)+C) + E \\ &\equiv F_r(C,D,E,\eta)
\end{align*} 
To summarize, we have proved
\begin{align*}
L_u(\alpha'_{-R}), L_u(\omega'_{-R}) &\le F_{-R}(C,D,E,\eta) \\
L_u(\alpha'_{R}), L_u(\omega'_{R}) &\le F_R(C,D,E,\eta)
\end{align*}
From this it follows that
\begin{align*}
L_u(\gamma_{-R}) &\le L_u(\alpha'_{-R}) + L_u(\beta_{-R}) + L_u(\omega'_{-R}) + 2C \\
 &\le L_u(\beta_{-R}) + 2F_{-R}(C,D,E,\eta) + 2C
\intertext{and similarly}
L_u(\gamma_{R}) &\le L_u(\beta_{R}) + 2F_R(C,D,E,\eta) + 2C
\end{align*}

Let $M = 2\max\{F_{-R}(C,D,E,\eta) + F_R(C,D,E,\eta)\} + 2C$, which we use below in setting the value of the threshold constant $A_{\ref{PropFlaringGeneral}}$.

Now apply Lemma~\ref{LemmaSpecialFlaring}, the special flaring condition for $L_u$, with the constant $\nu = 2\mu-1 >1$, to obtain integers $N \ge 1$ and $A_{\ref{LemmaSpecialFlaring}}$. Setting $R=N$, from the threshold requirement $L_u(\beta_0) \ge A_{\ref{LemmaSpecialFlaring}}$ it follows that 
\begin{align*}
\nu L_u(\beta_0) &= \nu L_u(\gamma_0) \\ 
 &\le \max\{L_u(\gamma_{-R}),L_u(\gamma_{R})\} \\
 &\le \max\{L_u(\beta_{-R}),L_u(\beta_{R})\} + M \\
\intertext{With the additional threshold requirement $L_u(\beta_0) \ge \frac{2 M}{\nu-1}$ we have:}
\nu L_u(\beta_0) &\le \max\{L_u(\beta_{-R}),L_u(\beta_{R})\} + \frac{\nu-1}{2} L_u(\beta_0) \\
\mu L_u(\beta_0) = \frac{\nu+1}{2} L_u(\beta_0) &\le \max\{L_u(\beta_{-R}), L_u(\beta_{R})\}
\end{align*}
Thus we have proved the (general) flaring condition for $L_u$ given any $\mu \ge 1$, $\eta \ge 0$, using $R_{\ref{PropFlaringGeneral}}=N$ and $A_{\ref{PropFlaringGeneral}} = \max\{A_{\ref{LemmaSpecialFlaring}},\frac{2M}{\nu-1}\}$.
\end{proof}


\subsection{Negative flaring of $L_u$} 
\label{SectionNegativeFlaring}

We continue to adopt Notations~\ref{NotationsFlaring}. 

In the context of Lemma~\ref{LemmaSpecialFlaring} on Special Flaring for $L_u$, the next lemma establishes flaring in the ``negative direction''. 

For any path $\sigma \subset G$, define $\lum(\sigma)$ to be the maximum of $l_u(\tau)$ over all paths $\tau$ in $G$ such that that $\tau$ is a subpath both of $\sigma$ and of some leaf $\ell$ of $\Lambda^-$ realized in~$G$. In this definition we~may always assume that $\ell$ is a generic leaf of $\Lambda^-$, because every finite subpath of every leaf is a subpath of every generic leaf. Notice that if $\sigma$ is already a subpath of a leaf of $\Lambda^-$ then $\lum(\sigma)=l_u(\sigma)$.
 
\begin{lemma} \label{flaring2} There exists $L_{\ref{flaring2}} \ge 1$ such that for each $L \ge L_{\ref{flaring2}}$ and each $a>0$ there exists an integer $N \ge 1$, so that the following holds: for each $n \ge N$, for each finite subpath $\tau$ of a generic leaf $\ell$ of $\Lambda^-$ realized in $G$ such that $\tau$ has endpoints at vertices and such that $\lum(\tau) = l_u(\tau) = L$, and for each finite path $\sigma$ in $G$, if $f^n_\#(\sigma) = \tau$ then $\lum(\sigma) \ge aL$. 
\end{lemma} 

Intuitively, at least in the nongeometric case (see the proof of Proposition 6.0.8 of \BookOne, page 609), the point of this lemma is that in any leaf of $\Lambda^-$, the illegal turns have a uniform density with respect to the path function $l_u$, and those turns die off at an exponential rate under iteration of $f_\#$.

\begin{proof} For the proof, a height~$u$ illegal turn in a path $\sigma$ is said to be \emph{out of play} if it is the illegal turn in some $\rho$ or $\bar\rho$ subpath of $\sigma$, otherwise that turn is \emph{in play}.

The generic leaf $\ell$ of $\Lambda^-$, as realized in $G$, is birecurrent, is not periodic, and is not weakly attracted to $\Lambda^+$ under iteration of $f_\#$. It follows that $\ell$ does not split as a concatenation of $u$-legal paths and copies of $\rho$ and $\bar\rho$. By applying Lemma~\ref{LemmaNoOverlap} it follows that $\ell$ has an illegal turn that is in play. In addition, $\ell$ is quasiperiodic with respect to edges of $H_u$ \cite[Lemma 3.1.8]{\BookOneTag}: for any integer $L>0$ there exists an integer $L'>0$ such that for any finite subpaths $\alpha$, $\beta$ of $\ell$, if $l_u(\alpha) = L$ and $l_u(\beta) \ge L'$ then $\beta$ contains a subpath which is a copy of $\alpha$. It follows that there exists an integer $L_{\ref{flaring2}}>0$ so that if $\tau$ is a subpath of $\ell$ such that $l_u(\tau) \ge L_{\ref{flaring2}}$, then $\tau$ has at least three height $u$ illegal turns that are in play.
 
Arguing by contradiction, if the lemma fails using the value of $L_{\ref{flaring2}}$ just given then there exist $L \ge L_{\ref{flaring2}}$, $a>0$, positive integers $n_i \to +\infinity$, finite paths $\tau_i$ in $G$ with endpoints at vertices, and paths $\sigma_i$ in $G$, such that $\tau_i$ is a subpath of $\ell$, and $l_u(\tau_i) = L$, and $f^{n_i}_\#(\sigma_i) = \tau_i$, and $\lum(\sigma_i) < a L$. We derive contradictions in separate cases. 

\textbf{Case 1: $l_u(\sigma_i)$ has an upper bound $B$.} Decompose $\sigma_i$ as $\sigma_i = \epsilon^-_i \eta_i \epsilon^+_i$ where $\epsilon^\pm_i$ are each either partial edges or trivial, and where $\eta_i$ is a path with endpoints at vertices. There exists a positive integer $d$ depending only on $B$ such that for each $i$ the path $f^d_\#(\eta_i)$ has a splitting into terms each of which is either a single edge in $H_u$, a copy of $\rho$ or $\bar\rho$, or a subpath of $G_{u-1}$, and therefore each height~$u$ illegal turn in the interior of $f^d_\#(\eta_i)$ is out of play (see \recognition\ Lemma 4.25; also see \hbox{\SubgroupsOne} Lemma~1.53). For each $i$ such that $n_i \ge d$ the paths $f^{n_i}_\#(\epsilon^\pm_i) $ are each $u$-legal, and each height $u$ illegal turn of $f^{n_i}_\#(\eta_i)= f^{n_i-d}_\#(f^d_\#(\eta_i))$ is out of play. Since $\tau_i$ is obtained from $f^{n_i}_\#(\epsilon^-_i) f^{n}_\#(\eta_i) f^{n_i}_\#(\epsilon^+_i)$ by tightening, at most two illegal turns of~$\tau_i$ are in play, a contradiction to our choice of $L_{\ref{flaring2}}$. 

\textbf{Case 2: $l_u(\sigma_i)$ has no upper bound.} Consider a line $M$ in $G$ which is a weak limit of a subsequence $\sigma_{i_m}$ such that $M$ crosses at least one edge of~$H_u$. We apply to each such $M$ the weak attraction results of \SubgroupsThree\ as follows. If $H_u$ is non-geometric then there are two options for $M$: either the closure of $M$ contains $\Lambda^-$; or $M$ is weakly attracted to $\Lambda^+$ (see Lemma 2.18 of \SubgroupsThree). If $H_u$ is geometric --- and so a height $u$ closed indivisible Nielsen path $\rho$ exists --- then there is a third option, namely that $M$ is a bi-infinite iterate of $\rho$ or~$\bar \rho$ (see Lemma 2.19 of \SubgroupsThree). 

Since no $\sigma_i$ contains a subpath of a leaf of $\Lambda^-$ that crosses $aL$ edges of $H_u$, neither does~$M$. This shows that the closure of $M$ does not contain $\Lambda^-$. If $M$ were weakly attracted to $\Lambda^+$ then for any $K>0$ there would exist $l> 0$ such that $f^l_\#(M)$, and hence $f^l_\#(\sigma_{i_m})$ for all sufficiently large $m$, contains a $u$-legal subpath that crosses $2K+1$ edges of $H_u$. By \BookOne\ Lemma 4.2.2 we can choose $K$ so that $f^l_\#(\sigma_{i_m})$ splits at the endpoints of the middle $H_u$ edge of that subpath, and hence the number of edges of $H_u$ crossed by $\tau_{i_m} = f_\#^{n_{i_m}-l}(f^l_\#(\sigma_{i_m}))$ goes to infinity with $m$, contradicting that $l_u(\tau_i)=L$.

We have shown that each of the first two options leads to a contradiction for any choice of $M$ as above. This concludes the proof if $H_u$ is nongeometric. 

It remains to show that if $H_u$ is geometric then the third option can also be avoided by careful choice of~$M$, and hence the desired contradiction is achieved. That is, we show that there exists a weak limit $M$ of a subsequence of $\sigma_i$ such that $M$ contains at least one edge of $H_u$ and $M$ is not a bi-infinite iterate of the closed path $\rho$ or $\bar \rho$. This may be done by setting up an application of Lemma 1.11 of \SubgroupsThree, but it is just as simple to give a direct proof. Lift~$\sigma_i$ to the universal cover of $G$ and write it as an edge path $\ti \sigma_i = \wt E_{i1} \wt E_{i2} \ldots \wt E_{iJ_i} \subset \wt G$; the first and last terms are allowed to be partial edges. Let $b$ equal twice the number of edges in~$\rho$. Given $j \in \{1+b,\ldots,J_i-b\}$, we say that $\wt E_{ij}$ is \emph{well covered} if $\wt E_{i,j-b}$, $\wt E_{i,j+b}$ are full edges and there is a periodic line $\ti \rho_{ij} \subset \wt G$ that projects to $\rho$ or to $\bar \rho$ and that contains $\wt E_{i,j-b}\ldots \wt E_{ij}\ \ldots \wt E_{i,j+b}$ as a subpath. Since the intersection of distinct periodic lines cannot contain two fundamental domains of both lines, $\ti \rho_{ij}$ is unique if it exists. Moreover, if both $\wt E_{ij}$ and $\wt E_{i,j+1}$ are well covered then $\ti \rho_{ij}=\ti \rho_{i,j+1}$. It follows that if $\wt E_{ij}$ is well covered then we can inductively move forward and backward past other well covered edges of $\tilde\sigma_i$ all in the same lift of $\rho$, until either encountering an edge that is not well covered, or encountering initial and terminal subsegments of $\ti\sigma_i$ of uniform length. After passing to a subsequence, one of the following is therefore satisfied:
\begin{enumerate}
 \item \label{good weak limit} There exists a sequence of integers $K_i$ such that $1 < K_i < J_i$, and $K_i \to \infinity$, and $J_i - K_i \to \infty$, and such that $\wt E_{iK_i} \subset \wt H_u$ is not well covered. 
 \item \label{just rho} $\sigma_i = \alpha_i \rho^{p_i} \beta_i$ where the number of edges crossed by $\alpha_i$ and $\beta_i$ is bounded independently of $i$ and $\abs{p_i} \to \infty$.
\end{enumerate} 
If subcase~\pref{good weak limit} holds then the existence of a weak limit that crosses an edge of $H_u$ and is not a bi-infinite iterate of $\rho$ or $\bar \rho$ follows immediately. If subcase~\pref{just rho} holds, $\tau_i$ is obtained from $f^{n_i}_\#(\alpha_i) \rho^{p_i}f^{n_i}_\#(\beta_i)$ by tightening. Since the number of illegal turns of height $u$ in the first and last terms is uniformly bounded, the number of edges that are cancelled during the tightening is uniformly bounded, and it follows that $\tau_i$ contains $\rho^{q_i}$ as a subpath where $|q_i| \to \infty$, a contradiction to the fact that $\rho$ is not a leaf of~$\Lambda^-$. 
\end{proof} 

\subsection{Positive flaring and other properties of $L_u$}
\label{SectionPositiveFlaring}

In this section we continue to adopt Notations~\ref{NotationsFlaring}.

In the context of Lemma~\ref{LemmaSpecialFlaring} on Special Flaring for $L_u$, item~\pref{estimate3} of Lemma~\ref{LemmaBddBelow} establishes flaring in the ``positive direction''. 
Lemma~\ref{LemmaBddBelow} also includes several useful metrical/dynamical properties of $L_u$, which will be used in Section~\ref{SectionFlaringProof} where we tie together the negative and positive flaring results to prove Lemma~\ref{LemmaSpecialFlaring}.

The lemma and its applications make use of the path map $f_\shsh$ and its properties as laid out in \SubgroupsOne\ Section 1.1.6, particularly Lemma~1.6 which we refer to as the ``$\shsh$ Lemma''. Roughly speaking, the path $f_\shsh(\sigma)$ is defined for any finite path $\sigma$ in $G$ as follows: for any finite path $\tau$ containing $\sigma$ as a subpath, the straightened image $f_\#(\tau)$ contains a subpath of $f_\#(\sigma)$ that is obtained by deleting initial and terminal subpaths of $f_\#(\sigma)$ that are no longer than the bounded cancellation constant of~$f$; the path $f_\shsh(\sigma)$ is the longest subpath of $f_\#(\sigma)$ which survives this deletion operation for all choices of~$\tau$. 


\begin{lemma} \label{LemmaBddBelow} 
The following conditions hold:
\begin{enumerate}
\item \label{ItemAlmostAdditive} 
For each path $\sigma$ with endpoints at vertices and any decomposition into subpaths $\sigma = \sigma_0 \sigma_1 \ldots \sigma_K$ with endpoints at vertices we have

\smallskip
\centerline{$\sum_0^K L_u(\sigma_k) - K \, l_u(\rho) \le L_u(\sigma ) \le \sum_0^K L_u(\sigma_k)$}
and similarly with $L_\PF$ and $l_\PF$ in place of $L_u$ and $l_u$ respectively.
\item \label{estimate1} For all positive integers $d$ there exist $B = B_{\ref{LemmaBddBelow}}(d) > 0$ and $b = b_{\ref{LemmaBddBelow}}(d) > 0$ so that for all finite paths $\sigma$ with endpoints at vertices, if $L_u(\sigma) \ge B$ then $L_u(f^d_{\shsh}(\sigma)) \ge b \, L_u(\sigma)$.
\item \label{estimate2} There exist constants $A > 0$ and $\kappa  \ge 1$ so that for all subpaths $\tau$ of leaves of $\Lambda^-$ in $G$, if $l_u(\tau) \ge A$ then 
$$l_u(\tau) \le L_u(\tau) \le \kappa \, l_u(\tau)
$$
\item \label{estimate3} There exists a positive integer $A = A_{\ref{LemmaBddBelow}}$ and a constant $0< R = R_{\ref{LemmaBddBelow}} <1$, so that if a path $\alpha \subset G$ splits as a concatenation of $u$-legal paths and copies of $\rho$ or $\bar \rho$, and if $L_u(\alpha) \ge A$, then for all $m \ge 1$ we have
$$L_u(f^m_{\shsh}(\alpha)) \ge R \, \lambda^{m/2} \, L_u(\alpha)
$$
where $\lambda$ is the expansion factor for $f$ and $\Lambda^+$. 
\end{enumerate}
\end{lemma}

\paragraph{Remark:} The notation $f^m_\shsh$ is disambiguated by requiring that the exponent binds more tightly than the $\shsh$-operator, hence $f^m_\shsh = (f^m)_\shsh$ --- this is how items~\pref{estimate1} and~\pref{estimate3} are applied in what follows. Note that this makes the statements of~\pref{estimate1} and~\pref{estimate3} weaker than if $f^m_\shsh$ were intepreted as $(f_\shsh)^m$, because $(f_\shsh)^m(\alpha)$ is a subpath of~$(f^m)_\shsh(\alpha)$.

\begin{proof} We prove~\pref{ItemAlmostAdditive} for $L_u$ when $K=1$; the proof for general $K$ follows by an easy induction, and the statement for $L_\PF$ is proved in the exact same manner. If there are non-trivial decompositions $\sigma_0 = \sigma'_0 \sigma''_0$ and $\sigma_1 =\sigma'_1 \sigma''_1$ such that $\sigma''_0\sigma'_1$ is a copy of $\rho$ or $\bar \rho$ then $H_u$-edges in $\sigma''_0\sigma'_1$ contribute to $L_u(\sigma_0) + L_u(\sigma_1)$ but not to $L_u(\sigma)$, and all other $H_u$-edges contribute to $L_u(\sigma_0) + L_u(\sigma_1)$ if and only if they contribute to $L_u(\sigma)$. In this case $L_u(\sigma) = L_u(\sigma_0) + L_u(\sigma_1) - l_u(\rho)$. If there are no such non-trivial decompositions then concatenating the decompositions of $\sigma_0$ and $\sigma_1$ given by Lemma~\ref{LemmaNoOverlap} produces the decomposition of $\sigma$ given by Lemma~\ref{LemmaNoOverlap} and $L_u(\sigma) = L_u(\sigma_0) + L_u(\sigma_1)$. This completes the proof of \pref{ItemAlmostAdditive}.
 
 \smallskip
 
We next prove \pref{estimate1}. Fix the integer~$d \ge 1$. The path $f^d_\shsh(\sigma)$ is obtained from $f^d_{\#}(\sigma)$ by removing initial and terminal segments that cross a number of edges that is bounded above, independently of $\sigma$, by the bounded cancellation constant of $f^d$. It follows that $\abs{L_u(f^d_{\shsh}(\sigma))-L_u(f^d_{\#}(\sigma))}$ is bounded above independently of $\sigma$, so it suffices to find $B,b > 0$ depending only on $d$ so that if $L_u(\sigma) > B$ then $L_u(f^d_{\#}(\sigma)) > b L_u(\sigma)$. Applying Lemma~\ref{LemmaLuBasics}~\pref{ItemQiU} we obtain $D > 1$, $E > 0$ so that $L_u(\gamma) \le D \, L_u(f_\#(\gamma)) + E$ for all finite paths $\gamma$ with endpoints at vertices, from which it follows by induction that
\begin{align*}
L_u(f^d_\#(\gamma)) &\ge \frac{1}{D^d} L_u(\gamma) - E\biggl( \frac{1}{D} + \cdots + \frac{1}{D^d} \biggr) \\
  &\ge \frac{1}{D^d} L_u(\gamma) - \frac{E}{D-1} \\
\intertext{and so if $L_u(\gamma) \ge \frac{2D^d E}{D-1}$ then}
L_u(f^d_\#(\gamma))   &\ge \frac{1}{D^d} L_u(\gamma) - \frac{1}{2D^d} L_u(\gamma) = \frac{1}{2D^d} L_u(\gamma)
\end{align*}

\smallskip

The first inequality of~\pref{estimate2} follows immediately from Corollary~\ref{CorollarySomeFormulas}, for any subpath $\tau$ of a leaf of $\Lambda^-$. To prepare for proving the second inequality, given a positive integer $C$ let $\Sigma_C$ be the set of paths in $G$ that do not contain a subpath of the form $\rho^{\epsilon C} $ where $\epsilon=\pm 1$. Since $\rho$ has an endpoint that is not contained in $G_{u-1}$ (\SubgroupsOne\ Fact 1.43), each maximal subpath of $\sigma$ of the form $\rho^k$ or $\bar \rho^k$ that is neither initial nor terminal in $\sigma$ is adjacent in $\sigma$ to an $H_u$-edge that contributes to $L_u(\sigma)$. Applying Corollary~\ref{CorollarySomeFormulas} it follows that
 \begin{description}
 \item [$(*)$] For any fixed $C$, amongst those paths $\sigma \in \Sigma_C$ for which $L_u(\sigma) > 0$, the ratio $l_u(\sigma) / L_u(\sigma)$ has a positive lower bound that is independent of $\sigma$. 
 \end{description} 
%
For proving \pref{estimate2}, the key observation is that there exists a positive integer $C$ so that each subpath $\tau$ of $\Lambda^-$ is contained in $\Sigma_C$ --- this is equivalent to saying that $\rho^\infty$ is not a weak limit of lines in $\Lambda^-$, which is equivalent to saying that $\rho^\infty$ is not a leaf of~$\Lambda^-$, which follows from Lemma~3.1.15 of \BookOne\ and the fact that $\Lambda^- \ne \rho^\infty$. Item \pref{estimate2} therefore follows by combining this observation with~$(*)$.

\smallskip

For proving \pref{estimate3}, we focus primarily on an analogue for $L_\PF$, connecting it with $L_u$ version stated in~\pref{estimate3} by applying Corollary~\ref{CorollaryLFormulas}. From the assumption on a splitting of $\alpha$ we have
$$L_\PF(f^m_\#(\alpha)) = \lambda^m L_\PF(\alpha)
$$
We shall show how to replace $f^m_\#$ by $f^m_\shsh$, at the expense of replacing $\lambda$ by its square root, and of requiring $L_\PF(\alpha)$ to exceed some threshold constant. To be precise, we have:
\begin{description}
\item[Claim:] There exists $A' \ge 0$ such that if $L_\PF(\alpha) \ge A'$ then for all $m \ge 0$ we have
$$L_\PF(f^m_\shsh(\alpha)) \ge \lambda^{m/2} L_\PF(\alpha)
$$
\end{description}
This suffices to prove~\pref{estimate3}, because if $L_u(\alpha) \ge K_{\ref{LemmaUPFQI}}(f) \cdot A' = A$ then from Corollary~\ref{CorollaryLFormulas} it follows that $L_\PF(\alpha) \ge A'$, from which using the Claim we obtain $L_\PF(f^m_\shsh(\alpha)) \ge \lambda^{m/2} L_\PF(\alpha)$, and then by two more applications of Corollary~\ref{CorollaryLFormulas} we obtain
$$
L_u(f^m_\shsh(\alpha)) \ge \frac{1}{K_{\ref{LemmaUPFQI}}(f)} L_\PF(\alpha) \ge \frac{1}{K_{\ref{LemmaUPFQI}}(f)} \lambda^{m/2} L_\PF(\alpha) \ge \frac{1}{(K_{\ref{LemmaUPFQI}}(f))^2} \lambda^{m/2} L_u(\alpha)
$$

To prove the claim, the case $m=0$ is evident, so suppose by induction that
$$L_\PF(f^{m-1}_\shsh(\alpha)) \ge \lambda^{(m-1)/2} L_\PF(\alpha)
$$
Since $f^{m-1}_\shsh(\alpha)$ is a subpath of $f^{m-1}_\#(\alpha)$, and since the latter splits into terms each of which is an edge of $H_u$, a copy of $\rho$ or $\bar\rho$, or a path in $G_{u-1}$, it follows that $f^{m-1}_\shsh(\alpha)$ may be deconcatenated in the form
$$f^{m-1}_\shsh(\alpha) = \zeta \hat\alpha \omega
$$
such that $\hat\alpha$ splits into terms exactly as above, and such that either $\zeta,\omega$ are both trivial, or $\rho$ exists and $\zeta,\omega$ are both proper subpaths of $\rho$ or $\bar\rho$; it follows that $l_\PF(\zeta),l_\PF(\omega) \le l_\PF(\rho)$.  Applying item~\pref{ItemAlmostAdditive} it follows that
\begin{align*}
L_\PF(f^{m-1}_\shsh(\alpha)) &\le L_\PF(\hat\alpha) + 2 l_\PF(\rho) \\
L_\PF(\hat\alpha) &\ge \lambda^{(m-1)/2} \, L_\PF(\alpha) -  2 l_\PF(\rho) \\
\intertext{Using the splitting of $\hat\alpha$ we obtain}
L_\PF(f_\#(\hat\alpha)) &= \lambda L_\PF(\hat\alpha) \ge \lambda^{(m+1)/2} \, L_\PF(\alpha) -  2 \, \lambda \, l_\PF(\rho)
\end{align*}
The path $f_\shsh(\hat\alpha)$ is obtained from $f_\#(\hat\alpha)$ by truncating initial and terminal segments no longer than the bounded cancellation constant of~$f$, and since this is a finite number of paths their $L_\PF$-values have a finite upper bound $C_2$, so by applying item~\pref{ItemAlmostAdditive} it follows that
$$L_\PF(f_\shsh(\hat\alpha)) \ge \lambda^{(m+1)/2} \, L_\PF(\alpha) -  2 \, \lambda \, l_\PF(\rho) - 2 C_2
$$
Now we apply the \shsh\ Lemma. Since $\hat\alpha$ is a subpath of $f^{m-1}_\shsh(\alpha)$ it follows that $f_\shsh(\hat\alpha)$ is a subpath of $f_\shsh(f^{m-1}_\shsh(\alpha))$ (\SubgroupsOne\ Lemma 1.6 (3)), which is a subpath of $f^m_\shsh(\alpha)$ (\SubgroupsOne\ Lemma 1.6 (4)). Thus we have $f^m_\shsh(\alpha) = \eta f_\shsh(\hat\alpha) \theta$ for some paths $\eta,\theta$, and hence by item~\pref{ItemAlmostAdditive} we have
\begin{align*}
L_\PF(f^m_\shsh(\alpha)) &\ge \lambda^{(m+1)/2} \, L_\PF(\alpha) -  2 \, \lambda \, l_\PF(\rho) - 2 C_2 - 2 \, l_\PF(\rho) \\
\intertext{To complete the induction we show that with appropriate threshold constant the quantity on the right is $\ge \lambda^{m/2} \, L_\PF(\alpha)$, equivalently}
\lambda^{(m+1)/2} \, L_\PF(\alpha) &\ge \lambda^{m/2} \, L_\PF(\alpha)   + \biggl(\underbrace{2 \, (\lambda+1) \, l_\PF(\rho) + 2 C_2}_{C_1}  \biggr) \\
 \lambda^{1/2} &\ge 1 + \frac{C_1}{\lambda^{m/2} \, L_\PF(\alpha)} 
\end{align*}
Since $\lambda>1$ is suffices to show
$$\lambda^{1/2} \ge 1 + \frac{C_1}{L_\PF(\alpha)}  \quad
\iff \quad L_\PF(\alpha) \ge \frac{C_1}{\lambda^{1/2} - 1}
$$
Taking the threshold constant to be 
$$A' = \frac{C_1}{\lambda^{1/2} - 1}
$$
the induction is complete.
\end{proof}

\subsection{Proof of Lemma \ref{LemmaSpecialFlaring}: the Special Flaring Condition for~$L_u$}
\label{SectionFlaringProof}
Once this proof is complete, the proof of the General Flaring Condition stated in Proposition~\ref{PropFlaringGeneral} will also be complete, as shown in Section~\ref{SectionPathFlaring}.

\newcommand\TwoNu{\marginpar{\tiny $2 \to \nu$ checked}}

If the Special Flaring Condition for $L_u$ fails then there exists a sequence $n_k \to \infty$, and there exist paths $\gamma_k \subset G$ with endpoints at vertices, such that $L_u(f^{n_k}_\#(\gamma_k)) \to \infty$ as $k \to \infty$, and such that 
$$(*) \qquad \nu \, L_u(f^{n_k}_\#(\gamma_k))  \ge \max\{ L_u(\gamma_k), L_u(f^{2 n_k}_\#(\gamma_k))\}
$$
Assuming this, we argue to a contradiction.

Consider the integer $L_{\ref{flaring2}} \ge 1$ satisfying the conclusions of Lemma~\ref{flaring2}. By Lemma~\ref{LemmaBddBelow}~\pref{estimate2} there is an integer $L_2$ so that if $\mu$ is a subpath of $\Lambda^-$ that crosses $\ge L_2$ edges of $H_u$ then $L_u(\mu) \ge 1$. Let $L_1 = \max\{L_{\ref{flaring2}},L_2\}$. Choose an integer $d \ge 1$ satisfying the conclusion of \cite[Lemma~3.1]{HandelMosher:BddCohomologyI}, the ``uniform splitting lemma'', with respect to the constant $L_1$. This conclusion says that for any finite path $\sigma$ in $G$ with endpoints at vertices of $H_u$, if $\ell^-_u(\sigma) < L_1$ then the path $f^d_\#(\sigma)$ splits into terms each of which is $u$-legal or a copy of $\rho$ or $\bar\rho$. In this context we shall refer to $d$ as the ``uniform splitting exponent''.
 
Let $\{\mu_{ik}\}$ be a maximal collection of subpaths of $f^{n_k}_\#(\gamma_k)$ with endpoints at vertices that have disjoint interiors, that are subpaths of $\Lambda^-$, and that cross $\ge L_1$ edges of~$H_u$. The complementary subpaths $\{\nu_{jk}\}$ of $f^{n_k}_\#(\gamma_k)$ all satisfy $\lum(\nu_{jk}) < L_1$ and all have endpoints at vertices as well.
 
Our first claim is that 
$$(i) \qquad\qquad \lim_{k \to \infty} \frac{\sum_i L_u(\mu_{ik})}{L_u(f^{n_k}_\#(\gamma_k))} = 0
$$
If not, then after passing to a subsequence we may assume that 
$$\sum_i L_u(\mu_{ik}) > \epsilon_1 L_u(f^{n_k}_\#(\gamma_k))
$$
for some $\epsilon_1> 0$ and all $k$. Choose subpaths $\sigma_{ik}$ of $\gamma_{k}$ with disjoint interiors such that $f_\#^{n_k}(\sigma_{ik}) = \mu_{ik}$. Since $l_u(\mu_{ik}) \ge L_1 \ge L_{\ref{flaring2}}$, and since $n_k \to +\infty$, we may apply ``Negative Flaring'', Lemma~\ref{flaring2}, to obtain subpaths $\sigma_{ik}'$ of $\sigma_{ik}$ which have endpoints at vertices and which are also subpaths of~$\Lambda^-$ such that for all $i$ we have
$$\lim_{k \to \infty} \frac{l_u(\sigma'_{ik})}{l_u(\mu_{ik})}= \infty
$$
The ratios $l_u(\sigma'_{ik})/L_u(\sigma'_{ik})$ and $l_u(\mu_{ik})/L_u(\mu_{ik})$ have positive upper and lower bounds independent of $i$ and $k$: the upper bound of $1$ follows from Corollary~\ref{CorollarySomeFormulas}; and the lower bound comes from Lemma~\ref{LemmaBddBelow}~\pref{estimate2}. For all $i$ we therefore obtain 
%
%
%
$$\lim_{k \to \infty} \frac{L_u(\sigma'_{ik})}{L_u(\mu_{ik})}= \infty
$$
Using this limit, and using that $L_u(\mu_{ik}) \ge 1$, it follows that for all sufficiently large $k$ we have
$$L_u(\gamma_k) \ge \sum_i (L_u(\sigma'_{ik}) - 2\abs{\rho}) > \frac{\nu}{\epsilon_1}\sum_i L_u(\mu_{ik}) > \nu \, L_u(f^{n_k}_\#(\gamma_k))
$$
where the first inequality follows by applying Lemma~\ref{LemmaBddBelow}~\pref{ItemAlmostAdditive}
to the subdivision of $\gamma_k$ into the paths $\sigma'_{ik}$ and their complementary subpaths. This contradicts $(*)$, verifying the first claim.

Our second claim is that for any constant $A$ (on which constraints will be placed below when this claim is applied) we have
$$(ii) \qquad\qquad \lim_{k \to \infty} \sum_{L_u(\nu_{jk})\ge A} \!\!\!\!\!\! L_u(\nu_{jk}) \!\! \biggm/ \!\!  L_u(f^{n_k}_\#(\gamma_k))= 1
$$
To see why, let $I_k$ be the number of $\mu_{ik}$ subpaths of $f^{n_k}_\#(\gamma_k)$, let $J_k$ be the number of $\nu_{jk}$ subpaths, and let $K_k = I_k + J_k$, and so $J_k \le I_k + 1$ and $K_k \le 2I_k+1$. By Lemma~\ref{LemmaBddBelow}~\pref{ItemAlmostAdditive} applied to $f^{n_k}_\#(\gamma_k)$ we obtain
\begin{align*}
L_u(f^{n_k}_\#(\gamma_k)) \, \le \, \sum_j L_u(\nu_{jk}) \, &+ \, \sum_i L_u(\mu_{ik}) \\ & \le \, L_u(f^{n_k}_\#(\gamma_k)) + \, K_k \, l_u(\rho) \\ \\
1 \le \frac{\sum_{L_u(\nu_{jk}) \ge A} L_u(\nu_{jk}) }{L_u(f^{n_k}_\#(\gamma_k))} &+ \underbrace{\frac{\sum_i L_u(\mu_{ik})}{L_u(f^{n_k}_\#(\gamma_k))}}_{\delta_k} + \underbrace{\frac{\sum_{L_u(\nu_{jk})<A}{ L_u(\nu_{jk})}}{L_u(f^{n_k}_\#(\gamma_k))}}_{\epsilon_k} \\ & \le 1 + \underbrace{\frac{K_k \, l_u(\rho)}{L_u(f^{n_k}_\#(\gamma_k))}}_{\zeta_k}
\end{align*}
From $(i)$ it follows that $\delta_k \to 0$ as $k \to +\infinity$. Multiplying the inequality 
$K_k \le 2I_k+1$ by $l_u(\rho) / L_u(f^{n_k}_\#(\gamma_k))$, and using that $L_u(\mu_{ik}) \ge 1$, it follows that 
$$0 \le \zeta_k \le 2 \, l_u(\rho) \, \delta_k + \frac{l_u(\rho)}{L_u(f^{n_k}_\#(\gamma_k))}
$$
and so $\zeta_k \to 0$ as $k \to +\infinity$. Multiplying the inequality $J_k \le I_k + 1$ by $A / L_u(f^{n_k}_\#(\gamma_k))$, it follows that 
$$0 \le \epsilon_k \le A \delta_k + \frac{A}{L_u(f^{n_k}_\#(\gamma_k))}
$$
and so $\epsilon_k \to 0$ as $k \to +\infinity$. This proves the second claim. 

In what follows we will be applying Lemma~\ref{LemmaBddBelow}~\pref{estimate3}, and we will use the constants $A_{\ref{LemmaBddBelow}}$, $R_{\ref{LemmaBddBelow}}$ involved in that statement.

By definition of $L_1$ and by the choice of the uniform splitting exponent $d$, since $\ell^-_u(\nu_{jk}) < L_1$ it follows that $f^d_\#(\nu_{jk})$ splits into terms each of which is either $u$-legal or a copy of $\rho$ or $\bar\rho$. Consider the constants $B = B_{\ref{LemmaBddBelow}}(d) > 0$ and $b = b_{\ref{LemmaBddBelow}}(d) > 0$ of Lemma~\ref{LemmaBddBelow}. Constraining $A \ge B$, we may combine $(ii)$ with Lemma~\ref{LemmaBddBelow}~\pref{estimate1} to obtain
$$(iii) \quad \frac{\sum_{L_u(\nu_{jk})\ge A} L_u(f^d_{\shsh}(\nu_{jk})) } {L_u(f^{n_k}_\#(\gamma_k))} \ge b \cdot \frac{\sum_{L_u(\nu_{jk})\ge A} L_u(\nu_{jk}) }{L_u(f^{n_k}_\#(\gamma_k))} > 3b/4
$$
for sufficiently large values of~$k$.

By construction, the paths $\{\nu_{jk}\}$ occur in order of the subscript $j$ as subpaths of $f^{n_k}_\#(\gamma_k)$ with disjoint interiors. By applying the $\shsh$ Lemma using $f^{n_k}$, it follows that the the paths $f^{n_k}_\shsh(\nu_{jk})$ occur in order as subpaths of the path $f^{2n_k}_\#(\gamma_k)$ with disjoint interiors (\SubgroupsOne\ Lemma 1.6~(5)). It then follows that $f^{n_k-d}_{\shsh}f^{d}_{\shsh}(\nu_{jk})$ is a subpath of $f^{n_k}_\shsh(\nu_{jk})$ (\SubgroupsOne\ Lemma 1.6~(4)). Putting these together we see that the paths $f^{n_k-d}_{\shsh}f^{d}_{\shsh}(\nu_{jk})$ occur in order as subpaths of the path $f^{2n_k}_\#(\gamma_k)$ with disjoint interiors. These subpaths being $J_k$ in number, together with their complementary subpaths one has a decomposition of $f^{2n_k}_\#(\gamma_k)$ into at most $2J_k+1$ paths. Ignoring the complementary subpaths, Lemma~\ref{LemmaBddBelow}~\pref{ItemAlmostAdditive} therefore implies 
\begin{align*} L_u(f^{2n_k}_\#(\gamma_k)) 
& \ge \sum_{\hphantom{L_u(\nu_{jk})\ge A}} L_u(f^{n_k-d}_{\shsh}f^{d}_{\shsh}(\nu_{jk})) \quad - \quad 2 \, J_k \, l_u(\rho) \\
& \ge \sum_{L_u(\nu_{jk})\ge A} L_u(f^{n_k-d}_{\shsh}f^{d}_{\shsh}(\nu_{jk})) \quad - \quad 2 \, J_k \, l_u(\rho) \\
 & \ge \sum_{L_u(\nu_{jk})\ge A} L_u(f^{n_k-d}_{\shsh}f^{d}_{\shsh}(\nu_{jk})) \quad - \quad l_u(\rho) \, L_u(f^{n_k}_\#(\gamma_k)) 
 \end{align*}
where the last inequality follows for sufficiently large $k$ by applying $(i)$ and the inequality $L_u(\mu_{ik}) \ge 1$ to conclude that 
$$L_u(f^{n_k}_\#(\gamma_k)) \ge 2\sum_i L_u(\mu_{ik})+2 \ge 2I_k+2 \ge 2J_k
$$
For sufficiently large $k$ we therefore have
$$(iv) \quad \frac{L_u(f^{2n_k}_\#(\gamma_k))}{L_u(f^{n_k}_\#(\gamma_k))} > 
\frac{\sum_{L_u(\nu_{jk})\ge A} L_u(f^{n_k-d}_{\shsh}f^{d}_{\shsh}(\nu_{jk}))}{L_u(f^{n_k}_\#(\gamma_k))} - l_u(\rho)
$$

We have already constrained $A$ so that $A \ge B$, and we now put one more constraint on~$A$. Applying Lemma~\ref{LemmaBddBelow}~\pref{estimate1} to $f^d$ it follows that if $L_u(\nu_{jk}) \ge B$ then $L_u(f^d_{\shsh}(\nu_{jk})) \ge b \, L_u(\nu_{jk})$, and so if $L_u(\nu_{jk}) \ge A = \max\{B,\frac{1}{b} A_{\ref{LemmaBddBelow}}\}$ it follows that $L_u(f^d_{\shsh}(\nu_{jk})) \ge A_{\ref{LemmaBddBelow}}$. This allows us to apply ``Positive Flaring'', Lemma~\ref{LemmaBddBelow}~\pref{estimate3}, with the conclusion that, letting $R = R_{\ref{LemmaBddBelow}}$,
$$(v) \quad L_u(f^{n_k-d}_{\shsh}f^{d}_{\shsh}(\nu_{jk})) \ge R \, \lambda^{(n_k - d)/2} L_u(f^d_{\shsh}(\nu_{jk})) 
$$
as long as $L_u(\nu_{jk}) \ge A$ and as long as $k$ is sufficiently large. 
Combining $(iv)$ and $(v)$, if $k$ is sufficiently large we obtain
\begin{align*}
\frac{L_u(f^{2n_k}_\#(\gamma_k))}{L_u(f^{n_k}_\#(\gamma_k))} & >
R \, \lambda^{(n_k - d)/2} \, \frac{\sum_{L_u(\nu_{jk})\ge A} L_u(f^d_{\#}(\nu_{jk})) }{L_u(f^{n_k}_\#(\gamma_k))} - l_u(\rho) \\
\intertext{and combining this with $(iii)$ we obtain}
\frac{L_u(f^{2n_k}_\#(\gamma_k))}{L_u(f^{n_k}_\#(\gamma_k))}   & > \frac{3 b R}{4} \lambda^{(n_k - d)/2} - l_u(\rho) \\ &> \nu
\end{align*}
where the second inequality holds for sufficiently large $k$. This gives us the final contradiction to $(*)$, which completes the proof of Lemma~\ref{LemmaSpecialFlaring}. 

\subsection{Appendix: The graph homotopy principle}
\label{SectionGraphHomotopy}
Lemma~\ref{LemmaGraphHomotopy} in this section was used earlier in the proof of Lemma~\ref{LemmaLuBasics}, and it will be used later in the construction of the ``homotopy semigroup action'' in Section~\ref{SectionHomotopyActionOnG}. It is an elementary result in homotopy theory; for precision we state the result in the language of category theory, and we give the complete proof.

Define the \emph{graph-point} category, a subcategory of the standard homotopy category of pairs, as follows. The objects are pairs $(G,P)$ where $G$ is a finite graph and $P \subset G$ is a finite subset. Each morphism, denoted $[f] \from (G,P) \mapsto (H,Q)$, is the homotopy class rel~$P$ of a homotopy equivalence $f \from G \to H$ that restricts to a bijection $f \from P \to Q$. Define the \emph{fundamental group} functor from the graph-point category to the category of \emph{indexed groups} as follows. To each pair $(G,P)$ we associate the indexed family of groups $\pi_1(G,P) = \bigl(\pi_1(G,p)\bigr)_{p \in P}$, and to each morphism $[f] \from (G,P) \to (H,Q)$ we associate the indexed family of group isomorphisms 
$$[f]_* \from \pi_1(G,P) \to \pi_1(H,Q) \,\, = \,\, \bigl(f_* \from \pi_1(G,p) \to \pi_1(H,f(p))\bigr)_{p \in P}
$$
The category and functor axioms implicit in this discussion are easily checked.

Let $\Aut^\gp(G,P)$ denote the group of automorphisms of $(G,P)$ in the graph-point category. Let $\Aut^\gp_0(G,P) \subgroup \Aut^\gp(G,P)$, which we call the \emph{pure automorphism group} of $(G,P)$ in the graph-point category, denote the finite index subgroup consisting of those $[f] \in \Aut^\gp(G,P)$ such that $f \from P \to P$ is the identity.




\begin{lemma}[The graph homotopy principle]
\label{LemmaGraphHomotopy}
\quad\hfill
\begin{enumerate}
\item\label{ItemGPHomInvExists}
The graph-point category is a groupoid: every morphism $[f] \from (G,P) \to (H,Q)$ has an inverse morphism $[g] \from (H,Q) \to (G,P)$, meaning that $g \circ f \from (G,P) \to (G,P)$ is homotopic to the identity rel~$P$ and $f \circ g \from (H,Q) \to (H,Q)$ is homotopic to the identity rel~$Q$.
\item\label{ItemGPFaithful}
The fundamental group functor is faithful: for any pair of morphisms \break $[f], [f'] \from (G,P) \to (H,Q)$, we have $[f]=[f']$ if and only if the restricted maps $f, f' \from P \mapsto Q$ are equal and the induced isomorphisms $f^{\vphantom\prime}_*, f'_* \from \pi_1(G,p) \to \pi_1(H,f(p))$ are equal for all $p \in P$.
\item\label{ItemGPInverse}
Two morphisms $[f] \from (G,P) \to (H,Q)$ and $[g] \from (H,Q) \to (G,P)$ are inverses if and only if their restrictions $f \from P \to Q$ and $g  \from Q \to P$ are inverses and the isomorphisms $f_* \from \pi_1(G,p) \to \pi_1(H,f(q))$ and $g_* \from \pi_1(H,f(q)) \to \pi_1(G,p)$ are inverses. 
\item\label{ItemGPAut}
The fundamental group functor restricts to an injective homomorphism defined on the pure automorphism group $\Aut^\gp_0(\pi_1(G,P)) \mapsto \oplus_{p \in P} \Aut(\pi_1(G,p))$.
\end{enumerate} 
\end{lemma}

\begin{proof} Once~\pref{ItemGPHomInvExists} and~\pref{ItemGPInverse} are proved, \pref{ItemGPFaithful} and~\pref{ItemGPAut} follow immediately.

\smallskip

To prove \pref{ItemGPInverse}, the ``only if'' direction is obvious, and for the ``if'' direction it suffices to prove this special case: for any self-morphism $[f] \from (G,P) \to (G,P)$, if $f$ fixes each $p \in P$ and induces the identity $\pi_1(G,p)$ for each $p \in P$, then $f$ is homotopic to the identity rel~$P$. For the proof, we know that $f$ is freely homotopic to the identity map~$\Id_G$, because $G$ is an Eilenberg-MacClane space and $f$ induces the identity on its fundamental group. Choose a homotopy $h \from G \times [0,1] \to G$ from $f$ to~$\Id_G$. 

We claim that for each $p \in P$ the closed path $\gamma_p(t) = h(p,t)$ $(0 \le t \le 1)$ is trivial in $\pi_1(G,p)$. Applying this claim, we alter the homotopy $h$ as follows: using the homotopy extension property, for each $p \in P$ we may homotope the map $h \from G \times [0,1] \to G$, keeping it stationary on $G \times 0$, stationary on $G \times 1$, and stationary outside of $U_p \times [0,1]$ for an arbitrarily small neighborhood $U_p$ of $p$, to arrange that $h(p \times [0,1])=p$; note that for a homotopy $X \times [0,1] \to Y$ to be ``stationary on $A \subset X$'' means that the restricted map $\{a\} \times [0,1] \to Y$ is constant for each $a \in A$. Doing this independently for each $p \in P$, we obtain a homotopy rel~$P$ from $f$ to the identity and we are done, subject to the claim.

To prove the claim, consider a closed path $\delta \from [0,1] \to G$ based at $p$, representing an arbitrary element $[\delta] \in \pi_1(G,p)$. We obtain a path homotopy $H_t \from [0,1] \to G$ from the path $H_0 = f \circ \delta$ to the concatenated path $H_1 = \gamma_p * \delta * \bar\gamma_p$ as follows: 
$$H_t(s) = \begin{cases}
\gamma_p(3s) &\quad\text{if $0 \le s \le t/3$} \\
h\biggl(\delta\bigl(\frac{3s \, - \, t}{3\, - \, 2t}\bigr),t\biggr) &\quad\text{if $t/3 \le s \le 1 - t/3$} \\
\gamma_p(3-3s) &\quad\text{if $1 - t/3 \le s \le 1$} 
\end{cases}
$$
Since for all $[\delta] \in \pi_1(G,p)$ we have $[\delta] = [f \circ \delta] = [\gamma_p] \cdot [\delta] \cdot [\gamma_p]^\inv$, it follows that $[\gamma_p]$ is in the center of $\pi_1(G,p) \approx F_n$, hence is trivial, completing the proof of~\pref{ItemGPInverse}.

\smallskip

To prove~\pref{ItemGPHomInvExists}, start with any homotopy inverse $g' \from H \to G$ of $f$. We may assume that the maps $f \from P \to Q$ and $g' \from Q \to P$ are inverses, because by the homotopy extension property we may homotope $g'$ to be stationary outside of a small neighborhood of $P$ so that for each $p \in P$ the track of the homotopy on the point $g'(f(p))$ moves it back to $p$. Since $g' \circ f \from G \to G$ fixes each point in $P$ and is homotopic to the identity, for each $p \in P$ the induced map $(g' \circ f)_* \from \pi_1(G,p) \to \pi_1(G,p)$ is an inner automorphism represented by some closed curve $\gamma_p$ based at $p$, and so for each element of $\pi_1(G,p)$ having the form $[\delta]$ for some closed curve $\delta$ based at $p$ we have $(g' \circ f)_*(\delta) = [\gamma_p * \delta * \bar\gamma_p]$. Let $h \from (G,p) \to (G,p)$ be the morphism obtained from the identity by a homotopy that is stationary outside a small neighborhood of $P$ and such that the track of the homotopy on each $p \in P$ is the closed curve $\bar\gamma_p$; again we are applying the homotopy extension property. Letting $g = h \circ g'$ we may apply \pref{ItemGPInverse} to conclude that the morphism $[f]$ is an isomorphism with inverse $[g]$.
\end{proof}

\paragraph{Remark.} Note that the proof of~\pref{ItemGPInverse} depends heavily on the fact that the center of $F_n$ is trivial. The proof breaks down, for instance, if $G$ is replaced by a torus; in fact the analogue of Lemma~\ref{LemmaGraphHomotopy}, where a graph is replaced by a torus and $P$ is a two-point subset of the torus, is false. On the other hand the analogue for any $K(\pi,1)$ space whose fundamental group has trivial center is true.

\section{Flaring in $T^*$ and hyperbolicity of $\S$.} 
\label{SectionPathFunctionsReT}

Throughout this section we continue with Notations~\ref{NotationsFlaring}~\pref{ItemCTEG}--\pref{ItemCTSplitSimple} regarding an outer automorphism $\phi \in \Out(F_n)$ and a relative train track representative $f \from G \to G$ having penultimate filtration element $G_{u-1}$ and top \eg\ stratum $H_u$. 

The main result of this section is the construction, carried out in Section~\ref{SectionHypOfS}, of the Gromov hyperbolic space $\S$ that is used in later sections for proving the multi-edge case of the Hyperbolic Action Theorem. The construction of~$\S$ is based on results found in Sections~\ref{SectionFreeSplittingT}--\ref{SectionThreeProps}, in particular Proposition~\ref{PropFlaringInTStar} which is a re-interpretation of the flaring result of Proposition~\ref{PropFlaringGeneral} expressed in the context of a certain natural free splitting. The statement of Proposition~\ref{PropFlaringInTStar} is found in Section~\ref{SectionThreeProps}, after preliminary work carried out in Sections~\ref{SectionFreeSplittingT}--\ref{SectionTStar}.

Once some of the definitions have been formulated, the reader may wish to pause to consider the ``Motivational Remarks'' found in Section~\ref{SectionMetricOnT} following Lemma~\ref{LemmaLittleLMetrics}.

\subsection{The free splitting $F_n \act T$ and its Nielsen lines.}
\label{SectionFreeSplittingT}

We begin with a description of the free splitting $F_n \act T$ associated to the marked graph $G$ and its subgraph $G_{u-1}$, together with a description of some features of $T$ associated to height~$u$ Nielsen paths in $G$.

Let $\F$ denote the free factor system corresponding to the subgraph $G_{u-1}$, having the form $\F = \{[A_1],\ldots,[A_K]\}$ where $G_{u-1}$ has noncontractible components $C_1,\ldots,C_K$ and $A_k \subgroup F_n$ is in the conjugacy class of the image of the injection $\pi_1(C_k) \inject \pi_1(G) \approx F_n$ (that injection determined up to inner automorphism of $F_n$ by appropriate choices of base points and paths between them).


\begin{definition}[The free splitting $F_n \act T$] \label{DefOfT}
Let $F_n \act T$ denote the free splitting corresponding to the subgraph $G_{u-1} \subset G$. What this means is that, starting from the deck action $F_n \act \wt G$ associated to the universal covering map $\wt G \mapsto G$, the tree $T$ is obtained from $\wt G$ by collapsing to a point each component of the total lift $\wt G_{u-1}$ of $G_{u-1}$. Let $p \from \wt G \to T$ denote the $F_n$-equivariant collapse map. Since $\wt G_{u-1}$ is $F_n$-invariant, the action $F_n \act \wt G$ induces via $p$ an action $F_n \act T$ which is evidently a free splitting, i.e.\ a minimal action on a simplicial tree with trivial edge stabilizers. 

Note that the set of conjugacy classes of nontrivial vertex stabilizers of this action is precisely the free factor system~$\F$ --- indeed the stabilizer of a vertex $v \in T$ equals the stabilizer of $p^\inv(v) \subset \wt G$ which is nontrivial if and only if $p^\inv(v)$ is a component of $\wt G_{u-1}$ covering some noncontractible component $C_k \subset G_{u-1}$, in which case the stabilizer of $v$ is conjugate to $A_k$.
\end{definition}

\begin{definition}[Lifting $f \from G \to G$ up to $\wt G$ and projecting down to $T$] 
\label{DefLiftingTTMap}
Fixing a standard isomorphism between the deck transformation group of the universal covering map $\wt G \mapsto G$ and the group $F_n \approx \pi_1(G)$, recall from covering space theory that the lifts $\wt G \to \wt G$ of $f \from G \to G$ are in bijective correspondence with the automorphisms $\Phi \in \Aut(F_n)$ that represent $\phi$, where the bijection $\Phi \leftrightarrow \ti f^\Phi$ is given by the following relation:
\begin{description}
\item[$\Phi$-twisted equivariance in $\wt G$:] \quad $\displaystyle \ti f^\Phi(\gamma \cdot x) = \Phi(\gamma) \cdot \ti f^\Phi(x) \quad\text{for all $\gamma \in F_n$, $x \in \wt G$}$
\end{description}
Since $f$ preserves $G_{u-1}$, any of its lifts $\ti f^\Phi$ preserves $\wt G_{u-1}$, and hence $\ti f^\Phi$ induces a map $f^\Phi_T \from T \to T$. The $\Phi$ twisted equivariance property in $\wt G$ implies a corresponding property in $T$:
\begin{description}
\item[$\Phi$-twisted equivariance in~$T$:] \quad $\displaystyle f^\Phi_T(\gamma \cdot x) = \Phi(\gamma) \cdot f^\Phi_T(x) \quad\text{for all $\gamma \in F_n$, $x \in T$}$
\end{description}
When the automorphism $\Phi$ is understood or is not important to the discussion, we will often drop it from the notations $\ti f^\Phi$~and~$f^\Phi_T$, writing simply $\ti f$ and $f_T$ instead.

In the next section we will impose an additional metric constraint on $f_T$; see under the heading ``Stretch properties of $f_T$''.
\end{definition}

\begin{definition}[Nielsen paths, the Nielsen set, and $\rho^*$ paths in $T$]
\label{DefNielsenSet}
In the geometric and parageometric cases, where $\rho,\bar\rho$ exist and are the unique inverse pair of Nielsen paths of height $u$, 
a \emph{Nielsen path} in $\wt G$ is any lift of $\rho$ or $\bar\rho$, and a \emph{Nielsen path} in $T$ is any projection to $T$ of any Nielsen path in $\wt G$. In the geometric case, where $\rho$ is closed and has distinct initial and terminal directions, a \emph{Nielsen line} in $\wt G$ is a line which projects to a bi-infinite iterate of $\rho$, and a \emph{Nielsen line} in $T$ is the projection of a Nielsen line in $\wt G$. 

The \emph{Nielsen set} $\cN$, a collection of subsets of $T$, is defined as follows. In the ageometric case, $\cN = \emptyset$; in the geometric case, $\cN$ is the set of Nielsen lines in~$T$; and in the parageometric case, $\cN$ is the set of Nielsen paths in~$T$. Furthermore, for each $N \in \cH$ its \emph{basepoint set} or \emph{base lattice}, denoted $Z(N)$, is defined as follows. In the geometric case, the Nielsen line $N$ has a unique decomposition as a bi-infinite concatenation of Nielsen paths, and $Z(N)$ is defined to be the set of concatenation points. In the parageometric case, where $N$ is just a single Nielsen path, $Z(N)$ is its endpoint pair.

Note that in the geometric case, basepoint sets of distinct Nielsen lines are disjoint --- for all $N \ne N' \in \cN$ we have $Z(N) \intersect Z(N') = \emptyset$. This follows from two facts about the base point $p$ of~$\rho$. First, each point $\ti p \in \wt G$ lying over $p$ is an endpoint of exactly two Nielsen paths in $\wt G$, both contained in the same Nielsen line in $\wt G$. Second, $p$ is not contained in~$G_{u-1}$ \cite[Corollary 4.19]{\recognitionTag}, so no lift $\ti p$ is contained in $\wt G_{u-1}$, and the projection map $\wt G \mapsto T$ is locally injective on the complement of $\wt G_{u-1}$.

Consider a finite path $\sigma_{\wt G}$ in $\wt G$ with endpoints at vertices, with projections $\sigma_G$ in $G$ and $\sigma_T$ in $T$. If $\sigma_G = \rho^i$ or $\bar\rho^i$ for some integer $i \ne 0$ --- in the parageometric case where $\rho$ is not closed, $i$ must equal~$1$ --- then we say that $\sigma_T$ is a \emph{$\rho^*$-path} in $T$, the superscript ${}^*$ representing the exponent. Note that $\rho^*$ paths in $T$ are precisely the paths of the form $\overline{QQ'}$ for which there exists $N \in \cN$ such that $Q,Q' \in Z(N)$. 
\end{definition}

\subsection{Path functions on $\wt G$ and $T$.} 
\label{SectionMetricOnT}

In a tree, a finite path with initial and terminal endpoints $V,W$ is determined by those endpoints and is denoted $\overline{VW}$.

Each of $l_u$, $l_\PF$, $L_u$, $L_\PF$ is a path function on $G$ that vanishes on paths in $G_{u-1}$ (see Section~\ref{SectionFlaringBasics}). Each lifts via the universal covering map $\wt G \to G$ to an $F_n$-invariant path function on $\wt G$ that vanishes on paths in $\wt G_{u-1}$, and hence projects via $q \from \wt G \to T$ to a well-defined and $F_n$-invariant path function on~$T$. We re-use the notations $l_u$, $l_\PF$, $L_u$, $L_\PF$ for these path functions on $\wt G$ and on $T$; the context should help to avoid ambiguities. For any path $\beta_{\wt G}$ in $\wt G$ with endpoints at vertices, letting $\beta_G$ be its projection to $G$ and $\beta_T$ its projection to $T$, it follows from the definitions that $l(\beta_G)=l(\beta_{\wt G})=l(\beta_T)$ for any of $l=l_u$, $l_\PF$, $L_u$, or $L_\PF$. 

\bigskip
\noindent
\textbf{Remark.} The point of ``well-definedness'' of, say, $L_\PF$ on $T$ is that for any vertices $V,W \in T$, if $V_1,V_2 \in \wt G$ both map to $V$ and if $W_1,W_2$ both map to $W$ then each of the paths $\overline{V_1V_2}$ and $\overline{W_1W_2}$ is either degenerate or is contained in $\wt G_{u-1}$, and hence $L_\PF(\overline{V_1W_1}) = L_\PF(\overline{V_2W_2}) = L_\PF(\overline{VW})$.

\smallskip

\begin{definition}\label{DefinitionMetricsOnT}
Associated to the path functions $l_u(\cdot)$, $l_\PF(\cdot)$, $L_u(\cdot)$, 
$L_\PF(\cdot)$ on $\wt G$ and on $T$, we have, respectively, $F_n$-equivariant functions $d_u(\cdot,\cdot)$, $d_\PF(\cdot,\cdot)$, $D_u(\cdot,\cdot)$, $D_\PF(\cdot,\cdot)$ on pairs of vertices in $\wt G$ and in $T$. For example
$$D_\PF(V,W) = L_\PF(\overline{VW})
$$
In particular, tracing back through the definitions one sees that for vertices $V,W \in T$ their distance $d_u(V,W)$ simply counts the number of edges of $T$ in the path $\overline{VW}$.
\end{definition}

For each of $l=l_u$, $l_\PF$, $L_u$, $L_\PF$ the quantity $l_\PF(E)=L_\PF(E)$ is bounded away from zero as $E \subset T$ varies over edges; this is an immediate consequence of the fact that as $e$ varies over the finite set of edges of $H_u$, the finite set of positive numbers $l(e)$ has a positive minimum. We record this fact as:
\begin{lemma}
\label{LemmaPFLowerBound}
There exists $\eta = \eta_{\ref{LemmaPFLowerBound}} > 0$ such that for each edge $E \subset T$ with endpoints $V \ne W \in T$, the values of $d_u$, $d_\PF$, $D_u$, $D_\PF$ on $V,W$ are all $\ge \eta$. \qed
\end{lemma}

Each of the path functions $l_u$ and $l_\PF$ is additive, meaning that its value on an edge path is the sum of its values on individual edges. It follows that each of $d_u$ and $d_\PF$ is a path metric on $T$. Furthermore, $d_u$ and $d_\PF$ are quasicomparable to each other, because $H_u$ has only finitely many edges hence $T$ has only finitely many edge orbits under the action of~$F_n$, and the values of $d_u$ and $d_\PF$ on the endpoint pair of each edge is positive (Lemma~\ref{LemmaPFLowerBound}). The ``metrics'' $D_u$ and $D_\PF$ are also quasicomparable to each other, by application of Corollary~\ref{CorollaryLFormulas}. However, $D_u$ and $D_\PF$ are not actual metrics because they may violate the triangle inequality. Nonetheless $D_u$ and $D_\PF$ do satisfy a coarse version of the triangle inequality, as a consequence of~Lemma~\ref{LemmaCTI}, and we will refer to this by saying that $D_u$ and $D_\PF$ are \emph{coarse metrics} on the vertex set of ~$T$. We record these observations as:

\begin{lemma}\label{LemmaLittleLMetrics}
$d_u$ and $d_\PF$ are quasicomparable, $F_n$-equivariant metrics on vertices of~$T$. Also, $D_u$ and $D_\PF$ are quasicomparable, $F_n$-equivariant coarse metrics on vertices in~$T$.
\end{lemma}

\subparagraph{Motivational remarks.} The metrics $d_u$ and $d_\PF$ may fail to satisfy the desired flaring condition: if $H_u$ is geometric then for any Nielsen line $N \subset T$, iteration of $f_T=f_T^\Phi$ produces a sequence of Nielsen lines $N_i = (f_T^k)_\#(N)$ ($k \in \Z$), and furthermore the map $f^k_T$ takes the base lattice $Z(N)$ to the base lattice $Z(N_k)$ preserving path distance. Since the base lattice of a Nielsen line has infinite diameter in any invariant path metric on $T$, this demonstrates the failure of flaring. In hopes of averting a failure of flaring, we might instead consider using the coarse metrics $D_u$ and $D_\PF$: when two vertices $V,W$ are contained in the base lattice of the same Nielsen line we have the equations $D_u(V,W)=D_\PF(V,W)=0$ which are precisely designed to correct the failure of flaring. Although using $D_u$ or $D_\PF$ creates its own problem because they are not actually metrics, in Section~\ref{SectionTStar} we shall solve that problem by using the coning construction often employed in studies of relative hyperbolicity, coning off those paths in $T$ which exhibit nonflaring behavior to obtain a graph $T^*$. Furthermore, this graph will come equipped with an actual path metric $d^*$ such that the inclusion $T \inject T^*$ is a quasi-isometry from $D_\PF$ to $d^*$; see Proposition~\ref{PropConeQI} in Section~\ref{SectionThreeProps}.

\bigskip

Next we translate several results on path functions in Section~\ref{SectionFlaringBasics} into the context of the coarse metric $D_\PF$ on $T$:


\begin{proposition}
\label{PropDecompInT}
For any vertices $V \ne W \in T$ there is a unique decomposition of the path $\overline{VW} = \mu_0 \nu_1 \mu_1 \cdots \nu_A \mu_A$ such that the following properties hold:
\begin{enumerate}
\item\label{ItemAIsZero}
If $\rho$ does not exist then $A=0$.
\item\label{ItemRhoStarSubpaths}
If $\rho$ exists then the $\nu_a$'s are precisely all of the maximal $\rho^*$ paths of $\overline{VW}$, and so each $\mu_a$ contains no $\rho^*$ subpath.
\item\label{ItemOneNotDeg}
If $\rho$ exists and if $1 \le a < a+1 \le A-1$ then at least one of the subpaths $\mu_a,\mu_{a+1}$ is nondegenerate; in the geometric case, all $\mu_a$-subpaths are nondegenerate. 
\item\label{ItemDPFFormula}
$D_\PF(V,W) = l_\PF(\mu_0) + \cdots + l_\PF(\mu_A)$
\qed
\end{enumerate}
Furthermore, given any path $\gamma$ in $T$ --- finite, singly infinite, or bi-infinite --- whose endpoints, if any, are at vertices, and assuming that $\rho$ exists, there is a unique decomposition of $\gamma$ as an alternating concatenation of its maximal $\rho^*$ paths (the $\nu$-subpaths) and paths that contain no $\rho^*$ subpath (the $\mu$-subpaths) such that for any two consecutive $\mu$-subpaths at least one is nondegenerate, all $\mu$-subpaths being nondegenerate in the geometric case.
\end{proposition}

\begin{proof} Items~\pref{ItemAIsZero}--\pref{ItemOneNotDeg} are translations of Corollary~\ref{CorollaryAltDecomp}, and item~\pref{ItemDPFFormula} is a translation of Corollary~\ref{CorollaryLFormulas}; the proofs are immediate from those results combined with the definitions. 

The ``Furthermore\ldots'' clause is a quick consequence of the following observations that hold for any nested pair of finite subpaths $\overline{VW} \subset \overline{V'W'}$ in $T$. First, every $\rho^*$ subpath of $\overline{VW}$ is a $\rho^*$ subpath of $\overline{V'W'}$. Also, every maximal $\rho^*$ subpath of $\overline{VW}$ whose $d_\PF$ distances from $V$ and from $W$ are greater than $l_\PF(\rho)$ is a maximal $\rho^*$ subpath of~$\overline{V'W'}$.
\end{proof}

\subparagraph{Remark on item~\pref{ItemOneNotDeg}.} Consider the paths $\mu_a$ for $1 \le a \le A-1$. In the geometric case each such $\mu_a$ is nondegenerate, and this can be used to improve the constants in some applications of~\pref{ItemOneNotDeg} (underlying nondegeneracy of $\mu_a$ is the fact that the base point of the closed Nielsen path $\rho$ is disjoint from $G_{u-1}$). In the parageometric case, on the other hand, one of the paths $\mu_a$, $\mu_{a+1}$ may be degenerate. This happens for $\mu_a$ only if, up to orientation reversal, the Nielsen path $\rho$ has initial vertex $p \in G_{u-1}$ (the terminal vertex is necessarily disjoint from~$G_{u-1}$, the fact which underlies item~\pref{ItemOneNotDeg}), in which case $\mu_a$ is degenerate if and only if $\nu_a \mu_a \nu_{a-1}$ lifts to a path in $\wt G$ that projects to a path in $G$ of the form $\bar\rho \mu \rho$ where $\mu$ is a nondegenerate closed path in $G_{u-1}$ based at~$p$.

\medskip




\subsection{Constructing $T^*$ by coning off Nielsen axes of $T$.} 
\label{SectionTStar}
Embed the tree $T$ into a graph denoted $T^*$, and extend the action $F_n \act T$ to an action $F_n \act T^*$, as follows. Index the Nielsen set as $\cN = \{N_j\}_{j \in J}$, letting $Z_j$ be the basepoint set of $N_j$ (Definition~\ref{DefOfT}). For each $j \in J$ we cone off $Z_j$ by adding a new vertex $P_j=P(N_j)$ and attaching a unique edge $\overline{P_j Q}$ for each $Q \in Z_j$. The points $P_j$ are called \emph{cone points} and the edges $\overline{P_j Q}$ are called \emph{cone edges}. Since the simplicial action $F_n \act T$ takes Nielsen paths to Nielsen paths and hence induces a basepoint preserving permutation of the Nielsen set $\cN$, this action extends uniquely to $F_n \act T^*$ permuting the cone points and the cone edges. 

Let $V^\infty \subset T^*$ be the set of vertices $v \in T^*$ whose stabilizer subgroup $\Stab(v)$ is infinite. 

\begin{lemma}
\label{LemmaVertexStabilizers}
The formula $v \mapsto \Stab(v)$ defines an injection from $V^\infty$ to the set of nontrivial subgroups of $F_n$. A subgroup $S \subgroup F_n$ is equal to $\Stab(v)$ for some $v \in V^\infty$ if and only if $S$ is conjugate to the fundamental group of some noncontractible component of $G_{u-1}$ in $\pi_1(G) \approx F_n$, or $H_u$ is a geometric stratum and $S$ is conjugate to the infinite cyclic subgroup $\<\rho\> \subgroup F_n$.
\end{lemma}

\begin{proof} This fact may be extracted from Theorem F of \SubgroupsThree. But a direct proof is easy; here is a sketch.

We have a partition 
$$
V^\infty = V^\infty_{u-1} \coprod V^\infty_\cN \quad\text{where}\quad V^\infty_{u-1} = V^\infty \intersect T \quad\text{and}\quad V^\infty_\cN = \{\text{cone points}\} = \{P_j\}
$$
The collapse map $\wt G \mapsto T$ induces an equivariant and hence stabilizer preserving bijection between $V^\infty_{u-1}$ and the set of components of $\wt G_{u-1}$ having nontrivial stabilizer. Using covering space theory, the stabilizers of the latter components are precisely those subgroups of $F_n$ conjugate to the fundamental group of some noncontractible component of $G_{u-1}$. Furthermore, since distinct components of $\wt G_{u-1}$ are disjoint subtrees of $\wt G$, the intersections of their stabilizers are trivial, and so the stabilizers are unequal if they are nontrivial. This completes the proof of $H_u$ is nongeometric.

If $H_u$ is geometric then we have equivariant and hence stabilizer preserving bijections $\wt N_j \leftrightarrow N_j \leftrightarrow P_j$ where $\wt N_j$ is the Nielsen line in $\wt G$ mapping to $N_j$ under the collapse map $\wt G \mapsto T$. By definition the $\wt N_j$ are precisely those lines in $\wt G$ that cover the closed Nielsen path $\rho$. The element of $\pi_1(G)$ represented by $\rho$ (and denoted by $\rho$) is root free in $\pi_1(G)$ because $\rho$ has a unique $u$-illegal turn (Notations~\ref{NotationsFlaring}~\pref{ItemCTiNP}), and so by covering space theory the stabilizers of the lines $\wt N_j$ are precisely the infinite cyclic subgroups in the conjugacy class of the group $\<\rho\> \subgroup \pi_1(G) \approx F_n$. And as before, two different such lines have distinct stabilizers. 

The proof is completed by noting that $\<\rho\>$ is not conjugate in $F_n$ to the fundamental group of a noncontractible component of $G_{u-1}$, because $\rho$ is a circuit not contained in $G_{u-1}$ (Notations~\ref{NotationsFlaring}~\pref{ItemCTiNP}).
\end{proof}

\begin{definition}[Piecewise Riemannian metric $ds^*$, and path metric $d^*$, on~$T^*$] 
\label{DefLittleDStar}
\quad\\ We may construct an $F_n$-equivariant piecewise Riemannian metric on the tree $T$, denoted $ds$, such that for vertices $V,W \in T$ we have $d_\PF(V,W)=\int_{\overline{VW}} ds$. We may extend $ds$ to an $F_n$-equivariant piecewise Riemannian metric denoted $ds^*$ on $T^*$ as follows. In the nongeometric case there is nothing to do. In the geometric case the group $F_n$ acts freely and transitively on the set of cone edges $\{\overline{P_j Q} \suchthat j \in J, \,\, Q \in Z(N_j)\}$; extend $ds$ over a single cone edge $\overline{P_j Q}$, then extend it over all other cone edges equivariantly to obtain $ds^*$; note that the length $\int_{\overline{P_jQ}} ds^*$ is independent of $j$ and $Q$. In the parageometric case, the group $F_n$ acts freely on the set of cone edges, and there are two orbits of cone edges corresponding to the two endpoints of $\rho$; we extend $ds$ over a single cone edge in each of the two orbits, then extend equivariantly to obtain $ds^*$; also, we require that length $\int_{\overline{P_j Q}} ds^*$ be the same on both orbits of cone edges. Define the \emph{cone height} to be the length of any cone edge in the metric $ds^*$.

Next let $d^*(\cdot,\cdot)$ be the path metric on $T^*$ obtained by minimizing path lengths: $d^*(x,y)$ equals the minimum of $\int_\gamma ds^*$ over all continuous paths $\gamma$ in $T^*$ having endpoints $x,y$. The infimum is evidently minimized by some embedded edge path in $T^*$ having endpoints $x,y$. Note that since $T^*$ is not a tree, embedded edge paths need not be determined by their endpoints. 
\end{definition}

\smallskip\noindent\textbf{Bypasses in $T^*$.} For any $\rho^*$ path $\overline{QQ'}$ we let $\widehat{QQ'}$ denote the path $\overline{QP(N)} * \overline{P(N)Q'}$ in $T^*$, called the \emph{bypass} of $\overline{QQ'}$. Note that a path in the graph $T^*$ is a bypass if and only if it is a two-edge path having a cone point as its midpoint, and furthermore a bypass is completely determined by its endpoints. We thus have a one-to-one correspondence $\nu \leftrightarrow \wh\nu$ between the set of $\rho^*$ paths and the set of bypasses. 

\smallskip\noindent\textbf{Extending the map $f^\Phi_T \from T \to T$ to $f^\Phi_{T^*} \from T^* \to T^*$.} Following Definition~\ref{DefLiftingTTMap}, and choosing $\Phi \in \Aut(F_n)$ representing $\phi$, lift $f \from G \to G$ to $\wt G$ and project to $T$ obtaining $\Phi$-twisted equivariant maps
$$\ti f = \ti f^\Phi \from \wt G \to \wt G \qquad f_T = f^\Phi_T \from T \to T
$$
With respect to the inclusion $T \inject T^*$ we extend $f^\Phi_T$ to a $\Phi$-twisted equivariant map 
$$\fTstar = \fTstar^\Phi \from T^* \to T^*
$$
as follows. As noted, for convenience we will often suppress $\Phi$ from the notation for these maps.

The action of $f_T$ on $T$ induces a well-defined action on the Nielsen set $\cN$, and so we can extend $f_T$ over the set of cone points by setting $\fTstar(P(N)) = P(f_T(N))$ for each $N \in \cN$. Furthermore, for each $N \in \cN$ the map $f_T$ restricts to a bijection of basepoint sets $f_T \from Z(N) \to Z(f_T(N))$, and so for each $Q \in Z(N)$ we can extend the endpoint map $(P(N),Q) \mapsto (P(f_T(N)), f_T(Q))$ uniquely to an isometry between cone edges $\fTstar \from \overline{P(N),Q}  \to \overline{P(f_T(N)),f_T(Q)}$. 

For each edge $E \subset T$ we have the following equation that follows from the ``big $L$'' eigenlength equation in Section~\ref{SectionFlaringBasics}:
\begin{description}
\item[Eigenlength equation in $T$:] \quad $\displaystyle \int_{f_T(E)} ds = \lambda \int_E ds$
\end{description}
It follows that by equivariantly homotoping $f_T$ relative to endpoints of edges we may arrange that $f_T$ stretches each edge of $T$ by a uniform factor~$\lambda$. Since its extension $f_{T*}$ is an isometry on edges of $T^* \setminus T$, it follows that $f_{T*}$ is $\lambda$-Lipschitz. These conditions constitute additional constraints on the maps $f_T$ and $f_{T^*}$ which we record here:

\medskip
\noindent
\textbf{Stretch properties of $f_T$ and $f_{T^*}$:}
\begin{itemize}
\item The maps $f_T,f_{T^*}$ stretch each edge $E \subset T$ by a constant factor $\lambda$ over the path $f_T(E)=f_{T^*}(E)$.
\item The map $T^*$ permutes cone edges, taking each cone edge isometrically to its image.
\end{itemize}
These stretch properties are the first step of Lemma~\ref{LemmaTStarQI} to follow.

\medskip

We recall here some basic definitions. Given constants $k \ge 1$, $c \ge 0$, a map of metric spaces $f \from X \to Y$ is a \emph{$(k,c)$-quasi-isometric embedding} if for all $p,q \in X$ we have
$$\frac{1}{k} \, d(p,q)-c \le d(f(p),f(q)) \le k \, d(p,q) + c
$$
If in addition each point of $Y$ has distance $\le c$ from some point of $f(X)$ then $f$ is a \emph{$(k,c)$-quasi-isometry}. If the domain $X$ is a subinterval of $\reals$ then we say that $f$ is a \emph{$(k,c)$-quasigeodesic}. Sometimes we conflate the constants $k,c$ by setting $k=c$ and using terminology like ``$k$-quasi-isometries'' etc. Sometimes we ignore $k,c$ altogether and use terminology like ``quasi-isometries'' etc. 

\begin{lemma}
\label{LemmaTStarQI}
The map $f_{T^*} \from T^* \to T^*$ is a quasi-isometry.
\end{lemma}

\begin{proof} Let $\Phi \in \Aut(F_n)$ be the representative of $\phi$ corresponding to $f_{T^*}$, and so $f_{T^*}$ satisfies the $\Phi$-twisted equivariance equation (see Definition~\ref{DefLiftingTTMap}). 

The map $f_{T^*}$ is Lipschitz, by the ``Stretch properties'' noted above. To complete the proof it suffices to show that there is a Lipschitz map $\bar f_{T^*} \from T^* \to T^*$ such that $f_{T^*}$ and $\bar f_{T^*}$ are \emph{coarse inverses}, meaning that each of the composed maps $\bar f_{T^*} \circ f_{T^*}$ and $f_{T^*} \circ \bar f_{T^*}$ moves each point of $T^*$ a uniformly bounded distance. We construct $\bar f_{T^*}$ by taking advantage of twisted equivariance of $f_{T^*}$ combined with the fact that the action $F_n \act T^*$ has finitely many vertex and edge orbits (this is a kind of ``twisted equivariant'' version of the Milnor-Svarc lemma).

Consider the vertex set $V^*$ of $T^*$ and its partition $V^* = V^0 \coprod V^\infty$ into points whose $F_n$-stabilizers are trivial and infinite, respectively. Using $\Phi$-twisted equivariance it follows that for any vertex $v \in V$ we have a subgroup inclusion $\Phi(\Stab(v)) \subset \Stab(f_{T^*}(v))$. It follows that $f_{T^*}(V^\infty) \subset V^\infty$. Furthermore, that subgroup inclusion is an equation $\Phi(\Stab(v)) = \Stab(f_{T^*}(v))$ --- this is a consequence of Lemma~\ref{LemmaVertexStabilizers} combined with the fact that $f \from G \to G$ restricts to a homotopy equivalence of the union of noncontractible components of $G_{u-1}$ and with the fact that $f_\#(\rho)=\rho$. It follows that the restricted map $f_{T^*} \from V^\infty \to V^\infty$ is a bijection of the set $V^\infty$. 

Define the restriction $\bar f_{T^*} \restrict V^\infty$ to equal the inverse of the restriction $f_{T^*} \restrict V^\infty$; this map $\bar f_{T^*} \restrict V^\infty$ is automatically $\Phi^\inv$-twisted equivariant. Define the restriction $\bar f_{T^*} \restrict V^0$ as follows: choose one representative $v \in V^0$ of each orbit of the action $F_n \act V^0$, choose $\bar f_{T^*}(v) \in V^0$ arbitrarily, and extend over all of $V^0$ by $\Phi^\inv$-twisted equivariance. Having defined a $\Phi^\inv$-twisted equivariant map $\bar f_{T^*} \from V^* \to V^*$, extend $\bar f_{T^*}$ over each edge to stretch distance by a constant factor, and hence we obtain a $\Phi^\inv$-twisted equivariant map $\bar f_{T^*} \from T^* \to T^*$. Since there are only finitely many orbits of edges, $\bar f_{T^*}$ is Lipschitz. Since $f_{T^*}$ is $\Phi$-twisted equivariant and $\bar f_{T^*}$ is $\Phi^\inv$-twisted equivariant, it follows that the two compositions $f_{T^*} \composed \bar f_{T^*}$ and $\bar f_{T^*} \composed f_{T^*}$ are equivariant in the ordinary untwisted sense. Each of these compositions therefore moves each point a uniformly bounded distance, hence $f_{T^*}$ and $\bar f_{T^*}$ are coarse inverses.
\end{proof}

\subsection{Geometry and dynamics on $T^*$.} 
\label{SectionThreeProps}
In this section we prove several propositions regarding $T^*$, including Proposition~\ref{PropFlaringInTStar} which is our interpretation of flaring in~$T^*$. The proofs will follow after stating all of the propositions.

\begin{proposition}[Quasicomparibility of $D_\PF$ and $d^*$]
\label{PropConeQI}
The inclusion of the vertex set of $T$ into the vertex set of $T^*$ is a quasi-isometry from the coarse metric $D_\PF$ to the metric $d^*$: there exist constants $K = K_{\ref{PropConeQI}} \ge 1$, $C=C_{\ref{PropConeQI}} \ge 0$ such that for all vertices $V,W \in T$ we have
$$\frac{1}{K} \, d^*(V,W) - C \le D_\PF(V,W) \le K \, d^*(V,W) + C
$$
\end{proposition}

\bigskip

Given an end $\xi \in \bdy T$, for any vertices $V,W \in T$ the intersection of the two rays $\overline{V\xi}$, $\overline{W\xi}$ is a subray of each. It follows that as we hold $\xi$ fixed and let $V$ vary, one of two alternatives holds: each ray $\overline{V\xi}$ has infinite $d^*$ diameter in which case we say that $\xi$ is \emph{infinitely far in $T^*$}; or each ray $\overline{V\xi}$ has finite $d^*$ diameter in which case we say that $\xi$ is \emph{finitely far in $T^*$}.

Recall (e.g.\ \cite{KapovichRafi:HypImpliesHyp}) that a path $f \from I \to T^*$ defined on an interval $I \subset \reals$ is a \emph{reparameterized quasigeodesic} if there exists a monotonic, surjective function $\mu \from J \to I$ defined on an interval $J \subset \reals$ such that the composition $f \circ \mu \from J \to T^*$ is a quasigeodesic. 

\begin{proposition}[Hyperbolicity and quasigeodesics of $T^*$]
\label{PropTStarHyp}
The graph $T^*$ with the metric $d^*$ is Gromov hyperbolic. Each geodesic segment, ray, or line in $T$ is a reparameterized quasigeodesic in $T^*$ with respect to $d^*$, with uniform quasigeodesic constants. Furthermore, there is an injection that assigns to each $\xi \in \bdy T$ which is infinitely far in $T^*$ a point $\xi^* \in \bdy T^*$, so that for any two points $\xi \ne \eta \in \bdy T$ both of which are infinitely far in $T^*$, the line $\overline{\xi\eta} \subset T$ is the unique line in $T$ which is a reparameterized quasigeodesic line in $T^*$ with ideal endpoints $\xi^*,\eta^*$.
%
%
\end{proposition}


\bigskip

We now state our re-interpretation of the earlier flaring result Proposition~\ref{PropFlaringGeneral} in the setting of $T^*$. 
Given $\eta>0$, and given a sequence of vertices $V_r \in T^*$ defined for $r$ in some interval of integers $a \le r \le b$, we say that this sequence is an \emph{$\eta$-pseudo-orbit} of $\fTstar$ if $d^*(\fTstar(V_r),V_{r+1}) \le \eta$ for all $a \le r < r+1 \le b$.

\begin{proposition}[Flaring in $T^*$]
\label{PropFlaringInTStar}
For each $\mu > 1$, $\eta \ge 0$ there exist integers $R \ge 1$ and $A \ge 0$ such that for any pair of $\eta$-pseudo-orbits $V_r$ and $W_r$ of $f_{T^*}$ defined on the integer interval $-R \le r \le R$, if $d^*(V_0,W_0) \ge A$ then
$$\mu \cdot d^*(V_0,W_0) \le \max \{d^*(V_{-R},W_{-R}),d^*(V_R,W_R)\}
$$
\end{proposition}

\paragraph{Proof of Proposition \ref{PropConeQI}: Quasicomparibility in $T^*$.} 
\medskip\noindent
In the ageometric case where $\rho$ does not exist then $T=T^*$ and $D_\PF = d^*$ and we are done. Henceforth we assume that $\rho$ exists. 

Letting $h$ denote the cone height, it follows that each bypass has length $2h$. 

Given vertices $V,W \in T$, using Proposition~\ref{PropDecompInT} we obtain a decomposition of the path $\overline{VW}$ and an accompanying formula for $D_\PF(V,W)$:
\begin{align*}
\overline{VW} &= \mu_0 \, \nu_1 \, \mu_1 \, \ldots \, \nu_{A-1} \, \mu_{A-1} \, \nu_A \, \mu_A \\
D_\PF(V,W) &= l_\PF(\mu_0) + \cdots + l_\PF(\mu_A)
\end{align*}
Each $\nu_i$ is a subpath of some element $N \in \cN$ of the Nielsen set and the endpoints of $\nu_i$ are distinct points in $Z(N)$, and so the corresponding bypass $\wh\nu_i$ is defined. We thus obtain a path in $T^*$ and a length calculation as follows:
\begin{align*}
\widehat{VW} &=  \mu_0 \, \wh\nu_1 \, \mu_1 \, \ldots \, \wh\nu_{A-1} \, \mu_{A-1} \, \wh\nu_A \, \mu_A \\
\Length(\widehat{VW}) &= l_\PF(\mu_0) + \cdots + l_\PF(\mu_A) + 2  h  A
\end{align*}
Applying Proposition~\ref{LemmaPFLowerBound} with its constant $\eta = \eta_{\ref{LemmaPFLowerBound}}$, and applying Proposition~\ref{PropDecompInT}~\pref{ItemOneNotDeg}, it follows that if $0 \le a < a+1  \le A$ then at least one of $l_\PF(\mu_a)$, $l_\PF(\mu_{a+1})$ is~$\ge \eta$, and hence
$$l_\PF(\mu_0) + \cdots + l_\PF(\mu_A) \ge  (A-1) \, \eta / 2
$$
We therefore have
\begin{align*}
d^*(V,W) &\le \Length(\widehat{VW}) \\
  &\le l_\PF(\mu_0) + \cdots + l_\PF(\mu_A) + 2h\left( \frac{2 (l_\PF(\mu_0) + \cdots + l_\PF(\mu_A))}{\eta}  + 1 \right) \\
  &= \bigl(1 + 4h/\eta\bigr) D_\PF(V,W)  + 2h
\end{align*}
This proves the first inequality using any $K \ge 1 + 4h/\eta$ and $C = 2h/K$.
%

\medskip\noindent

We turn to the opposite inequality. Before starting, we shall normalize the choice of cone height to be $h = \frac{1}{2} l_\PF(\rho)$ and so each bypass $\wh\nu$ has length $l_\PF(\rho)$. Proving Proposition~\ref{PropConeQI} with $h$ normalized in this fashion implies the proposition for any value of $h$, because the normalized version of $d^*$ is bi-Lipschitz equivalent to any other version.

Given vertices $V,W \in T$, choose an embedded edge path $\gamma^* \subset T^*$ with endpoints $V,W$ satisfying two optimization conditions:
\begin{description}
\item[(i)] $\int_\gamma ds^*$ is minimal, and so $d^*(V,W) = \int_\gamma ds^*$.
\item[(ii)] Subject to $(i)$, $\int_{\gamma \intersect T} ds^* = \int_{\gamma \intersect T} ds$ is minimal.
\end{description}
There is a unique decomposition of $\gamma$ as an alternating concatenation of subpaths in $T$ and bypasses:
$$\gamma^* = \mu_0 \, \wh\nu_1 \, \mu_1 \, \ldots \, \mu_{A-1} \wh\nu_A \, \mu_A
$$
from which we obtain the formula
$$d^*(V,W) = l_\PF(\mu_0) + \cdots + l_\PF(\mu_A) + A \, l_\PF(\rho)
$$
We claim that each $\mu_a$ has no $\rho^*$-subpath. Otherwise we can decompose $\mu_a = \mu' \nu' \mu''$ where $\nu'$ is a Nielsen path in $T$. In the case that each of $\mu',\mu''$ is nondegenerate, or that $a=0$ and $\mu''$ is nondegenerate, or that $a=A$ and $\mu'$ is nondegenerate, construct a path $\gamma'$ from $\gamma$ by replacing $\nu'$ with the corresponding bypass; from the choice of normalization we have $\int_{\gamma'} ds^* = \int_\gamma ds^*$ but $\int_{\gamma' \intersect T} ds < \int_{\gamma \intersect T} ds$, a contradiction. The other cases all lead to a path $\gamma'$ exhibiting the same contradiction, and are described as follows. In the case that $a \ge 1$ and $\mu'$ is degenerate, replace the subpath $\wh\nu_{a} \nu'$ of $\gamma$ with the unique bypass having the same endpoints. And in the case that $a \le A-1$ and $\mu''$ is degenerate, replace the subpath $\nu' \wh\nu_{a+1}$ with the unique bypass having the same endpoints. And in the last remaining case, where $a=A=0$ and $\mu',\mu''$ are both degenerate, we have $\nu'=\gamma^*$ and we let $\gamma' = \wh\nu'$ be the corresponding bypass. 

Consider the concatenated edge path in $T$ denoted
$$\mu_0 \, \nu_1 \, \mu_1 \, \ldots \, \mu_{A-1} \, \nu_A \, \mu_A
$$
which is obtained from $\gamma^*$ by replacing each bypass $\wh\nu_a$ with the corresponding $\rho^*$-subpath $\nu_a \subset T$. Straightening this concatenation in $T$ produces the path $\overline{VW}$ to which we may inductively apply the coarse triangle inequality for $L_\PF$ given in Lemma~\ref{LemmaCTI}, with the conclusion that
\begin{align*}
D_\PF(V,W) &= L_\PF(\overline{VW}) \\
 &\le L_\PF(\mu_0) + L_\PF(\nu_1) + L_\PF(\mu_1) + \cdots \\ & \qquad\qquad +L_\PF(\mu_{A-1}) + L_\PF(\nu_A) + L_\PF(\mu_A) + (2A-1) C_{\ref{LemmaCTI}} \\
\intertext{Since $\mu_a$ has no $\rho^{\pm i}$ subpath we have $L_\PF(\mu_a)=l_\PF(\mu_a)$, and since $L_\PF(\rho^{\pm i}) = 0$ we have $L_\PF(\nu_a)=0$, and therefore}
D_\PF(V,W) &\le l_\PF(\mu_0) + \cdots + l_\PF(\mu_A) + 2AC_{\ref{LemmaCTI}} \\
 &\le l_\PF(\mu_0) + \cdots + l_\PF(\mu_A) + \frac{2C_{\ref{LemmaCTI}}}{l_\PF(\rho)} \cdot A l_\PF(\rho)  \\
 & \le K d^*(V,W)
\end{align*}
for any $K \ge \max\{1,2 C_{\ref{LemmaCTI}} / l_\PF(\rho)\}$. This completes the proof of Proposition~\ref{PropConeQI}. \qed

\paragraph{Proof of Proposition \ref{PropTStarHyp}: Hyperbolicity of $T^*$.} If $\rho$ does not exist, i.e.\ in the ageometric case, since $T^*=T$ is a tree we are done. The parageometric case is similarly easy, but it is also subsumed by the general proof when $\rho$ does exist for which purpose we will apply a result of Kapovich and Rafi \cite[Proposition~2.5]{KapovichRafi:HypImpliesHyp}. 

Let $T^{**}$ be the graph obtained from $T$ by attaching an edge $\overline{Q\,Q'}$ for each unordered pair $Q \ne Q' \in Z(N)$ for each element $N \in \cN$ of the Nielsen set. We put the simplicial metric on $T^{**}$, assigning length~$1$ to each edge. We have a map $T^{**} \to T^*$ extending the inclusion $T \inject T^*$, defined to take each attached edge $\overline{Q\,Q'} \subset T^{**}$ to the corresponding bypass $\wh{QQ'}$. This map is evidently a quasi-isometry from the vertices of $T^{**}$ to the vertices of $T^*$, and this quasi-isometry commutes with the inclusions of~$T$ into $T^*$ and into $T^{**}$. The conclusions in the first two sentences of Proposition~\ref{PropTStarHyp}, namely hyperbolicity of $T^*$ and the fact that geodesics in $T$ are uniform reparameterized quasigeodesics in $T^*$, will therefore follow once we demonstrate the same conclusions with $T^{**}$ in place of $T^*$. Those conclusions for $T^{**}$ are identical to the conclusions of \cite[Proposition~2.5]{KapovichRafi:HypImpliesHyp} applied to the inclusion map $T \inject T^{**}$: the graph $T^{**}$ is hyperbolic; and each arc $\overline{VW} \subset T$ is uniformly Hausdorff close in $T^{**}$ to each geodesic in $T^{**}$ with the same endpoints $V,W$. So we need only verify that the inclusion map $T \inject T^{**}$ satisfies the hypotheses of \cite[Proposition~2.5]{KapovichRafi:HypImpliesHyp} with respect to the simplicial path metrics on $T$ and $T^{**}$ that assign length~$1$ to each edge.

One hypothesis \cite[Proposition~2.5]{KapovichRafi:HypImpliesHyp} is that $T$ be hyperbolic, which holds for all path metrics on trees. Another hypothesis is that the inclusion $T \inject T^{**}$ be a Lipschitz graph map which means that it takes each edge of $T$ to a bounded length edge path of $T^*$, but this is immediate since the inclusion is an isometry onto its image. 

The remaining hypotheses of \cite[Proposition~2.5]{KapovichRafi:HypImpliesHyp} are numbered (1), (2), (3), the first two of which are trivially satisfied using that the inclusion map from vertices of $T$ to vertices $T^{**}$ is surjective (this is why we use $T^{**}$ instead of $T^*$). The last hypothesis (3) says that there exists an integer $M > 0$ so that for any vertices $V \ne W \in T$, \, if $V,W$ are connected by an edge in $T^{**}$ then the diameter in $T^{**}$ of the path $\overline{VW} \subset T$ is uniformly bounded. The only case that needs attention is when $\overline{VW}$ is not a single edge in $T$ but $V,W$ are connected by an edge in $T^{**}$. This happens only if $\overline{VW}$ is a $\rho^*$ subpath of some element $N \in \cN$ of the Nielsen set, and so $\overline{VW}$ is a concatenation of a sequence of Nielsen paths $\overline{Q_{k-1}Q_k}$ in $N$, where the concatenation points form a consecutive sequence in $Z(N) \intersect \overline{VW}$ of the form
$$V=Q_0,Q_1,\ldots,Q_K=W
$$
Each pair $Q_i,Q_j$ is connected by an edge $\overline{Q_iQ_j} \subset T^{**}$ $(i \le j =0,\ldots,K)$. Since each vertex of $T$ along $\overline{VW}$ is contained in one of the Nielsen paths $\overline{Q_{k-1}Q_k}$ and hence its distance to one of $Q_{k-1}$ or $Q_k$ is at most $l_u(\rho) / 2$, it follows that the diameter of $\overline{VW}$ in $T^{**}$ is bounded above by $M = 1 + l_u(\rho)$.

\smallskip

It remains to prove the ``Furthemore'' sentence. Given $\xi \in \bdy T$, for any vertex $V \in T$ the ray $\overline{V\xi}$ is a reparameterized quasigeodesic in $T^*$. If $\xi$ is infinitely far in $T^*$ then the reparameterization of $\overline{V\xi}$ defines a quasigeodesic ray in $T^*$ which therefore limits on a unique $\xi^* \in \bdy T^*$. This point $\xi^*$ is well-defined because for any other vertex $W \in T$ the intersection of the two rays $\overline{V\xi}$, $\overline{W\xi}$ has finite Hausdorff distance in $T^*$ from each of them, hence the reparameterizations of those two rays limit on the same point in $\bdy T^*$. For any two points $\xi \ne \eta \in \bdy T$ that are both infinitely far in $T^*$, consider the line $\overline{\xi\eta}$. Choose a vertex $V \in T$, and choose sequences $x_i$ in $\overline{V\xi}$ and $y_i$ in $\overline{V\eta}$ so that in $T$ the sequence $x_i$ limits to $\xi$ and $y_i$ limits to $\eta$. It follows that in $T^*$ the sequence $x_i$ limits to $\xi^*$ and $y_i$ limits to $\eta^*$. Since $d^*(x_i,y_i) \to \infty$ and since $\overline{\xi\eta}$ is a reparameterized quasigeodesic, this is only possible if $\xi^* \ne \eta^*$, proving injectivity of the map $\xi \mapsto \xi^*$. To prove the required uniqueness property of $\overline{\xi\eta}$, consider any other two points $\xi' \ne \eta' \in T$. If one or both of $\xi'$ or $\eta'$ is finitely far in $T^*$ then the reparameterization of $\overline{\xi'\eta'}$ is a quasigeodesic ray or segment, hence has infinite Hausdorff distance in $T^*$ from the quasigeodesic line $\overline{\xi\eta}$. If both of $\xi',\eta'$ are infinitely far in $T^*$, and if $\overline{\xi\eta}$ and $\overline{\xi'\eta'}$ have finite Hausdorff distance in $T^*$, then it follows that $\{\xi^*,\eta^*\}=\{\xi'{}^*,\eta'{}^*\}$, and hence by injectivity we have $\{\xi,\eta\} = \{\xi',\eta'\}$ and therefore $\overline{\xi\eta}=\overline{\xi'\eta'}$. 

\qed

\paragraph{Proof of Proposition~\ref{PropFlaringInTStar}: Flaring in $T^*$.} We denote the constants of Proposition~\ref{PropConeQI} in shorthand as $K=K_{\ref{PropConeQI}}$, $C=C_{\ref{PropConeQI}}$.

Fix $\mu>1$ and $\eta \ge 0$. Consider $R \ge 1$, to be specified, and a pair of $\eta$-pseudo-orbits $V_r$, $W_r$ for $f_{T^*}$ defined for $-R \le r \le R$. Choose vertices $\wt V_r$, $\wt W_r \in \wt G$ projecting to $V_r$, $W_r$ respectively. In $\wt G$ denote the path $\ti\beta_r = \overline{V_r W_r}$, and let $\beta_r$ be its projection to $G$. Also, for $-R<r\le R$ denote $\ti\alpha_r = \overline{V_r, \ti f(V_{r-1})}$ and $\ti\omega_r = \overline{\ti f(W_{r-1}), W_r}$, which are the unique paths such that $\ti\beta_r = [\ti\alpha_r \ti f_\#(\ti\beta_{r-1}) \ti\omega_r]$. Let $\alpha_r,\omega_r$ be their projections to $G$. By assumption we have $d^*(f_{T^*}(V_{r-1}),V_r) \le \eta$ and $d^*(f_{T^*}(W_{r-1}),W_r) \le \eta$, and applying Proposition~\ref{PropConeQI} it follows that
$$L_\PF(\alpha_r),L_\PF(\omega_r) \le K \, \eta + C \equiv \eta'
$$
and so $f_\#(\beta_{r-1}) \coarsesim{\eta'}{L_\PF} \beta_r$. 

Let $\mu' = 2K^2\mu$.
By Proposition~\ref{PropFlaringGeneral} the path function $L_\PF$ satisfies the flaring condition of Definition~\ref{DefPathFlaring} with respect to $f$, from which using $\mu'$ and $\eta'$ we obtain constants $R' \ge 1$, $A' \ge 1$. We now specify $R = R'$. Also let 
$$A = \max\{KA' + C, 2K^2C + \frac{2C}{\mu}\}
$$
Applying Proposition~\ref{PropConeQI} it follows that if $d^*(V_0,W_0) \ge A$ then $L_\PF(\beta_0) \ge A'$. The flaring condition of $L_\PF$ therefore applies with the conclusion that
\begin{align*}
\mu' L_\PF(\beta_0) &\le \max\{L_\PF(\beta_{-R}), L_\PF(\beta_R)\} \\
\intertext{and so}
\mu \, d^*(V_0,W_0) &\le K \mu \, L_\PF(\beta_0) + K C \mu \\
 &\le \frac{K \mu}{\mu'} \max\{L_\PF(\beta_{-R}),L_\PF(\beta_R)\} + K C \mu\\
\intertext{Applying Proposition~\ref{PropConeQI} again we have}
 \mu \, d^*(V_0,W_0) 
      &\le \frac{K^2 \mu}{\mu'} \max\{d^*(V_{-R},W_{-R}),d^*(V_R,W_R)\} + K^2C\mu+C \\
 &\le \frac{1}{2} \max\{d^*(V_{-R},W_{-R}),d^*(V_R,W_R)\} + \frac{1}{2} \mu \, A \\
 &\le \frac{1}{2} \max\{d^*(V_{-R},W_{-R}),d^*(V_R,W_R)\} + \frac{1}{2} \mu \, d^*(V_0,W_0) \\
 \frac{1}{2} \mu \, d^*(V_0,W_0) &\le \frac{1}{2} \max\{d^*(V_{-R},W_{-R}),d^*(V_R,W_R)\}
\end{align*}
which completes the proof.
\qed

\subsection{Construction of $\S$.}
\label{SectionHypOfS}

We continue to fix the choice of $\Phi \in \Aut(F_n)$ representing $\phi$, and we consider the corresponding $\Phi$-twisted equivariant map  $f_{T^*} = f^\Phi_{T^*} \from T^* \to T^*$. Our definition of the suspension space $\S$ will formally depend on this choice (but see remarks at the end of the section regarding dependence on $\Phi$ of the constructions to follow).


\begin{definition}[The suspension space $\S$, its slices, fibers, and semiflow]
\label{DefSuspensionSpace}
 We define $\S$ to be the suspension space of $T^*$, namely the quotient of $T^* \times \Z \times [0,1]$ modulo the gluing identification $(x,k,1) \approx (f_{T^*}(x),k+1,0)$ for each $k \in \Z$ and $x \in T^*$. Let $[x,k,r] \in \S$ denote the equivalence class of $(x,k,r)$. We have a well-defined and continuous projection map $p \from \S \to \reals$ given by $p[x,k,r] = k+r$. For each closed connected subset $J \subset \reals$ we denote $\S_J = p^\inv(J)$ which we refer to as a \emph{slice} of~$\S$. In the special case of a singleton $s \in \reals$ we refer to $\S_s = p^\inv(s)$ as a \emph{fiber}. Each fiber may be regarded as a copy of~$T^*$, in the sense that if $k = \lfloor s \rfloor$ and $r = s-k$ then we obtain a homeomorphism $j_s \from T^* \to S_s$ given by $j_s(x)=[x,k,r]$; in the special case that $s$ is an integer we have $j_s(x)=[x,s,0]$. 
 
We have an \emph{action} $F_n \act \S$ which is induced by the action $F_n \act T^*$ as follows: for each $\gamma \in F_n$ and each $[x,k,r] \in \S$ we have
$$\gamma \cdot [x,k,r] = [\Phi^k(\gamma) \cdot x, k, r]
$$
This action is well-defined because, using $\Phi$-twisted equivariance of $f_{T^*} \from T^* \to T^*$, we have
\begin{align*}
\gamma \cdot [x,k,1] &= [\Phi^k(\gamma) \cdot x,k,1] = [f_{T^*}(\Phi^{k}(\gamma) \cdot x),k+1,0] =  \\
 &=  [\Phi(\Phi^{k}(\gamma)) \cdot f_{T^*}(x),k+1,0] \\
 &= [\Phi^{k+1}(\gamma) \cdot f_{T^*}(x),k+1,0] \\
&= \gamma \cdot [f_{T^*}(x),k+1,0] 
\end{align*}
Note that the homeomorphism $j_0 \from T^* \to \S_0$ is equivariant with respect to $F_n$ actions. For generally, for each integer $k$ the homeomorphism $j_k \from T^* \to \S_k$ is $\Phi^{-k}$-twisted equivariant, because
$$j_k(\gamma \cdot x) = [\gamma \cdot x,k,0] = \Phi^{-k}(\gamma) \cdot [x,k,0] = \Phi^{-k}(\gamma) \cdot j_k(x)
$$

We have a \emph{semiflow} $\S \times [0,\infinity) \to \S$, which is partially defined by the formula 
$$[x,k,s] \cdot t = [x,k,s+t] \qquad \text{($x \in T^*$, \, $k \in \Z$, \, $0 \le s \le 1$, \, $0 \le t \le 1-s$)}
$$
and which uniquely extends to all $t \ge 0$ by requiring the semiflow equation $(p \cdot t) \cdot u = p \cdot (t+u)$ to hold for all $t,u \ge 0$. In particular $\S_s \cdot t = \S_{s+t}$ for all $s \in \reals$, $t \ge 0$. For each $b \in \reals$ we define the \emph{first hitting} map $h_b \from \S_{(-\infinity,b]} \mapsto \S_b$ by letting $\xi \in \S_{(-\infty,b]}$ flow forward from $\xi \in \S_{p(\xi)}$ to $\S_b$ along the flow segment $\xi \cdot [0,b-p(\xi)]$, thus obtaining the formula $h_b(\xi) = \xi \cdot (b-p(\xi))$. 

This completes Definition~\ref{DefSuspensionSpace}.
\end{definition}

\begin{definition}[Piecewise Riemannian metric and geodesic metric on $\S$]
\label{DefSuspensionMetric}
We define a piecewise Riemannian metric on $\S$. Recall the piecewise Riemannian metric $ds^*$ on~$T^*$ (Section~\ref{SectionTStar}). For each edge $E \subset T^*$ and each integer $n$ we define a Riemannian metric $d_E$ on $E \times n \times [0,1]$ ($\approx E \times [0,1]$ and not depending on $n$), in two cases. In the case $E \subset T$, use the metric $d_E^2=\lambda^{2t} (ds^*)^2 + dt^2$; note that $f_{T^*}$ stretches the length of $E$ by a constant factor~$\lambda$, and hence the gluing map $(x,n,1) \mapsto (f_{T^*}(x),n+1,0)$ takes $E \times n \times 1$ isometrically onto $f_{T^*}(E) \times (n+1) \times 0$. In the case where $E \subset T^* \setminus T$, equivalently we are in the geometric or parageometric case and $E$ is a cone edge, use the metric $d_E^2 = (ds^*)^2 + dt^2$; note that $E'=f_{T^*}(E)$ is also a cone edge and that $f_{T^*}$ maps $E$ isometrically to $E'$, and so once again the gluing map $(x,n,1) \mapsto (f_{T^*}(x),n+1,0)$ takes $E \times n \times 1$ isometrically onto $f_{T^*}(E) \times (n+1) \times 0$. These metrics on all of the rectangles $E \times n \times [0,1]$ glue up isometrically along their common boundaries in~$\S$, as follows. First, for any two edges $E,E' \subset T^*$ having a common endpoint $x = E \intersect E' = \bdy E \intersect \bdy E'$, the restrictions to $[0,1] \approx x \times n \times [0,1]$ of the metric $d_E$ on $E \times [0,1] \approx E \times n \times [0,1]$ and the metric $d_{E'}$ on $E' \times [0,1] \approx E' \times n \times [0,1]$ are both equal to $dt$. Fixing $n$ and letting $E \subset T^*$ vary, this allows us to glue up all of the rectangles $E \times n \times [0,1]$ to get a well-defined piecewise Riemannian metric on $T^* \times n \times [0,1]$. Next, as noted above the gluing map from $T^* \times n \times 1$ to $T^* \times (n+1) \times 0$ maps each $E \times n \times 1$ isometrically onto $f_{T^*}(E) \times (n+1) \times 0$; letting $n$ vary this allows us to glue up $T^* \times n \times [0,1]$ to get the desired piecewise Riemannian metric on~$\S$.

By minimizing path lengths using the above piecewise Riemannian metric on~$\S$, we obtain a geodesic metric $d_\S(\cdot,\cdot)$ on~$\S$. We may similarly define a geodesic metric $d_J(\cdot,\cdot)$ on any slice $\S_J$, by minimizing path length with respect to the restricted piecewise Riemannian metric on $\S_J$. In the special case of a fiber $\S_s$, letting $n = \lfloor s \rfloor$, for each edge $E \subset T^*$ with image edge $j_s(E) \subset \S_s$ the metrics $d^*$ on $E$ and $d_s$ in $j_s(E)$ are related so that if $E \subset T$ then $j_s$ stretches the metric by a factor of $\lambda^{s-n}$, whereas if $E \subset T^* \setminus T$ then $j_s$ preserves the metric. In particular the map $j_s \from T^* \to \S_s$ is a $\lambda^{s-n}$ bilipschitz homeomorphism; it is therefore an isometry if $s$ is an integer, and a $\lambda$-bilipschitz homeomorphism in general.

This completes Definition~\ref{DefSuspensionMetric}.
\end{definition}

For use at the very end of the proof of the Hyperbolic Action Theorem, when verifying the WWPD conclusions, we shall need the following metric property of $\S$:

\begin{lemma}
\label{LemmaProjectNonIncrease}
For any two fibers $\S_s,\S_t \subset \S$ and any $x \in \S_s$, $y \in \S_t$ we have $d_\S(x,y) \ge \abs{s-t}$. 
\end{lemma}

\begin{proof}
The lemma follows by noting that for each edge $E \subset T^*$, the projection function $p \from E \times [0,1] \to [0,1]$ has the property that for each tangent vector $v \in T(E \times [0,1])$ we have $\abs{Dp(v)} \ge \abs{v}$, using the Riemannian metric $d_E$ on the right hand side of the inequality. 
\end{proof}

The following metric property will be used later to study how the inclusion map of fibers $\S_i \inject \S$ can distort distance.

\begin{lemma}
\label{LemmaSectionQI}
For each integer $m$ there exist constants $k_m \ge 1$, $c_m \ge 0$ such that for each integer $a$ and each $s \in J = [a,a+m]$ the inclusion map $i \from \S_s \inject \S_J$ is a $k_m,c_m$ quasi-isometry.
\end{lemma}

\begin{proof} We construct a cycle of equivariant Lipschitz maps as follows:
$$\xymatrix{
\S_s \ar[r]^{i} & \S_J \ar[d]^{h_{a+m}} \\
\S_a \ar[u]^{h_r} & \S_{a+m} \ar[l]^{\bar h}
}$$
The inclusion map $i$ is $1$-Lipschitz. The equivariant maps $h_{a+m}$ and $h_r$ are first hitting maps, each of which is $\lambda^m$ Lipschitz. The map $\bar h$ will be an equivariant coarse inverse to the equivariant map $h_{a+m} \from \S_a \to \S_{a+m}$, for the construction of which we consider the commutative diagram
$$\xymatrix{
T^* \ar[r]^{j_a} \ar[d]_{f^m_{T^*}} 
                                            & \S_a \ar[d]^{h_{a+m}} \\
T^* \ar[r]_{j_{a+m}} \ar@/^3pc/@{.>}[u]^{\bar f^m_{T^*}}       
                                            & \S_{a+m} \ar@/_3pc/@{.>}[u]_{\bar h}
}$$
In this diagram the map $f^m_{T^*}$ is $\Phi^m$-twisted equivariant, and the map $h_{a+m}$ is (untwisted) equivariant. The top and bottom maps are the instances $k=a$ and $k=a+m$ of the $\Phi^{-k}$-twisted equivariant isometry $j_k \from T^* \to \S_k$, whose inverse $j^\inv_k \from \S_k \to T^*$ is $\Phi^k$ twisted equivariant. By Lemma~\ref{LemmaTStarQI} the map $f_{T^*}$ has a $\Phi^{-1}$-twisted equivariant Lipschitz coarse inverse $\bar f_{T^*}$. It follows that $f^m_{T^*}$ has a $\Phi^{-m}$-twisted equivariant Lipschitz coarse inverse $\bar f^m_{T^*}$ whose Lipschitz and coarse inverse constants depend on $m$. We may therefore fill in the diagram with the map $\bar h = j_a \circ \bar f^m_{T^*} \circ j_{a+m}^\inv$ which makes the diagram commute and which is an untwisted equivariant Lipschitz coarse inverse for $h_{a+m}$, with Lipschitz and coarse inverse constants depending on~$m$. 

Going round the cycle from $\S_s$ to itself and from $\S_J$ to itself we obtain two equivariant self-maps both of which move points uniformly bounded distances depending on $m$. The maps $i \from \S_s \to \S_J$ and $h_r \circ \bar h \circ h_{a+m} \from \S_J \to \S_s$ are therefore Lipschitz coarse inverses and hence quasi-isometries with constants depending on~$m$.
\end{proof}

\paragraph{Remark.} The definition of the suspension space $\S$ depends ostensibly on the choice of representative $\Phi \in \Aut(F_n)$ of $\phi$, but in fact different choices of $\Phi$ produces suspension spaces which are $F_n$-equivariantly isometric, as can easily be checked using the fact that distinct choices of $\Phi$ differ by an inner automorphism of~$F_n$.

\subsection{Proof of hyperbolicity of $\S$.}
\label{SectionSHypProof}

We prove that $\S$ is a hyperbolic metric space by applying the Mj-Sardar combination theorem \cite{MjSardar:combination}, a descendant of the Bestvina--Feighn combination theorem \cite{BestvinaFeighn:combination}. The hypotheses of the Mj-Sardar theorem consist of an opening hypothesis and four numbered hypotheses which we must check. The last of those four is the ``flaring condition'' which we prove by application of Proposition~\ref{PropFlaringInTStar}.

\smallskip

\newcommand\hbullet[1]{\hphantom{x}\quad$\bullet$ \emph{#1}}

\emph{Opening hypothesis of \cite[Theorem 4.3]{MjSardar:combination}: A metric bundle.} This hypothesis says that $\S$ is a metric bundle over $\reals$ with respect to the map $p \from \S \to \reals$. We must therefore verify that the map $p \from \S \to \reals$ satisfies the definition of a metric bundle as given in  \cite[Definition 1.2]{MjSardar:combination}. First, $p$ must be Lipschitz map, and each fiber $\S_s$ must be a geodesic metric space with respect to the path metric induced from~$\S$, each of which we have already verified. Also, item 2) of \cite[Definition 1.2]{MjSardar:combination}, when translated into our setting, requires that for each interval $[a,b] \subset \reals$ such that $b-a \le 1$, and for each $s \in [a,b]$ and $\xi \in \S_s$, there exists a path in $\S$ of uniformly bounded length that passes through $\xi$ and has endpoints on $\S_a$ and $\S_b$ respectively. To obtain this path choose $\eta \in \S_a$ so that $\eta \cdot (s-a)=\xi$ and take the path $t \mapsto \eta \cdot t$ defined on $a \le t \le b$, whose length equals $b-a \le 1$.

The remaining verification needed for $\S$ to be a metric bundle is that the set of inclusions $\S_s \inject \S$ ($s \in \reals$) is uniformly proper, meaning that these inclusions are uniformly Lipschitz --- in fact they are all $1$-Lipschitz by construction --- and that there exists a single nondecreasing ``gauge'' function $\delta \from [0,\infinity) \to [0,\infinity)$ for these inclusions, having the following property:
\begin{itemize}
\item[$(*)$] for any $s \in \reals$, any $x,y \in \S_s$, and any $D \ge 0$, if $d_\S(x,y) \le D$ then $d_s(x,y) \le \delta(D)$. 
\end{itemize}
To define $\delta$ consider $x,y \in \S_s$ connected by a geodesic path $\gamma$ in $\S$ with $\Length(\gamma) = d_\S(x,y) \le D$. The projection $p \composed \gamma$ in $\reals$ has length $\le D$ and so, letting $m = \lfloor D+2 \rfloor$, there are integers $a$ and $b=a+m$ such that $\image(p \composed \gamma) \subset [a,b]$, implying that $\image(\gamma) \subset \S_{[a,b]}$ which implies in turn that $d_\S(x,y) = d_{[a,b]}(x,y)$. Applying Lemma~\ref{LemmaSectionQI} we have $\frac{1}{k_m} \, d_s(x,y) - c_m \le d_{[a,b]}(x,y)$ and hence $d_s(x,y) \le k_m d_\S(x,y) + k_m \, c_m$. We may assume that $k_m$ and $c_m$ are nondecreasing functions of $m$ and hence $k_{\lfloor D \rfloor}$ and $c_{\lfloor D \rfloor}$ are nondecreasing functions of $D$, and so using $\delta(D) = k_{\lfloor D \rfloor}(D + c_{\lfloor D \rfloor})$ we are done.

We record for later use the following:

\begin{lemma}\label{LemmaUnifProp}
For each $s \in \reals$ the inclusion $\S_s \inject \S$ is uniformly proper. In particular, the inclusion $T^* = \S_0 \inject \S$ is uniformly proper.
\qed\end{lemma}

%
%

\smallskip

\emph{Hypotheses (1), (2) of \cite[Theorem 4.3]{MjSardar:combination}: Base and fiber hyperbolicity.} These hypotheses require that the base space $\reals$ is hyperbolic which is evident, and that the fibers $\S_s$ are hyperbolic with uniform hyperbolicity constant. Proposition~\ref{PropTStarHyp} gives us a constant $\delta' \ge 0$ such that $T^*$ is $\delta'$-hyperbolic. Since each fiber $\S_s$ is $\lambda$-bilipschitz homeomorphic to $T^*$, it follows that $\S_s$ is hyperbolic with a uniform hyperbolicity constant $\delta$ depending only on $\delta'$ and $\lambda$.

\smallskip

\emph{Hypothesis (3) of \cite[Theorem 4.3]{MjSardar:combination}: Barycenters.} This hypothesis says that the barycenter maps $\bdy^3 \S_s \to \S_s$ are uniformly coarsely surjective as $s$ varies. We review what this means from \cite[Section 2 around the heading ``The barycenter map'']{MjSardar:combination}. Given a $\delta$-hyperbolic geodesic metric space~$X$ with Gromov boundary $\bdy X$, consider 
the triple space 
$$\bdy^3 X = \{(\xi_1,\xi_2,\xi_3) \in (\bdy X)^3 \suchthat \xi_i \ne \xi_j \,\text{if}\, i \ne j\}
$$
The barycenter map $\bdy^3 X \to X$ is a coarsely well-defined map as follows. There exists constants $K \ge 1$, $C \ge 0$ depending only on $\delta$ such that for any two points $\xi_1 \ne \xi_2 \in \bdy X$ there exists a $K,C$-quasigeodesically embedded line in $X$ having endpoints $\xi_1,\xi_2$; we use the notation $\overline{\xi_1\xi_2}$ for any such quasigeodesic line. By \cite[Lemma 2.7]{MjSardar:combination} there exist constants $D,L \ge 0$ depending only on $\delta$ such that for each triple $\xi=(\xi_1,\xi_2,\xi_3) \in \bdy^3 X$ there exists a point $b_\xi \in X$ which comes within distance $D$ of any of the lines $\overline{\xi_i\xi_j}$, $i \ne j \in \{1,2,3\}$, and for any other such point $b'_\xi$ the distance between $b_\xi$ and $b'_\xi$ is~$\le L$. Once the constants $K, C, D, L$ have been chosen, any such point $b_\xi$ is called a \emph{barycenter} of $\xi$, and any map $\bdy^3 X \to X$ taking each triple $\xi$ to a barycenter $b_\xi$ is called a \emph{barycenter map} for $X$.

To say that the barycenter maps $\bdy^3\S_s \to \S_s$ are uniformly coarsely surjective means that there exists a ``coboundedness constant'' $E \ge 0$ such that for each $s \in \reals$ the image of each barycenter map $\bdy^3\S_s \to \S_s$ comes within distance $E$ of each point of~$\S_s$. For the hyperbolic space $T^*$, the action $F_n \act T^*$ has a fundamental domain $\tau \subset T^*$ of bounded diameter, and so $E'' = \diam(\tau)$ is a coboundedness constant for any $F_n$-equivariant barycenter map $\bdy^3 T^* \to T^*$, hence there is a uniform coboundedness constant $E'$ for all barycenter maps $\bdy^3 T^* \to T^*$. Since each of the fibers $\S_s$ comes equipped with a $\lambda$-bilipschitz homeomorphism $j_s \from T^* \to \S_s$, their barycenter maps have a uniform coboundedness constant $E = \lambda E'$.

\smallskip

\emph{Hypothesis (4) of \cite[Theorem 4.3]{MjSardar:combination} aka \cite[Definition 1.12]{MjSardar:combination}.} Here is a slight restatement of the hypothesis specialized to our present context.

\begin{description}
\item[Flaring 1:] For all $k_1 \ge 1$, $\nu_1 > 1$, there exist integers $A_1, R \ge 0$ such that for any~$s \in \reals$ and any two $k_1$-quasigeodesics 
$$\gamma_1,\gamma_2 \from [s-R,s+R] \to \S
$$
which are sections of the projection map $p$ --- meaning that $p \composed \gamma_i$ is the identity on the interval~$[s-R,s+R]$ --- the following implication holds:

\quad $\bf{(F_1)}$ if $d_s(\gamma_1(s),\gamma_2(s)) \ge A_1$ then
$$\nu_1 \,\cdot\, d_s\bigl(\gamma_1(s),\gamma_2(s)\bigr) \le \max\{ d_{s-R}\bigl(\gamma_1(s-R),\gamma_2(s-R)\bigr), d_{s+R}\bigl(\gamma_1(s+R),\gamma_2(s+R)\bigr) \}
$$
\end{description}
This statement tautologically implies the flaring hypothesis given in \cite[Definition 1.12]{MjSardar:combination}, the difference being that in the latter statement the quantifier order starts out as ``For all $k \ge 1$ there exist $\mu > 1$ and integers $A,r \ge 0$ such that\ldots'' with the remainder of the statement unchanged (a simple geometric estimation argument yields the converse implication, but we will not need this).

For proving Flaring 1 we first reduce it to a ``discretized'' version taking place solely in the fibers $\S_r$ for integer values of $r$, as follows:

\begin{description}
\item[Flaring 2:] For all $k_2 \ge 1$, $\nu_2 > 1$, there exist integers $A_2, R \ge 0$ such that for any integer~$m \in \Z$ and any two $k_2$-quasigeodesic maps
$$\delta_1,\delta_2 \from \{m-R,\ldots,m,\ldots,m+R\} \to \S
$$
which are sections of the projection map $p$, the following implication holds:

\quad $\bf{(F_2)}$ if $d_m(\delta_1(m),\delta_2(m)) \ge A_2$ then
\end{description}
\vspace*{-2mm}
$$\nu_2 \,\cdot\, d_m\bigl(\delta_1(m),\delta_2(m)\bigr) \le \max\{d_{m-R}\bigl(\delta_1(m-R),\delta_2(m-R)\bigr), d_{m+R}\bigl(\delta_1(m+R),\delta_2(m+R)\bigr) \}
$$

\medskip\noindent
To show that Flaring 2 implies Flaring 1, choose $k_1,\nu_1$. Consider any integer $R \ge 0$, any $s \in \reals$, and any pair $\gamma_1,\gamma_2 \from [s-R,s+R] \to \S$ of $k_1$-quasigeodesic sections of the projection map~$p$. Let $m = \lfloor s \rfloor$ and let $t = s - m$. The semiflow restricts to $\lambda$-bilipschitz homeomorphisms $h_r \from \S_r \mapsto \S_{r+t}$ defined by $h_r[x,r,0] = [x,r,t]$ for any integer~$r$, having the property that the distance from each $[x,r,0]$ to its $h_r$-image $[x,r,t]$ in $\S$ is at most~$1$. It follows that the functions
$$\delta_j \from \{m-R,\ldots,m+R\} \to \S, \quad \delta_j(r) = h_r^\inv(\gamma_j(r+t))
$$
are $k_2$-quasi-isometric sections of~$p$, where the constant $k_2$ depends only on $k_1$ and~$\lambda$. Applying Flaring~2, for any $\nu_2 > 1$ there exist integers $A_2,R \ge 0$ such that the implication $\bf{(F_2)}$ holds. Again using that the maps $h_r$ are $\lambda$-bilipschitz homeomorphisms, if we take $\nu_2 = \nu_1 \cdot \lambda^2$ and $A_1 = A_2 \cdot \lambda$ then the hypothesis of $\bf{(F_1)}$ implies the hypothesis of $\bf{(F_2)}$ and the conclusion of $\bf{(F_2)}$ implies the conclusion of $\bf{(F_1)}$. 

It remains to verify Flaring 2 by applying Proposition~\ref{PropFlaringInTStar} which we restate here for convenience in a form matching that of Flaring~2: 
\begin{description}
\item[Flaring 3:] For each $\mu > 1$, $\eta \ge 0$ there exist integers $A_2 \ge 0$, $R \ge 1$ such that for any pair of $\eta$-pseudo-orbits $V_r$ and $W_r$ of $f_{T^*} \from T^* \to T^*$ defined for integers $-R \le r \le R$, the following implication holds:

\quad$\bf{(F_3)}$ if $d^*(V_0,W_0) \ge A_2$ then 
\vspace*{-2mm}
$$\mu \cdot d^*(V_0,W_0) \le \max\{ d^*(V_{-R},W_{-R}) , d^*(V_R,W_R) \}
$$
\end{description}
For the proof that Flaring~3$\implies$Flaring~2, for each $r \in \Z$ we have a commutative diagram
$$\xymatrix{
T^* \ar[r]^{f_{T^*}}  \ar[d]_{j_{r-1}} & T^* \ar[d]^{j_r} \\
\S_{r-1} \ar[r]_{h_r} & \S_r
}$$
where $h_r$ is the first hitting map. Note that in this diagram the bottom is equivariant, the top is $\Phi$-twisted equivariant, the left is a $\Phi^{1-r}$-twisted equivariant isometry, and the right is $\Phi^{-r}$-twisted equivariant isometry.

Choose $k_2 \ge 1$, $\nu_2>1$, consider integers $R \ge 0$ and $m$, and consider a pair of $k_2$-quasigeodesic maps $\delta_i \from \{m-R,\ldots,m+R\} \to \S$ which are sections of $p$, for $i=1,2$. It follows that if $m-R \le r-1 < r \le m+R$ then $d_\S(\delta_i(r-1),\delta_i(r)) \le 2 k_2$ (recall that ``$k_2$-quasigeodesic'' is synonymous with ``$(k_2,k_2)$-quasigeodesic''). For each $x \in \S_{r-1}$ we have $d_\S(h_r(x),x)) \le 1$ and hence $d_\S(h_r(\delta_i(r-1)),\delta_i(r)) \le 2k_2 + 1$. For $r \in \{-R,\ldots,+R\}$ denote $V_r = j_{m+r}^\inv(\delta_1(m+r))$, $W_r = j_{m+r}^\inv(\delta_2(m+r))$. It follows that $d_*(f_{T^*}(V_{r-1}),V_r), d_*(f_{T^*}(W_{r-1},W_r) \le 2k_2+1$, and so the sequences $(V_r)$ and $(W_r)$ are both $\eta$-pseudo-orbits of $f_{T^*}$, defined for $m-R \le r \le m+R$ and with $\eta = 2k_2+1$. Since $j_m$ is an isometry we have $d_m(\delta_1(m),\delta_2(m))=d^*(V_0,W_0)$, and so the hypothesis of $\bf{(F_2)}$ implies the hypothesis of $\bf{(F_3)}$. Similarly since $j_{m-R}$ and $j_{m+R}$ are isometries, the conclusion of $\bf{(F_3)}$ implies the conclusion of $\bf{(F_2)}$ using $\mu=\nu_2$. We have therefore proved that Flaring~3$\implies$Flaring~2.

\medskip

This completes the verification of the hypotheses of the Mj-Sardar combination theorem, and so the space $\S$ is therefore hyperbolic.

\section{Abelian subgroups of $\Out(F_n)$}
\label{SectionAbelianSubgroups}
This section reviews some background material needed for the rest of the paper. Section~\ref{SectionBackgroundReview} contains basic material from \recognition\ regarding automorphisms and outer automorphisms (see also \SubgroupsOne\ for a comprehensive overview). Section~\ref{SectionDisintegration} reviews elements of the theory of abelian subgroups of $\Out(F_n)$ developed in \cite{FeighnHandel:abelian}, focussing on disintegration subgroups. 


\subsection{Background review}
\label{SectionBackgroundReview}

\subsubsection{More about \cts} 
\label{SectionMoreCTs}

In Notations~\ref{NotationsFlaring} we reviewed features of a \ct\ $f \from G \to G$ with associated $f$-invariant filtration $G_1 \subset \cdots \subset G_u$ under the assumption that the top stratum $H_u$ is \eg. In studying disintegration groups we shall need some further defining properties and derived properties of~\cts\ (for this section only the reader may ignore the assumption that the top stratum is \eg). We shall refer to \recognition\ for material specific to \cts, to \BookOne\ for more general material, and to \SubgroupsOne\ for certain ``compilation'' results with multiple sources in \recognition\ or \BookOne. One may also consult \SubgroupsOne\ for a comprehensive overview.

\paragraph{General properties of strata.} \cite[Section 2.6]{\recognitionTag} \quad Each stratum $H_i = G_i \setminus G_{i-1}$ is either an \emph{irreducible stratum} meaning that its transition matrix $M_i$ is an irreducible matrix, or a \emph{zero stratum} meaning that $M_i$ is a zero matrix. Each irreducible stratum satisfies one of the following:

\smallskip\emph{$H_i$ is an \eg\ stratum:} \cite[Remark 3.20]{\recognitionTag} The matrix $M_i$ is a $k \times k$ Perron-Frobenius matrix for some $k \ge 2$, having eigenvalue $\lambda>1$; or

\smallskip\emph{$H_i$ is an \neg\ stratum:} \cite[Section 4.1]{\recognitionTag} $H_i = E_i$ is a single edge with an orientation such that $f(E_i) = E_i u$ where $u$ is either trivial or a closed path in $G_{i-1}$ having distinct initial and terminal directions. 
%
%
An \neg\ stratum $H_i$ is a \emph{fixed stratum} if $u$ is trivial, a \emph{linear stratum} if $u$ is a Nielsen path (equivalently $u$ is a periodic Nielsen path), and a \emph{superlinear stratum} otherwise. 

\paragraph{Properties of \neg-linear strata:} An \neg-linear stratum $H_i = E_i$ will also be referred to as a \emph{linear edge} of~$G$. The linear edges of $G$ have the following features:

\smallskip\emph{Twist path and twist coefficient:} \cite[Section 4.1]{\recognitionTag} \, For each linear edge $E_i$ we have $f(E_i) = E_i w_i^{d_i}$ for a unique closed Nielsen path $w_i$ which is \emph{root free} meaning that $w_i$ is not an iterate of any shorter closed path (equivalently, if $p$ is the base point of $w_i$, then the element of the group $\pi_1(G,p)$ represented by $w_i$ is root free). We say that $w_i$ is the \emph{twist path} of $E_i$ and that the integer $d_i \ne 0$ is its \emph{twist coefficient}. If $E_j \ne E_i$ is another linear edge having twist path $w_j$, and if $w_i$ and $w_j$ determine the same conjugacy class up to inversion, then $w_i=w_j$ and $d_i \ne d_j$.

\smallskip\emph{\neg\ Nielsen paths:} \cite[Definition 4.7]{\recognitionTag}  \, For each \neg\ edge $E_i$, if there is an indivisible Nielsen path contained in $G_i$ but not in $G_{i-1}$ then $E_i$ is a linear edge, and every such Nielsen path has the form $E_i w_i^k \overline E_i$ for some $k \ne 0$.

\smallskip\emph{Exceptional paths:} \cite[Definition 4.1]{\recognitionTag} \, These are paths of the form $E_i w^k \overline E_j$ where $E_i \ne E_j$ are linear edges having the same twist path $w=w_i=w_j$ and having twist coefficients $d_i,d_j$ of the same sign.

\paragraph{Properties of \eg\ strata:} These properties were stated in Notations~\ref{NotationsFlaring} with respect to the top \eg\ stratum $H_u$, but we go over them again here for an arbitrary \eg-stratum~$H_i$, and with a somewhat different emphasis. 

\smallskip\emph{Lines:} \cite[Section 2.2]{\BookOneTag} \, Recall the spaces of lines in $F_n$ and in $G$ and the canonical homeomorphism between them:
\begin{align*}
\B(F_n) & = \{ \text{2 point subsets of $\bdy F_n$} \} / F_n \\
\B(G)  & = \{ \text{bi-infinite paths in $G$} \} / \text{reparameterization} 
\end{align*}
where the topology on $\B(F_n)$ is induced by the Hausdorff topology on compact subsets of $\bdy F_n$, and the topology on $\B(G)$ has a basis element for each finite path $\gamma$ consisting of all bi-infinite paths having $\gamma$ as a subpath. The homeomorphism $\B(F_n) \leftrightarrow \B(G)$ is induced by the universal covering map $\wt G \to G$ and the natural bijection $\bdy F_n \approx \bdy \wt G$. We refer to this homeomorphism by saying that a line in $F_n$ is \emph{realized} by the corresponding line in~$G$.
 
\smallskip\emph{Attracting laminations:} \cite[Section 3.1]{\BookOneTag}, \cite[Remark 3.20]{\recognitionTag} \, Associated to $H_i$ is its \emph{attracting lamination} $\Lambda_i \subset \B(F_n)$ which is the set of all lines in $F_n$ whose realization $\ell$ in $G$ has the property that for each finite subpath $\gamma$ of $\ell$ and each edge $E \subset H_i$ there exists $k \ge 1$ such that $\gamma$ is a subpath of $f^k_\#(E)$. For distinct \eg\ strata $H_i, H_j$ ($i \ne j$) the corresponding laminations $\Lambda_i,\Lambda_j$ are distinct. The set of laminations $\L(\phi)=\{\Lambda_i\}$ is independent of the choice of \ct\ representative $f \from G \to G$. 

\smallskip\emph{EG Nielsen paths:} \cite[Corollary 4.19]{\BookOneTag}, \SubgroupsOne\ Fact 1.42. \, Up to inversion there exists at most one indivisible periodic Nielsen path $\rho$ contained in $G_i$ but not in $G_{i-1}$. Its initial and terminal directions are distinct, and at least one endpoint of $\rho$ is not contained in $G_{i-1}$. 

\smallskip\emph{Geometricity:} \SubgroupsOne, Fact 2.3. \, The stratum $H_i$ is \emph{geometric} if and only if $\rho$ exists and is a closed path.

\smallskip\emph{Fixed circuits:} \SubgroupsOne, Fact 1.39. \, If $\sigma$ is a circuit fixed by $f_\#$ then $\sigma$ is a concatenation of fixed edges and indivisible Nielsen paths.

\paragraph{Properties of zero strata:} \cite[Definition 2.18, Definition 4.4, Definition 4.7]{\recognitionTag}

Each zero stratum $H_i \subset G$ has the following properties: 

\smallskip\emph{Envelopment:} There exist indices $s<i<r$ such that $H_s$ is an irreducible stratum, $H_r$ is an \eg\ stratum, each component of $G_r$ is noncontractible, and each $H_j$ with $s<j<r$ is a zero stratum and a contractible component of $G_{r-1}$. We say that the zero strata $H_j$ with $s<j<r$ are \emph{enveloped} by $H_r$, and we denote $H^z_r$ to be the union of $H_r$ with its enveloped zero strata. The filtration element $G_s$ is the union of the noncontractible components of $G_{r-1}$, and $H^z_r = G_r \setminus G_s$. 

\smallskip\emph{Taken paths.}  These are the paths $\mu$ in $H_i$ for which there exists an edge $E$ of some irreducible stratum $H_q$ with $q>i$, and there exists $k \ge 1$, such that $\mu$ is a maximal subpath in $H_i$ of the path $f^k_\#(E)$; we say more specifically that the path $\mu$ is \emph{$q$-taken}. If $H_r$ is the \eg\ stratum that envelopes $H_i$ then every edge $E \subset H_i$ is an $r$-taken path, from which it follows that the endpoints of $E$ are vertices of $H_r$ not contained in $G_{r-1}$.

\paragraph{Properties of complete splittings:} \cite[Section 4]{\BookOneTag}, \cite[Definition 4.4]{\recognitionTag}

A \emph{splitting} of a path $\gamma$ in $G$ is a concatenation expression $\gamma = \gamma_1 \cdot \ldots \cdot \gamma_J$ such that for all $k \ge 1$ we have $f^k_\#(\gamma) = f^k_\#(\gamma_1) \cdot \ldots \cdot f^k_\#(\gamma_J)$.  The characteristic property of~a~\ct\ --- short for ``completely split relative train track map'' --- is the following:

\smallskip\emph{Complete splitting:} For each edge $E$ there is a unique splitting $f(E) = \sigma_1 \cdot \ldots \cdot \sigma_n$ which is \emph{complete}, meaning that each term $\sigma_i$ is either an edge in an irreducible stratum, an \eg\ indivisible Nielsen path, an \neg\ indivisible Nielsen path, an exceptional path, or a maximal taken subpath in a zero stratum.

\subsubsection{Principal automorphisms and rotationless outer automorphisms.} 
\label{SectionPrincipal}
Consider an automorphism $\Phi \in \Aut(F_n)$,  with induced boundary homeomorphism denoted $\wh\Phi \from \bdy F_n \to \bdy F_n$, and with fixed subgroup denoted $\Fix(\Phi) \subgroup F_n$. Denote the sets of periodic points and fixed points of $\wh\Phi$ as $\Per(\wh\Phi)$ and $\Fix(\wh\Phi) \subset \bdy F_n$, respectively. Consider $\xi \in \Per(\wh\Phi)$ of period $k \ge 1$. We say $\xi$ is an \emph{attractor} for $\wh\Phi$ if it has a neighborhood $U \subset \bdy F_n$ such that $\wh\Phi^k(U) \subset U$ and the collection $\{\wh\Phi^{ki}(U) \suchthat i \ge 0\}$ is a neighborhood basis of $\xi$.  Also, $\xi$ is a \emph{repeller} for $\wh\Phi$ if it is an attractor for $\wh\Phi^\inv$.
Within $\Per(\wh\Phi)$ and $\Fix(\wh\Phi)$ denote the sets of attracting, repelling, and nonrepelling points, respectively, as 
\begin{align*}
\Per_+(\wh\Phi), \, \Per_-(\wh\Phi), \, \Per_N(\wh\Phi) &\subset \Per(\wh\Phi) \\
\Fix_+(\wh\Phi), \, \Fix_-(\wh\Phi), \, \Fix_N(\wh\Phi) &\subset \Fix(\wh\Phi)
\end{align*}
For each $c \in F_n$ associated inner automorphism $i_c(a)=cac^\inv$ we use the following special notations 
$$\bdy c = \Fix(\,\wh i_c) = \{\bdy_-c, \bdy_+c\}, \quad \{\bdy_- c\} = \Fix_-(\,\wh i_c), \quad \{\bdy_+ c\} = \Fix_+(\,\wh i_c)
$$
The following equivalences are collected in \cite[Lemma 2.1]{\recognitionTag} based on results from \cite{GJLL:index} and~\cite{BFH:Solvable}.

\begin{fact}\label{FactFixPhi}
For each $\Phi \in \Aut(F_n)$ and each nontrivial $c \in F_n$ we have
\begin{align*}
\Phi \in \Stab(c) &\iff i_c \,\, \text{commutes with} \,\, \Phi  \,\, \iff \Fix(\wh\Phi) \,\, \text{is invariant under} \,\, \hat i_c \\
  &\iff \bdy c \subset \Fix(\wh\Phi)
\end{align*}
\end{fact}

\paragraph{Principal automorphisms.} \cite[Definition 3.1]{\recognitionTag}. \quad We say that $\Phi$ is a \emph{principal automorphism} if $\abs{\Fix_N(\Phi)} \ge 2$, and furthermore if $\abs{\Fix_N(\Phi)}=2$ then $\Fix_N(\Phi)$ is neither equal to $\bdy c$ for any nontrivial $c \in F_n$, nor equal to the set of endpoints of a lift of a generic leaf of an attracting lamination of the outer automorphism representing $\Phi$. For each $\phi \in \Out(F_n)$ let $P(\phi) \subset \Aut(F_n)$ denote the set of principal automorphisms representing $\phi$. 

For each $\phi \in \Out(F_n)$ let $P^\pm(\phi)$ denote\footnote{In \abelian\ the definition of $P^\pm(\phi)$ was incorrectly stated as $P(\phi) \union P(\phi^\inv)$. The definition given here should replace the one in \abelian. No further changes are required in \abelian, because the arguments there which use $P^\pm(\phi)$ are all written using the current, correct definition.} the set of all $\Phi \in \Aut(F_n)$ representing $\phi$ such that either $\Phi$ or $\Phi^\inv$ is principal, equivalently, 
$$P^\pm(\phi) = P(\phi) \union (P(\phi^\inv))^\inv
$$
%
%
The symmetry equation $P^\pm(\phi^\inv)=(P^\pm(\phi))^\inv$ is useful in situations where one is trying to prove a certain property of $\Phi \in P^\pm(\phi)$ that is symmetric under inversion of~$\Phi$: one may reduce to the case that $\Phi$ is principal by replacing $\phi \in \Out(F_n)$ and $\Phi$ by their inverses in the case where $\Phi$ is not already principal. We use this reduction argument with little comment in many places.

\paragraph{Rotationless outer automorphisms.}
\cite[Definition 3.1 and Remark 3.2]{\abelianTag}. \quad The concept of \emph{forward rotationless} outer automorphisms is defined in \recognition, where it is proved that the forward rotationless outer automorphisms are precisely those outer automorphisms which have \ct\ representatives. Here we use the stricter property of \emph{rotationless} defined in \abelian, which is symmetric under inversion, and which is better adapted to the study of abelian subgroups. 

We say that $\phi \in \Out(F_n)$ is \emph{rotationless} if two conditions hold. First, for each $\Phi \in P(\phi)$ we have $\Per(\Phi)=\Fix(\Phi)$ (in ``forward rotationless'' this condition is replaced by the weaker $\Per_N(\Phi)=\Fix_N(\Phi)$). Second, for each integer $k \ge 1$ the map $\Phi \mapsto \Phi^k$ is a bijection between $P^\pm(\phi)$ and $P^\pm(\phi^k)$ --- recall from \cite[Remark 3.2]{\abelianTag} that injective is always true, and so bijective holds if and only if surjective holds.

\paragraph{Expansion factor homomorphisms.} \cite[Section 3.3]{\BookOneTag} \, Consider $\phi \in \Out(F_n)$ and an attracting lamination $\Lambda \in \L(\phi)$. Under the action of $\Out(F_n)$ on the set of lines $\B(F_n)$, consider the subgroup $\Stab(\Lambda) \subgroup \Out(F_n)$ that stabilizes $\Lambda$. The \emph{expansion factor homomorphism} of $\Lambda$ is the unique surjective homomorphism $\PF_\Lambda \from \Stab(\Lambda) \to \Z$ such that for each $\psi \in \Stab(\Lambda)$ we have $\PF_\Lambda(\psi) \ge 1$ if and only if $\Lambda \in \L(\psi)$. Furthermore, there exists $\mu > 1$ such that if $\psi \in \Stab(\Lambda)$ and if $\PF_\Lambda(\psi) > 1$ then any relative train track representative $f \from G \to G$ of~$\psi$ has an \eg-aperiodic stratum corresponding to $\Lambda$ on which the Perron-Frobenius eigenvalue of the transition matrix equals~$\mu^{\PF_\Lambda(\psi)}$.

\paragraph{Twistors (aka Axes).} 

\cite[Lemma 4.40 and the preceding page]{\recognitionTag}. Recall that two elements $a,b \in F_n$ are said to be \emph{unoriented conjugate} if $a$ is conjugate to $b$ or to~$b^\inv$. The ordinary conjugacy class of $a$ is denoted $[a]$ and the unoriented conjugacy class is denoted $[a]_u$. For any marked graph $G$, nontrivial conjugacy classes of $F_n$ correspond bijectively with circuits $S^1 \mapsto G$ up to orientation preserving homeomorphisms of the domain $S^1$. Note that $a$ is root free in $F_n$ if and only if the oriented circuit $S^1 \mapsto G$ representing $[a]$ does not factor through a nontrivial covering map of $S^1$, and $a,b$ are unoriented conjugate if and only if the oriented circuits representing $[a],[b]$ differ by an arbitrary homeomorphism of $S^1$.

Consider $\phi \in \Out(F_n)$ and a nontrivial, unoriented, root free conjugacy class $\mu = [c]_u$, $c \in F_n$. For any two representatives $\Phi_1 \ne \Phi_2 \in \Aut(F_n)$ of $\phi$, the following are equivalent: 
\begin{align*}
\Phi_1, \Phi_2 \in \Stab(c) & \iff \bdy c = \Fix(\wh\Phi_1) \intersect \Fix(\wh\Phi_2)
\end{align*}
Furthermore, if these equivalent conditions hold then $\Phi_2^\inv \Phi^{\vphantom\inv}_1 = i_c^d$ for some integer $d \ne 0$. If there exists a pair $\Phi_1 \ne \Phi_2 \in P(\phi)$ such that the above equivalent conditions hold, then $\mu$ is said to be a \emph{twistor} for~$\phi$, and the number of distinct elements of the set $P(\phi) \intersect \Stab(c)$ is called the \emph{multiplicity} of $\mu$ as a twistor for $\phi$; these properties are independent of the choice of $c \in F_n$ representing $\mu$. 

The number of twistors of $\phi$ and the multiplicity of each twistor is finite, as follows. First, for any \ct\ $f \from G \to G$ representing $\phi$, for $\mu$ to be a twistor it is equivalent that some \neg\ linear edge $E$ twists around $\mu$, meaning that $E$ has twist path $w$ such that either $w$ or $w^\inv$ represents $\mu$. Furthermore, the multiplicity of $\mu$ is one more than the number of linear edges that twist around~$\mu$ \cite[Lemma 4.40]{\recognitionTag}.

\subsubsection{Rotationless abelian subgroups}
\label{SectionRotationlessAbelian}
\paragraph{Principal sets for rotationless abelian subgroups.} \quad\cite[Definition 3.9, Corollary 2.13]{FeighnHandel:recognition}. 
An abelian subgroup $A \subgroup \Out(F_n)$ is \emph{rotationless} if each of its elements is rotationless, equivalently $A$ has a rotationless generating set. Every abelian subgroup contains a rotationless subgroup of index bounded uniformly by a constant depending only on~$n$, namely the subgroup of $k^{\text{th}}$ powers where $k \ge 1$ is an integer such that the $k^{\text{th}}$ power of every element of $\Out(F_n)$ is rotationless \cite[Lemma 4.42]{\recognitionTag}.

Assuming $A \subgroup \Out(F_n)$ is rotationless, a subset $\X \subset \partial F_n$ with $\ge 3$ elements is a \emph{principal set} for $A$ if each $\psi \in A$ is represented by $\Psi \in \Aut(F_n)$ such that $\X \subset \Fix(\wh \Psi)$ (which determines $\Psi$ uniquely amongst representatives of $\Psi$) and such that $\Psi \in P^\pm(\psi)$. When the principal set $\X$ is fixed, the map $\psi \mapsto \Psi$ defines a homomorphism $s \from A \to \Aut(F_n)$ that is a section of the canonical map $\Out(F_n) \to \Aut(F_n)$ over the subgroup~$A$. Also, the set
$$\Y = \cap_{\psi \in A}\Fix(\wh{s(\psi)})
$$ 
is the unique principal set which is maximal up to inclusion subject to the property that $\Y$ contains $\X$; this set $\Y$ defines the same section $s \from A \to \Aut(F_n)$ as $\X$ \cite[Remark 3.10]{FeighnHandel:recognition}.



\paragraph{Comparison homomorphisms of rotationless abelian subgroups.} \cite[Definition 4.1 and Lemma 4.3]{\abelianTag}. \\ Consider a rotationless abelian subgroup~$A$. 
Consider also two principal sets $\X_1,\X_2$ for $A$ that define lifts $s_1,s_2 \from A \to \Aut(F_n)$. Let $\Y_1,\Y_2$ be the maximal principal sets containing $\X_1,\X_2$ respectively, and so $s_1,s_2$ are also the lifts defined by $\Y_1,\Y_2$. Suppose that $s_1 \ne s_2 \from A \to \Aut(F_n)$ (this if and only if $\Y_1 \ne \Y_2$), and suppose also that $\X_1 \cap \X_2 \ne \emptyset$ and hence $\Y_1 \cap \Y_2 \ne \emptyset$. It follows that the set ${\cal Y}_1 \cap {\cal Y}_2$ is fixed by distinct automorphisms representing the same element of $ A$, and so $\Y_1 \intersect \Y_2 = T^\pm_c$ for some nontrivial root free $c \in F_n$. In this situation there is an associated \emph{comparison homomorphism} $\omega \from A \to \Z$ which is characterized by the equation
$$s_2(\psi) = i_c^{\omega(\psi)} s_1(\psi) \quad\text{for all $\psi \in A$.}
$$
The number of distinct comparison homomorphisms $A \to \Z$ is finite.

\paragraph{The coordinate homomorphism.} \cite[Lemma 4.4, Definition 4.5, Lemma~4.6]{\abelianTag}. \quad Every abelian subgroup $A \subgroup \Out(F_n)$ is a finite lamination group, that is, the set $\L(A) = \union_{\phi \in A} \L(\phi)$ is finite. If in addition $A$ is rotationless then each $\Lambda \in \L(A)$ is $\psi$-invariant for all $\psi \in A$, and so 
$$A \subgroup \bigcap_{\Lambda \in \L(A)} \Stab(\Lambda)
$$
We thus obtain a finite collection of expansion factor homomorphisms defined on $A$, namely 
$$\PF_\Lambda \from A \to \Z, \quad \Lambda \in \L(A)
$$
Choosing an enumeration $\Omega_1,\ldots,\Omega_N$ of the expansion homomorphisms and the comparison homomorphisms, we obtain the \emph{coordinate homomorphism} 
$$\Omega \from A \to \Z^N
$$
which is injective.

\subsection{Disintegration subgroups}
\label{SectionDisintegration}

In this section we fix a rotationless $\phi \in \Out(F_n)$ and a \ct\ representative $f \from G \to G$ with corresponding $f$-invariant filtration $G_1 \subset \cdots \subset G_u = G$ (that is all we need from Notations~\ref{NotationsFlaring} for this section).

Using the \ct\ structure of $f$ (see Section~\ref{SectionMoreCTs}) one builds up to the definition of the \emph{disintegration subgroup} associated to~$f$, a certain abelian subgroup $\D(f) \subgroup \Out(F_n)$ constructed in \cite[Section 6]{\abelianTag} by an explicit combinatorial procedure that we review here, using which one also obtains a description of the coordinate homomorphism $\Omega \from \D(f) \to \Z^N$ (see Section~\ref{SectionRotationlessAbelian}). Noting that $\D(f)$ depends on the choice of $f$, all definitions in this section depend on~$f$.



\subsubsection{QE-paths and QE-splittings} 

We describe structures associated to the collection of linear edges of~$G$, augmenting the structures described under the heading ``Properties of \neg-linear strata'' in Section~\ref{SectionMoreCTs}.

\smallskip
\noindent
\textbf{\cite[Def.~4.7]{FeighnHandel:recognition}} 
\quad Two linear edges $E_i,E_j \subset G$ with a common twist path $w$ are said to be \emph{linearly equivalent}, abbreviated to \emph{LIN-equivalent}, and the associated LIN-equivalence class is denoted $\lin_w$. Recall that distinct elements $E_i \ne E_j \in \lin_w$ have distinct twist coefficients $d_i \ne d_j$. 

\smallskip
\noindent
\textbf{\cite[Lemma~6.1 and Def.~6.2]{FeighnHandel:abelian}} \quad Given a twist path $w$ and distinct linear edges $E_i \ne E_j \in \lin_w$, a path of the form $E_i w^p \overline E_j$ is called a \emph{quasi-exceptional path} or \emph{QE-path} (it is an exceptional path if and only if the twist coefficients $d_i,d_j$ have the same sign). For every completely split path~$\sigma$, any QE-subpath $E_i w^p \overline E_j$ of $\sigma$ is a concatenation of terms of the complete splitting of $\sigma$: either $d_i,d_j$ have the same sign and $E_i w^p \overline E_j$ is an exceptional path and hence a single term; or the terms consist of the initial $E_i$, the terms of the complete splitting of $w^p$, and the terminal~$E_j$. No two QE-subpaths can overlap in an edge, and so there is a uniquely defined \emph{QE-splitting} of $\sigma$ which is obtained from the complete splitting by conglomerating each QE-subpath of $\sigma$ into a single term. 

Consider a twist path $w$. For each $E_i,E_j \in \lin_w$, the associated \emph{quasi-exceptional family} is the set of paths
$$E_i w^* \overline E_j = \{E_i w^p \overline E_j \suchthat p \in \Z\}
$$
Also, associated to $w$ itself is its \emph{linear family}, the set
$$\lin_w \,\, \union \!\!\!\!\!\! \bigcup_{\quad E_i \ne E_j \in \lin_w} \!\!\!\!\!\! E_i w^* \overline E_j
$$ 
Every quasi-exceptional path belongs to a unique quasi-exceptional family, and every \neg\ linear edge or quasi-exceptional path belongs to a unique linear family. 

\subsubsection{Almost invariant subgraphs} 

The intuition behind the ``almost invariant subgraphs'' of $G$ is that they are a partition of the collection of non-fixed strata determined by the following requirement: if one perturbs $f \from G \to G$ by letting its restrictions $f \restrict H_i$ to non-fixed strata be replaced by iterates $f^{a_i} \restrict H_i$ with varying integer exponents $a_i \ge 0$, and if one wishes to do this perturbation so that the resulting outer automorphisms commute with the outer automorphism represented by~$f$, then the exponents $a_i$ should be constant on the edges contained in each almost invariant subgraph.

\smallskip\noindent
\textbf{\cite[Def.~6.3]{FeighnHandel:abelian}} \quad
Define an equivalence relation on the non-fixed irreducible strata $\{H_i\}$ of $G$ as follows. Given $H_i$, a path $\mu$ is a \emph{short path for $H_i$} if $\mu$ an edge in $H_i$, or if $H_i$ is \eg\ and $\mu$ is a taken connecting path in a zero stratum enveloped by $H_i$. Define a relation amongst the non-fixed irreducible strata denoted $H_i \dg H_j$, meaning that there exists a short path $\mu$ for $H_i$ such that some term of the QE-splitting of $f_\#(\mu)$ is an edge of $H_j$; note that if $H_i \dg H_j \dg \cdots \dg H_k$ then $i \ge j \ge \cdots \ge k$. Let $B$ be the directed graph whose vertex set is the set of non-fixed irreducible strata, with a directed edge for each relation $H_i \dg H_j$. Let $B_1,\ldots, B_S$ denote the connected components of the graph~$B$. For each $s = 1,\ldots,S$ define the \emph{almost invariant subgraph} $X_s \subset G$ to be the union of the strata $H_i$ comprising the vertices of $B_s$, together with any zero stratum enveloped by one of these $H_i$. Note that the set of almost invariant subgraphs $\{X_1,\ldots,X_S\}$ partitions the set of the nonfixed strata.

\subsubsection{Admissible $S$-tuples; quasi-exceptional families} 
\label{SectionAdmissible}

\newcommand\aTup{{\bf{a}}}
\newcommand\bTup{{\bf{b}}}

\noindent
\textbf{\cite[Defn.~6.6, Lemma~6.7, Defn.~6.8]{FeighnHandel:abelian}} \quad For each $S$-tuple $\aTup =(a_1,\ldots, a_S)$ of non-negative integers define $f^\aTup : G \to G$ on each edge $E \subset G$ as follows:\footnote{In \abelian\ the notation $f_\aTup$ was used for what we here denote $f^\aTup$.} if $E \subset X_s$ for some $s$ then $f^\aTup(E) = f^{a_s}_\#(E)$; otherwise $E$ is fixed by $f$ and then $f^\aTup(E) = E$. Each such $f^\aTup$ is a homotopy equivalence representing an outer automorphism denoted~$\phi^\aTup$. By construction $f^\aTup$ preserves the given filtration~$G_1 \subset \cdots \subset G_u=G$.

One would like to arrange (among other things) that $(f^\aTup)^k_\#(E) = f_\#^{a_s k}(E)$ for all $k \ge 1$, all $s=1,\ldots,S$, and all edges $E \subset X_s$. This need not hold for general $S$-tuples, but it does hold when a certain relation amongst linear edges and QE-splitting terms is satisfied. 

Given an almost invariant subgraph~$X_r$ and linear edges $E_i$, $E_j$, we say that the triple $(X_r,E_i,E_j)$ is \emph{quasi-twist related} if there exists a stratum $H_k \subset X_r$, a short path $\mu$ for $H_k$, a twist path $w$, and an integer $p \ge 1$ such that $E_i, E_j \in \lin_w$, and such that $E_i w^p \overline E_j$ is a term in the QE-splitting of~$f_\#(\mu)$.

We say that an $S$-tuple $\aTup$ is \emph{admissible} if for all quasi-twist related triples $(X_r,E_i,E_j)$, letting $d_i,d_j$ be the twist coefficients of $E_i,E_j$ respectively, and letting $X_s,X_t$ be the almost invariant subgraphs containing $E_i,E_j$ respectively, the following ``twisting relation'' holds:
$$a_r(d_i - d_j) = a_s d_i - a_t d_j
$$

\smallskip
\noindent
\textbf{\cite[Notation~6.12]{\abelianTag}} \quad Consider an almost invariant graph $X_r$. For each pair of linear edges $E_i,E_j$ such that the triple $(X_r,E_i,E_j)$ is quasi-twist related, the associated \emph{quasi-exceptional family} is defined to be the set of all paths of the form $E_i w^* \overline E_j$. We let $\cQ_r$ denote the set of all quasi-exceptional families associated to quasi-twist related triples $(X_r,E_i,E_j)$. 

\smallskip
\noindent
\textbf{QE-equivalence of linear edges.} We say that a pair of linear edges $E_i,E_j$ is \emph{QE-related} if there exists an almost invariant subgraph $X_r$ such that $(X_r,E_i,E_j)$ is quasi-twist related. Equivalently $E_i,E_j$ are QE-related if and only if they are in the same linear family and, letting $w$ be the unique twist path for which $E_i, E_j \in \lin_w$, there exists~$r$ such that the family $E_i w^* \overline E_j$ is in the set~$\cQ_r$. 

The equivalence relation on linear edges generated by being QE-related is called \emph{QE-equivalence} and is written $\sim_{QE}$. Note that the QE-equivalence relation amongst linear edges is a refinement of the LIN-equivalence relation. 

Note that if an exceptional path $E_i w^p \overline E_j$ occurs as a term in the complete splitting of some iterate of some short path in $X_r$ then the quasi-exceptional family $E_i w^* \overline E_j$ is an element of $\cQ_r$ and the linear edges $E_i$ and $E_j$ are QE-related.

\subsubsection{$X_s$ paths.} 
\label{SectionXsPaths}
\textbf{\cite[Notation 6.12, Corollary 6.14]{FeighnHandel:abelian}} \quad
For each almost invariant subgraph $X_s$, we use the terminology \emph{$X_s$-paths} to refer to the subpaths of $G$ which form the elements of the set $\mathcal{P}_s$ that is defined in \cite[Notation 6.12]{FeighnHandel:abelian} --- namely, the completely split paths $\gamma$ such that each term of the QE-splitting of $\gamma$ is one of the following: 
\begin{enumerate}
\item a Nielsen path; or
\item a short path in a stratum contained in~$X_s$, for example any edge in $H^z_i$ for any \eg\ stratum $H_i \subset X_s$; or
\item a quasi-exceptional path in the family~$\cQ_s$. 
\end{enumerate}
%
%
Furthermore, for any admissible $S$-tuple $\aTup$, any almost invariant subgraph $X_s$, and any $X_s$ path~$\gamma$, we have the following:
\begin{enumeratecontinue}
\item\label{ItemXsPathIterates}
Each iterate $f^k_\#(\gamma)$ is also an $X_s$ path;
\item\label{ItemXsPathImages}
$f^\aTup_\#(\gamma) = f^{a_s}_\#(\gamma)$.
\end{enumeratecontinue}

\subsubsection{The disintegration subgroup $\D(f)$
}  
\label{SectionDisintSubgpDef}

Here we recall the definition of the disintegration subgroup $\D(f) \subgroup \Out(F_n)$. We also recall the ``admissible semigroup'' $\S(f)$, 
which we will use as an aid to understanding properties of $\D(f)$.

\paragraph{The subgroup $\D(f)$.} 
\textbf{\cite[Defn.~6.10, Cor.~6.16, Lem.~6.18, Cor.~3.13]{\abelianTag}} \quad This is the subgroup of $\Out(F_n)$ generated by the set of elements 
$$\{\phi^\aTup \suchthat \,\text{$\aTup$ is an admissible $S$-tuple}\}
$$
The subgroup $\D(f)$ is abelian and rotationless. 
The dependence of $\D(f)$ on $f$ rather than just $\phi$ was suppressed in \cite[Definition 6.10]{FeighnHandel:abelian} where the notation $\D(\phi)$ was used; see \cite[Example~6.11]{FeighnHandel:abelian}.

\paragraph{The admissible semigroup $\S(f)$.} \textbf{\cite[Corollary 7.6 and its proof]{FeighnHandel:abelian}} \break Let~$\S(f) \subset \Z^S$ denote the set of admissible $S$-tuples, which forms a sub-semigroup of~$\Z^S$. Let $\L(f) \subgroup \Z^S$ be the subgroup generated by~$\S(f)$. The map $\S(f) \mapsto \D(f)$ defined by $\aTup \mapsto \phi^\aTup$ is an injective semigroup homomorphism, and it extends to an isomorphism $\L(f) \mapsto \D(f)$. Every element of $\L(f)$ can be written as the difference $\aTup - \bTup$ of two elements $\aTup,\bTup \in \S(f)$, and so every $\psi \in \D(f)$ can be written in the form $\psi = (\phi^\bTup)^\inv \phi^\aTup$ for some $\aTup,\bTup \in \S(f)$.

We record here a simple consequence of these definitions for later use:

\begin{fact}\label{FactSharpHomomorphism}
The function which assigns to each $\aTup \in \S(f)$ the map $f^\aTup \from G \to G$ satisfies the following ``homomorphic'' property: 
$$f^\aTup_\# (f^\bTup_\#(E)) = f^{\aTup+\bTup}_\#(E) \quad \text{for each $\aTup,\bTup \in \S(f)$ and each edge $E \subset G$}
$$
\end{fact}

\begin{proof}
If $E$ is a fixed edge of $f$ then both sides equal~$E$. Otherwise $E$ is an edge in some almost invariant graph $X_s$, hence $E$ is an $X_s$-path, hence $f^\bTup_\#(E) = (f^{b_s})_\#(E)$ is an $X_s$ path, hence both sides equal $(f^{a_s+b_s})_\#(E)$.
\end{proof}

\bigskip

We repeat here the theorem cited in Section~\ref{SectionMultiEdgeIntro}:

\begin{disintegrationtheorem}[\protect{\cite[Theorem 7.2]{\abelianTag}}]
For every rotationless abelian subgroup $\cH \subgroup \Out(F_n)$ there exists $\phi \in \cH$ such that for every \ct\ $f \from G \to G$ representing $\phi$, the intersection $\cH \intersect \D(f)$ has finite index in $\cH$.
\end{disintegrationtheorem}

\noindent\emph{Remark.} In this statement of the Disintegration Theorem we have made explicit the fact that \cite[Theorem 7.2]{\abelianTag} holds for any choice of \ct\ representing $\phi$. This was already implicit in the notation $\D(\phi)$ used for disintegration groups in \cite[Definition 6.10]{\abelianTag} where any representative \ct\ is allowed.

\subsubsection{The coordinate homomorphism of a disintegration group $\D(f)$} 
\label{SectionCoordHomDf}
The disintegration group $\D(f)$, being rotationless and abelian, has an injective coordinate homomorphism $\Omega \from \D(f) \mapsto \Z^N$ as defined at the end of Section~\ref{SectionRotationlessAbelian}. The individual coordinate functions of $\Omega$ are the comparison homomorphisms  and the expansion factor homomorphisms of the abelian group $\D(f)$. 

We review here how one may use the structure of the \ct\ $f$ to sift through the coordinate functions of $\Omega$ --- keeping a subset of the comparison homomorphisms, and keeping a normalized version of each expansion factor homomorphism --- to obtain a homomorphism denoted $\Omega^f \from \D(f) \to \Z^\I$ which is still injective. We will \emph{not} normalize comparison homomorphisms, because it would screw up the ``twisting relation'' described in Section~\ref{SectionAdmissible}.

\medskip
\noindent\textbf{\cite[Lemmas~7.4 and 7.5 and preceding paragraphs]{\abelianTag}.} \, \break
Let $\I = \{i \suchthat \, \text{$H_i$ is either \neg\ linear or \eg}\}$. The homomorphism $\Omega^f \from \D(f) \to \Z^\I$ has coordinate functions denoted $\omega_i \from \D(f) \to \Z$ for each $i \in \I$. For each $i \in \I$, the function $\omega_i$ is characterized in one of two ways, depending on whether $H_i$  is \neg\ linear or \eg. Let $X_s$ be the almost invariant subgraph containing~$H_i$. 
\begin{description}
\item[\neg\ linear coordinate homomorphism (a component of $\Omega$ and of $\Omega^f$):] \quad\break If $H_i = E_i$ is \neg\ linear, with twist path $w$ and twist coefficient $d_i$ satisfying $f(E_i) = E_i w^{d_i}$, then for each admissible $S$-tuple~$\aTup$ we set $\omega_i(\phi^\aTup) = a_s d_i$.
\end{description}
Consider also the difference homomorphisms $\omega_{i,j} \from \D(f) \to \Z$, one for each LIN-equivalent pair of linear edges $E_i,E_j$, defined by:
$$\omega_{i,j} = \omega_i - \omega_j
$$
As shown in \cite[Section 7]{\abelianTag} just following Theorem 7.2, each of these difference homomorphisms is a comparison homomorphism on $\D(f)$ (Section~\ref{SectionRotationlessAbelian}), and furthermore the set of comparison homomorphisms on $\D(f)$ is exactly the set of functions consisting of the \neg\ linear coordinate homomorphisms $\omega_i$, their additive inverses $-\omega_i$, and their differences $\omega_{i,j} = \omega_i - \omega_j$ for LIN-equivalent linear edges $E_i,E_j$.

\begin{description}
\item[Unnormalized \eg\ coordinate homomorphism (a component of $\Omega$):] \quad\break If $H_i$ is \eg\ with associated attracting lamination $\Lambda_i \in \L(\phi)$ then for each admissible $S$-tuple $\aTup$ we set $\omega_i(\phi^\aTup) = \PF_{\Lambda_i} (\phi^\aTup)$. 
\end{description}
As alluded to in Section~\ref{SectionPrincipal}, this definition makes sense because $\D(f) \subgroup \Stab(\Lambda_i)$.

\smallskip

In both the \neg\ linear case and the \eg\ case, for each admissible $S$-tuple $\aTup$ the following equation holds  \cite[Lemma 7.5]{\abelianTag}:
$$\omega_i(\phi^{\aTup}) = a_s \omega_i(\phi)
$$

Noting that the constant sequence $\aTup = (1,1,\ldots,1)$ is admissible and satisfies the equation $\phi^{\aTup} = \phi$, it follows that the subgroup $\image(\omega_i) \subgroup \Z$ is generated by the integer $\omega_i(\phi)$. If $H_i$ is \eg\ then we normalize the function $\omega_i$ by dividing it by the integer $\omega_i(\phi)$; the preceding equation still holds. After this normalization we have:
\begin{description}
\item[Normalized \eg\ coordinate homomorphism (a component of $\Omega^f$):] \quad\break If $H_i$ is \eg\ then, with notation as above, we have
$$\omega_i(\phi^\aTup) = a_s
$$
and in particular $\omega_i(\phi) = 1$. It follows that if $\lambda_i$ is the Perron-Frobenius eigenvalue of the transition matrix of $f$ on $H_i$, then the expansion factor of $\phi^\aTup$ on the lamination $\Lambda_i$ is equal to $\lambda_i^{\omega_i(\phi^\aTup)}  = \lambda_i^{a_s}$.
\end{description}
%
%
%
%
%



For later reference we summarize the properties of $\Omega^f \from \D(f) \to \Z^\I$ as follows:

\smallskip\noindent\textbf{$\bullet$} The homomorphism $\Omega^f \from \D(f) \to \Z^\I$ is injective.

\smallskip\noindent\textbf{$\bullet$}  For each $i \in \I$ and each admissible $S$-tuple~$\aTup$ we have $\omega_i(\phi^\aTup) = a_s \omega_i(\phi)$.

\smallskip\noindent\textbf{$\bullet$} The coordinate function $\omega_i$ of $\Omega^f$ associated to an \eg\ stratum $H_i$ is the normalized version of $\PF_{\Lambda_i}$ satisfying $\omega_i(\phi)=1$.

\section{A train track semigroup action}
\label{SectionTTSemiAction}

Throughout this section we continue to adopt Notations~\ref{NotationsFlaring} regarding a rotationless $\phi \in \Out(F_n)$ with a  \ct\ representative $f \from G \to G$ whose top stratum $H_u$ is \eg\ with associated attracting lamination $\Lambda$. 

Consider the disintegration group $\D=\D(f)$, reviewed in Section~\ref{SectionDisintegration}. Recall that each element of $\D$ stabilizes $\Lambda$. We shall focus on the subsemigroup $\D_+ \subgroup \D$ consisting of all $\psi \in \D$ for which the value of top coordinate homomorphism $\omega(\psi)$ is non-negative. The value of $\omega(\psi)$ is a logarithm of the asymptotic stretch factor of $\psi$ on $\Lambda$, hence for $\omega(\psi)$ to be non-negative means that either $\psi$ does not stretch $\Lambda$ or $\Lambda$ is an attracting lamination of $\psi$.

Using the theory of disintegration groups, for each $\psi \in \D_+$ we construct a topological representative $f^\psi \from G \to G$ whose action on the edges of $H_u$ agrees with the action of the appropriate iterate of the \ct\ $f \from G \to G$, namely~$f^{\omega(\psi)}_\#$. Then, by lifting these topological representatives to the universal cover~$\wt G$, projecting to the tree~$T$, and extending to the coned off graph~$T^*$, we obtain semigroup actions of $\D_+$ on $T$ and on $T^*$, detailed properties of which are described in Section~\ref{SectionActionOnT}. An important feature of the action $\D_+ \act T$ is that it is \emph{not} an isometric action, in fact the action of each $\Psi \in \D_+$ will stretch lengths in $T$ by a uniform amount depending on the appropriate value of the coordinate homomorphism~$\omega$ (see Section~\ref{SectionActionOnT}~\pref{Item_f_PsiEdges}). On the other hand, by restricting the semigroup actions $\D_+ \act T,T^*$ to the subgroup $\D_0$ we obtain true isometric actions $\D_0 \act T,T^*$, some properties of which are studied in Section~\ref{SectionTStarDynamics}. 

The actions constructed in this section will be the basis for the construction in Section~\ref{SectionSuspension} of an isometric group action of $\D$ on the hyperbolic suspension complex~$\S$.

\medskip

For use throughout Sections~\ref{SectionTTSemiAction} and Section~\ref{SectionSuspension} we establish some further notations.


\begin{notations} 
\label{NotationsActionOnT}
Given a rotationless $\phi \in \Out(F_n)$ with \ct\ representative $f \from G \to G$ having associated filtration $G_1 \subset \cdots \subset G_u$ with its top stratum $H_u$ being \eg, and with all associated objects and notations as described in Notations~\ref{NotationsFlaring}, we have the following additional objects and notations:
\begin{enumerate}
\item \label{ItemHzu}
[Section~\ref{SectionMoreCTs}, ``Properties of zero strata''] \,
$H^z_u$ is the union of $H_u$ with all zero strata enveloped by~$H_u$.
For some $t \le u-1$ this set of zero strata has the form $\{H_i \suchthat t+1 \le i \le u-1\}$ (which is empty if $t=u-1$), where 
\begin{enumerate}
\item $H_t$ is the highest irreducible stratum below $H_u$, 
\item $G_t$ is the union of noncontractible components of $G_{u-1}$, the contractible components being the zero strata $H_i$, $t+1 \le i \le u-1$. 
\end{enumerate}
\item \label{ItemDefinitionOfT}
[Definitions~\ref{DefOfT}, \ref{DefNielsenSet}] \, 
$F_n \act T$ denotes the free splitting corresponding to the marked graph pair $(G,G_{u-1})$, with associated Nielsen set $\{N_j\}_{j \in J}$, each $N_j$ having basepoint set $Z_j \subset N_j$. Also, $F_n \act T^*$ denotes the action obtained from $F_n \act T$ by coning off each basepoint set $Z_j$ to a cone point $P_j$ with edges $\overline{P_n Q}$~($Q \in Z_j$).
\item \label{ItemLiftChoices}
[Definition~\ref{DefLiftingTTMap}] \,
Associated to each automorphism $\Phi \in \Aut(F_n)$ representing $\phi$ are unique $\Phi$-twisted equivariant maps as follows: 
\begin{enumerate}
\item $\ti f^\Phi  \from \wt G \to \wt G$, a lift of $f$; 
\item\label{ItemFSubPhi} $f^\Phi_T \from T \to T$, induced by $\ti f^\Phi$ with respect to the collapse map $\wt G \mapsto T$.
\item $f_{T^*} \from T^* \to T^*$, an extension of $f^\Phi_T$ that permutes cone points and cone edges.
\end{enumerate}
The maps $f^\Phi_T$, $f^\Phi_{T^*}$ satisfy the ``Stretch properties'' recorded in Section~\ref{SectionTStar}. 
\item \label{ItemDisintGroup}
The disintegration group of $f$ is denoted $\D = \D(f)$. Its full pre-image in $\Aut(F_n)$ is denoted $\wh\D$ and is called the \emph{extended disintegration group}. 
\item Associated to the top \eg\ stratum $H_u$ we have the following objects: 
\begin{enumerate}
\item The almost invariant subgraph $X_s \subset G$ containing $H^z_u$; 
\item\label{ItemCoordHomDf}
 (Section~\ref{SectionCoordHomDf}) The coordinate homomorphism $\omega \from \D \to \Z$ associated to~$X_s$, a scaled copy of the expansion factor homomorphism $\PF_\Lambda$ that is normalized so that
$$\omega(\phi)=1
$$
In particular, $\omega$ is surjective.
\item The \emph{lifted coordinate homomorphism} $\wh\omega \from \wh\D \to \Z$, obtained by pre-composing $\omega$ with the projection map $\wh\D \mapsto\D$.
\item The kernels of these homomorphisms denoted
$$\D_0 = \ker(\omega) \qquad \wh\D_0 = \ker(\wh\omega)
$$
\end{enumerate}
We thus have a commutative diagram with short exact rows:
$$\xymatrix{
1 \ar[r] & \ker(\wh\omega) = \wh\D_0 \ar[r] 
   & \wh\D \ar[r]^{\wh\omega} \ar@{=}[d] 
   & \Z \ar[r] & 1 \\
1 \ar[r] & F_n \approx \Inn(F_n) \ar[r] \ar@{=}[d] \ar[u]_{\subset}
           & \wh\D \ar[r] \ar[d]^{\subset} 
           & \D \ar[r] \ar[d]^{\subset} \ar[u]_{\omega}
           &1 \\
1 \ar[r] & F_n \approx \Inn(F_n) \ar[r] & \Aut(F_n) \ar[r] & \Out(F_n) \ar[r] & 1 \\
}$$
\item\label{ItemSubsemigroups}
Letting $[0,\infty) = \{0,1,2,\ldots\}$ we have subsemigroups $D_+$ and $\wh\D_+$ and inclusions as follows:
$$\D_0 = \omega^\inv(0) \,\, \subgroup \,\, \D_+ = \omega^\inv[0,\infty) \,\, \subgroup \,\, \D 
$$
$$\wh\D_0 = \wh\omega^\inv(0) \,\,\subgroup\,\, \wh\D_+ = \wh\omega^\inv[0,\infty) \,\,\subgroup\,\, \wh\D
$$
\end{enumerate}
\end{notations}

%
\subsection{A ``homotopy semigroup action'' of $\D_+$ on $G$.} 
\label{SectionHomotopyActionOnG}

To prepare for the construction of the semigroup action $\wh\D_+ \act T$, in this section we work downstairs in~$G$ and construct a ``homotopy semigroup action'' of~$\D_+$ on~$G$. What this means will be clear from the construction, but the intuitive idea is we associate a topological representative $f^\psi \from G \to G$ to each $\psi \in \D_+$ so that although the action equation $f^\psi \circ f^{\psi'} = f^{\psi\psi'}$ does not hold exactly, it does hold ``up to homotopy relative to~$G_t$''. The values of $f^\psi$ on edges of $H^z_n$ are determined by appropriate iterates of $f$ itself, and the values on $G_t$ are determined, up to homotopy, by the ``graph homotopy principle'' of Section~\ref{SectionGraphHomotopy}.

\renewcommand\Vert{\text{Vert}}

Letting $\Vert(G)$ denote the set of vertices of $G$, we define subsets
$$P \subset Fix \subset V \subset \Vert(G)
$$
as follows. First, $V = \Vert(G) \intersect H_u = \Vert(G) \intersect H^z_u$, the latter equation holding because each edge $E \subset H^z_u \setminus H_u$ has both endpoints in $H_u$ \cite[Definition 4.10 (Zero Strata)]{\recognitionTag}. Next, $P = H^z_u \intersect G_t = H_u \intersect G_t$; note that $P \subset \Vert(G)$ because $H^z_u$ and $G_t$ are subgraphs having no edges in common; also $P$ is the frontier both of $H^z_u$ and of $G_t$ in $G = H^z_u \union G_t$. Finally, $Fix$ denotes the set of points in $V$ fixed by $f$, and the inclusion $P \subset Fix$ follows from \cite[Remark 4.9]{\recognitionTag}. 

\smallskip

Here is a summary of the construction of the ``homotopy semigroup action'': 
%
%
%
\begin{enumerate}
\item\label{ItemOnHzN}
For each $\psi \in \D_+$ we define a topological representative $f^\psi \from G \to G$ such that the following hold:
\begin{enumerate}
\item\label{ItemOnIdentity}
We have $f^{\Id_{\D_+}} = \Id_G$ (where $\Id_{\D_+} \in \D_+$ denotes the identity group element, and $\Id_G \from G \to G$ the identity map).
\item\label{ItemOnGNMinusOne} $f^\psi(G_{u-1}) \subset G_{u-1}$ and $f^\psi(G_t) \subset G_t$.
\item\label{ItemOnE}
For each edge $E \subset H^z_u$ we have $f^\psi_{\vphantom\#}(E) = f^{\omega(\psi)}_\#(E)$. In particular, if $\psi \in \D_0$ then $f^\psi \restrict H^z_u$ is the identity.
\item\label{ItemOnFrontier} $f^\psi \restrict V = f^{\omega(\psi)} \restrict V$. In particular $f^\psi(V) \subset V$ and $f^\psi$ fixes each point of $Fix$ and of its subset $P$.
\item\label{ItemOnNielsen} If a height~$u$ indivisible Nielsen path $\rho$ exists then $f^\psi$ fixes each endpoint of $\rho$ and $f^\psi_\#(\rho)=\rho$.
\end{enumerate}
\item\label{ItemSemiUnique}
For each $\psi,\psi' \in \D_+$ we define a homotopy $h^{\psi,\psi'} \from G \times [0,1] \to G$ between $f^\psi \composed f^{\psi'}$ and $f^{\psi\psi'}$ such that the following hold:
\begin{enumerate}
\item\label{ItemPsiPsiPrimeEEquation}
For each edge $E \subset H_u$ the homotopy $h^{\psi,\psi'}$ straightens the edge path \break $f^\psi \composed f^{\psi'} \restrict E$ relative to its endpoints to form the path $f^{\psi\psi'} \restrict E$ by cancelling only edges of $G_{u-1}$, without ever cancelling any edges of $H_u$.
\item\label{ItemVNStationary}
$h^{\psi,\psi'}$ is stationary on $V$. 
\item\label{ItemHomotopyPreservesGNMinusOne}
$h^{\psi,\psi'}$ preserves $G_t$.
\end{enumerate}
\end{enumerate}
Recall that for a homotopy $h \from A \times [0,1] \to A$ to ``preserve'' a subset $B \subset A$ means that $h(B \times [0,1]) \subset B$, and for $h$ to be ``stationary'' on $B$ means that $h(x \times [0,1])$ is constant for each $x \in B$.

First we will construct the maps $f^\psi$ and then the homotopies $h^{\psi,\psi'}$, along the way proving the various requirements of~\pref{ItemOnHzN} and~\pref{ItemSemiUnique}.

\smallskip\textbf{Constructing $f^\psi$.} We being with the construction of $f^\psi_t = f^\psi \restrict G_t \from G_t \to G_t$ by applying the graph homotopy principle Lemma~\ref{LemmaGraphHomotopy}.
Recall from Section~\ref{SectionDisintSubgpDef} the abelian semigroup of admissible $S$-tuples $\S(f)$. Recall also the injection $\S(f) \approx \L(f) \subgroup \D$ defined by $\aTup \mapsto \phi^\aTup$ with topological representative $f^\aTup \from G \to G$; this injection is a semigroup homomorphism whose image $\L(f)$ generates~$\D$. 
Since the restricted map $f^\aTup_t = f^\aTup \restrict G_t \from G_t \to G_t$ is a homotopy equivalence that fixes each point of $P$, by Lemma~\ref{LemmaGraphHomotopy}~\pref{ItemGPHomInvExists} the map of pairs $f^\aTup_t \from (G_t,P) \to (G_t,P)$ is a homotopy equivalence in the category of pairs, and hence its homotopy class rel~$P$ is an element $[f^\aTup_t]$ of $\Aut^\gp_0(G_t,P)$, the pure automorphism group of $(G_t,P)$ in the graph point category (Lemma~\ref{LemmaGraphHomotopy}~\pref{ItemGPAut}). By Fact~\ref{FactSharpHomomorphism}, for each $\aTup,\bTup \in \S(f)$ the maps $f^\aTup_t \circ f^\bTup_t$ and $f^{\aTup+\bTup}_t \from (G_t,P) \to (G_t,P)$ are homotopic rel~$P$, and so the map $\A \from \L(f) \mapsto \Aut^\gp_0(G_t,P)$ defined by $\A(\phi^\aTup) = [f^\aTup_t]$ is a semigroup homomorphism. Since the commutative group $\D$ is generated by its subsemigroup $\L(f)$, a simple semigroup argument shows that $\A$ extends uniquely to a group homomorphism $\A \from \D \to \Aut^\gp_0(G_t,P)$. For each $\psi \in \D$ choose $f^\psi_t \from (G_t,P) \to (G_t,P)$ to be a representative of $\A(\psi)$; if $\psi = \phi^\aTup$ is already in $\L(f)$ then we choose $f^\psi_t=f^\aTup_t$, and so if $\psi$ is the identity then $f^\psi_t$ is the identity. Notice that no straightening is carried out when $f^\psi_t$ is applied, and so there is no need for the ``$\#$'' operator in the definition of~$f^\psi_t$.

For later use, for each $\aTup \in \S(f)$ we denote $\bar f^\aTup_t = f^{(\phi^\aTup)^\inv}_t \from (G_t,P) \to (G_t,P)$, which is the homotopy inverse of $f^\aTup_t$ that represents $\A((\phi^\aTup)^\inv) = \A(\phi^\aTup)^\inv \in \Aut^\gp_0(G_t,P)$.

For each $\psi \in \D_+$ we may now define $f^\psi \from G \to G$ as follows: 
\begin{align*}
f^\psi(E) &= f^{\omega(\psi)}_\#(E) \quad\text{for each $E \subset H^z_u$, and} \\
f^\psi \restrict G_t & = f^\psi_t
\end{align*}
If $\psi \in \D$ is the identity then clearly $f^\psi \from G \to G$ is the identity, verifying~\pref{ItemOnIdentity}. By construction $f^\psi$ satisfies~\pref{ItemOnGNMinusOne} and~\pref{ItemOnE}. For item~\pref{ItemOnFrontier}, note first that for each $x \in V$ there exists an oriented edge $E \subset H_u$ with initial vertex $x$, and for each such $E$ the initial vertex of the path $f^\psi(E) = f^{\omega(\psi)}_\#(E)$ equals $f^{\omega(\psi)}(x)$; and if in addition $x \in P$ then both of $f^{\omega(\psi)}$ and $f^\psi_t$ fix $x$. This shows that $f^\psi$ is well-defined on $V$, and that it restricts to $f^{\omega(\psi)}$ on $V$, completing the proof of item~\pref{ItemOnFrontier}; it also follows that $f^\psi$ is continuous. 

The proof that $f^\psi$ is a homotopy equivalence and is a topological representative of $\psi$ will be delayed to the very end.

\smallskip\textbf{$X_s$-paths under $f^\psi$.} Item~\pref{ItemOnNielsen} is encompassed in the following generalization of~\pref{ItemOnE}, which will also be used for item~\pref{ItemSemiUnique}:
\begin{enumeratecontinue}
\item\label{Item_f_psiOnXsPath} For each $\psi \in \D_+$ and each $X_s$-path $\gamma$ with endpoints in $V$ we have \hbox{$f^\psi_\#(\gamma) = f^{\omega(\psi)}_\#(\gamma)$.}
\end{enumeratecontinue}
To see why this holds, the general $X_s$-path $\gamma$ with endpoints in $V$ is a concatenation of three special types, and it suffices to check~\pref{Item_f_psiOnXsPath} only for these types: 
\begin{description}
\item[Type (a):] edges in $H^z_u$; 
\item[Type (b):] $X_s$-paths in $G_t$ having endpoints in the set $P$. 
\item[Type (c):] a height~$u$ indivisible Nielsen path of $f$;
\end{description}
Type~(a) is the special case handled already in item~\pref{ItemOnE}. 

If $\gamma$ is of type~(b), first note that for $\aTup \in \S(f)$ and $\psi = \phi^\aTup \in \L(f)$ we have 
$$f^\psi_\#(\gamma) = (f^\psi_t)_\#(\gamma)=(f^\aTup_t)_\#(\gamma)=f^\aTup_\#(\gamma)=f^{a_s}_\#(\gamma) = f^{\omega(\psi)}_\#(\gamma)
$$
where the second-to-last equation follows from Section~\ref{SectionXsPaths}~\pref{ItemXsPathImages}. For more general $\psi \in \D_+$, choose $\aTup,\bTup \in \S(f)$ so that $\psi = (\phi^{b})^\inv \phi^\aTup$, and note that $a_s - b_s = \omega(\phi^\aTup) - \omega(\phi^\bTup) = \omega(\psi) \ge 0$. In the group $\Aut^\gp(G_t,P)$ we have the equation $[f^\psi_t] = [\bar f^\bTup_t] \, [f^\aTup_t]$ and so we may calculate
%
\begin{align*}
f^\psi_\#(\gamma)  = (f^\psi_t)_\#(\gamma) &\,=\, (\bar f^\bTup_t \composed f^\aTup_t)_\#(\gamma) \,=\, (\bar f^\bTup_t)_\#((f^\aTup_t)_\#(\gamma)) \\
&\,=\, (\bar f^\bTup_t)_\#(f^\aTup_\#(\gamma)) \,=\, (\bar f^\bTup_t)_\#(f^{a_s}_\#(\gamma)) \\
&\,=\, (\bar f^\bTup_t)_\#(f^{b_s}_\#(f^{a_s-b_s}_\#(\gamma))) \,=\, (\bar f^\bTup_t)_\#(f^\bTup_\#(f^{\omega(\psi)}_\#(\gamma))) \\
&\,=\, (\bar f^\bTup_t)_\#((f^\bTup_t)_\#(f^{\omega(\psi)}_\#(\gamma))) \,=\, (\bar f^\bTup_t \composed f^\bTup_t)_\#(f^{\omega(\psi)}_\#(\gamma)) \\
&\,=\, f^{\omega(\psi)}_\#(\gamma)
\end{align*}
where the second equations of the second and third lines follow from Section~\ref{SectionXsPaths} \pref{ItemXsPathIterates} and~\pref{ItemXsPathImages}.

For type~(c), let $\gamma$ be a height $u$ indivisible Nielsen path of $f$. We may write $\gamma$ as an alternating concatenation of nontrivial paths of the form
$$\gamma = \eta_0 \, \mu_1 \, \eta_1 \, \cdots \, \eta_{K-1} \, \mu_K \, \eta_K
$$
where each $\eta_k$ is a path in $H^z_u$ and each $\mu_k$ is a Nielsen path of $f$ in $G_t$ with endpoints in $P$ \cite[Lemma 4.24]{\recognitionTag}. By definition of $f^\psi$ we have
$$f^\psi_\#(\gamma) = [f^{\omega(\psi)}_\#(E_1) \,\, (f^\psi_t)_\#(\mu_1) \, f^{\omega(\psi)}_\#(\eta_1) \,\, \cdots \,\, f^{\omega(\psi)}_\#(\eta_{K-1}) \,\, (f^\psi_t)_\#(\mu_K) \,\, f^{\omega(\psi)}_\#(\eta_K)]
$$
We claim that each $\mu_k$ is a Nielsen path of $f^\psi_t$ for each $\psi \in D$, and so $(f^\psi_t)_\#(\mu_k)=\mu_k$. To prove this claim, it holds if $\psi = \phi^\aTup \in \L(f)$ for some $\aTup \in \S(f)$, because in that case the left hand side equals $(f^\aTup)_\#(\mu_k) = \mu_k$. Using that $[\bar f^\psi_t] = [f^\psi_t]^\inv \in \Aut^\gp(G,P)$, we have $\mu_k = (\bar f^\psi_t)_\#(f^\psi_t(\mu_k)) = (\bar f^\psi_t)_\#(\mu_k)$, and so the claim holds if $\psi = (\phi^\aTup)^\inv$ for some $\aTup \in \S(f)$. The general case holds because $\psi = (\phi^\bTup)^\inv \phi^\aTup$ for some $\aTup,\bTup \in \S(f)$. Applying the claim we have:
\begin{align*}
f^\psi_\#(\gamma) &= [f^{\omega(\psi)}_\#(E_1) \,\, \mu_1 \, f^{\omega(\psi)}_\#(\eta_1) \,\, \cdots \,\, f^{\omega(\psi)}_\#(\eta_{K-1}) \,\, \mu_K \,\, f^{\omega(\psi)}_\#(\eta_K)] \\ 
&= [f^{\omega(\psi)}_\#(E_1) \,\, f^{\omega(\psi)}_\#(\mu_1) \, f^{\omega(\psi)}_\#(\eta_1) \,\, \cdots \,\, f^{\omega(\psi)}_\#(\eta_{K-1}) \,\, f^{\omega(\psi)}_\#(\mu_K) \,\, f^{\omega(\psi)}_\#(\eta_K)] \\
 &= f^{\omega(\psi)}_\#(\gamma)
\end{align*}
This complete the proof of~\pref{Item_f_psiOnXsPath}. 

This also completes the construction of $f^\psi \from G \to G$ and the proof of~\pref{ItemOnHzN} for each $\psi \in \D_+$, \emph{except} that we will delay further the proof that $f^\psi$ is a homotopy equivalence and a topological representative of $\psi$.

\smallskip\textbf{Constructing $h^{\psi,\psi'}$.} Given $\psi,\psi' \in \D_+$ we turn to the construction of the homotopy $h^{\psi,\psi'} \from G \times [0,1] \to G$ from $f^\psi \circ f^{\psi'}$ to $f^{\psi\psi'}$. Let $\theta = \psi\psi' \in \D_+$. First, using the homomorphism $\A \from \D \to \Aut^\gp_0(G,P)$ we have $\A(\psi\psi')=\A(\psi)\A(\psi')$ which translates to $[f^{\psi\psi'}_t] = [f^\psi_t] \, [f^{\psi'}_t]$ and so there exists a homotopy rel $P$ from $f^\psi_t \circ f^{\psi'}_t$ to $f^{\psi\psi'}_t$ denoted $h^{\psi,\psi'}_t \from G_t \times [0,1] \to G_t$. Second, for each $E \subset H^z_u$ we have 
$$f^{\vphantom\prime\psi}_\# (f^{\psi'}_\#(E)) = f^\psi_\#(f^{\omega(\psi')}_\#(E)) = f^{\omega(\psi)}_\#(f^{\omega(\psi')}_\#(E)) = f^{\omega(\psi\psi')}_\#(E) = f^{\psi\psi'}_\#(E)
$$
where the second equation holds by applying~\pref{Item_f_psiOnXsPath} together with the fact that $f^{\omega(\psi')}_\#(E)$ is an $X_s$-path (Section~\ref{SectionXsPaths}~\pref{ItemXsPathIterates}). Putting these together, we may define $h^{\psi,\psi'}$ so that for each edge $E \subset H^z_u$ its restriction to $E \times [0,1]$ is a homotopy rel endpoints from $f^{\vphantom\prime\psi} \circ f^{\psi'} \restrict E$ to $f^{\psi\psi'} \restrict E$, and its restriction to $G_t \times[0,1]$ is equal to $h^{\psi,\psi'}_t$. Items \pref{ItemVNStationary} and \pref{ItemHomotopyPreservesGNMinusOne} are evident from the construction. Item~\pref{ItemPsiPsiPrimeEEquation} follows the definition of a relative train track map, which tells us that for each edge $E \subset H_u$ and each integer $i \ge 0$ the path $f^i_\#(E)$ is $u$-legal, and that for each $u$-legal path $\gamma$ the homotopy that straightenes $f^i(\gamma)$ to produce $f^i_\#(\gamma)$ cancels no edges of $H_u$. 

\medskip\textbf{Topological representatives.} For each $\psi \in \D_+$, since $\L(f)$ generates $\D$ we may choose $\aTup,\bTup \in \S(f)$ so that $\psi = (\phi^\bTup)^\inv \phi^\aTup$, and hence $\phi^\bTup \psi = \phi^\aTup$. Since all three of $\phi^\bTup$, $\psi$, $\phi^\aTup$ are in $\D_+$, we have a homotopy $h^{\phi^\bTup,\psi}$ from $f^\bTup \circ f^\psi$ to $f^\aTup$. Since $f^\aTup,f^\bTup$ are homotopy equivalences it follows that $f^\psi$ is also, and since $f^\aTup,f^\bTup$ are topological representatives of $\phi^\aTup,\phi^\bTup$ respectively it follows that $f^\psi$ is a topological representative of $(\phi^\bTup)^\inv \phi^\aTup=\psi$.

\subsection{Semigroup actions of $\wh\D_+$ on $T$ and $T^*$}
\label{SectionActionOnT}
We turn now to the construction of the action $\wh\D_+ \act T$, deriving various properties of the construction, and then we extend the action to $\wh\D_+ \act T^*$. 

Associated to each $\Psi \in \wh \D_+$ we define a map $f^\Psi \from T \to T$ by carrying out the following steps. First let $\psi \in \D_+$ be the image of $\Psi$ under the homomorphism $\Aut(F_n) \mapsto \Out(F_n)$, and consider the map $f^\psi \from (G,G_{u-1}) \to (G,G_{u-1})$, part of the ``homotopy semigroup action'' constructed in Section~\ref{SectionHomotopyActionOnG}. Let $\ti f^\Psi \from (\wt G,\wt G_{u-1}) \to (\wt G,\wt G_{u-1})$ be the unique $\Psi$-twisted equivariant lift of $f^\psi$ to $\wt G$. Let $f^\Psi \from T \to T$ be the $\Psi$-twisted equivariant map induced from $\ti f^\Psi$ by collapsing each component of $\wt G_{u-1}$ to a point and then straightening on each edge $E \subset T$ so that $f^\Psi \restrict E$ stretches length by a constant factor. 

We record several properties of this semigroup action:
\begin{enumerate}
\item \emph{Twisted equivariance:}
\label{Item_f_Twisted}
The map $f^\Psi$ is $\Psi$-twisted equivariant. \qed
\item \emph{Semigroup action property:}
\label{Item_f_PsiComp}
$f^\Psi \composed f^{\Psi'} = f^{\Psi\Psi'}$ for all $\Psi,\Psi' \in \wh\D_+$. 
\end{enumerate} 
\vspace{-4pt}
Property~\pref{Item_f_PsiComp} follows from the fact that $f^\Psi(f^{\Psi'}(E)) = f^{\Psi\Psi'}(E)$ for each edge $E \subset T$, which is a consequence of Section~\ref{SectionHomotopyActionOnG} item~\pref{ItemPsiPsiPrimeEEquation}  applied to edges of $H_u$. \qed

\begin{enumeratecontinue}
\item \emph{Vertex Action Property:}
\label{Item_f_PsiVertices}
$f^\Psi$ takes vertices to vertices, and it restricts to a bijection of the set of vertices having nontrivial stabilizer.
\end{enumeratecontinue}
\vspace{-8pt}
For the proof, denote vertex sets by $V(G) \subset G$ and $V(T) \subset T$, and let $V^\nt(T)$ denote the subset of all $v \in V(T)$ such that $\Stab_{F_n}(v)$ is nontrivial. By Section~\ref{SectionHomotopyActionOnG} items~\pref{ItemOnFrontier} and~\pref{ItemOnGNMinusOne}, the map $f^\psi$ takes the set $G_{u-1} \union V(G)$ to itself, and since $f^\psi$ is a topological representative it follows that $f^\psi$ restricts further to a homotopy equivalence amongst the noncontractible components of $G_{u-1} \union V(G)$. Letting $\wt G_{u-1} \subset \wt G$ and $V(\wt G) \subset \wt G$ be the total lifts of $G_{u-1}$ and $V(G)$, it follows that $\ti f^\Psi \from \wt G \to \wt G$ restricts to a self-map of the components of $\wt G_{u-1} \union V(\wt G)$, and it restricts further to a bijection amongst the components having nontrivial stabilizer. Under the quotient map $\wt G \to T$, the set of components of $\wt G_{u-1} \union V(\wt G)$  corresponds bijectively and $F_n$-equivariantly to the vertex set $V(T)$. It follows that $f^\Psi \from T \to T$ restricts to a self map of $V(T)$, and that it restricts further to a bijection of $V^\nt(T)$, proving property~\pref{Item_f_PsiVertices}. \qed

\begin{enumeratecontinue}
\item \emph{Stretch Property:}
\label{Item_f_PsiEdges}
$f^\Psi$ maps each edge $E \subset T$ to an edge path $f^\Psi(E) \subset T$, stretching length by a uniform factor of~$\lambda^{\wh\omega(\Psi)}$. 
\end{enumeratecontinue}
\vspace{-6pt}
This follows from the definition of the piecewise Riemannian metric on $T$ in Section~\ref{SectionMetricOnT}, the eigenlength equation in $T$ in Section~\ref{SectionTStar}, and Section~\ref{SectionHomotopyActionOnG} item~\pref{ItemOnE}. \qed

\begin{enumeratecontinue}
\item\emph{Train Track Property:}
\label{Item_f_PsiTTProp}
For each $\Psi,\Psi' \in \wh\D$ and each edge $E$ of $T$, the restriction of $f^\Psi$ to the edge path $f^{\Psi'}(E)$ is injective. 
\end{enumeratecontinue}
\vspace{-4pt}
This follows from property~\pref{Item_f_PsiComp} together with property~\pref{Item_f_PsiEdges} as applied to $\Psi\Psi'$. \qed

\medskip

For the statement of property~\pref{Item_f_KernelActs}, recall from Section~\ref{SectionFreeSplittingExtension} the subgroup $\Stab[T] \subgroup \Out(F_n)$ and its pre-image $\wtStab[T] \subgroup \Aut(F_n)$. Recall particularly Lemma~\ref{SectionFreeSplittingExtension} of that section, which says that each subgroup of $\wtStab[T]$ containing $\Inn(F_n)$ has a unique action on $T$ extending the given action of the free group $F_n \approx \Inn(F_n)$ such that each element of the action satisfies a twisted equivariance property.

\begin{enumeratecontinue}
\item \emph{Restricted action of $\wh\D_0$:}
\label{Item_f_KernelActs}
We have $\D_0 \subgroup \Stab[T]$ and hence $\wh\D_0 \subgroup \wtStab[T]$. Furthermore, the restriction to $\wh\D_0$ of the semi-action $\wh\D_+ \act T$ is identical to the action $\wh\D_0 \act T$ given by Lemma~\ref{SectionFreeSplittingExtension}: the unique isometric action of $\wh\D_0$ on~$T$ such that each $\Psi \in \wh\D_0$ satisfies $\Psi$-twisted equivariance.
\end{enumeratecontinue}
\vspace{-4pt}
For the proof, consider $\Psi \in \wh\D_0$ with projected image $\psi \in \D_0$. The map $f^\psi \from G \to G$ restricts to the identity on each edge $E \subset H_u$, by Section~\ref{SectionHomotopyActionOnG} item~\pref{ItemOnE}. The map $\ti f^\Psi \from \wt G \to \wt G$ therefore permutes the edges of $\wt H_u$, mapping each isometrically to its image. It follows that the map $f^\Psi \from T \to T$ is an isometry. Since $f^\Psi$ satisfies $\Psi$-twisted equivariance (property~\pref{Item_f_Twisted}), it follows that $\psi \in \Stab[T]$ (Lemma~\ref{LemmaTwistedUniqueness}). Since $\psi \in \D_0$ is arbitrary, we have proved $\D_0 \subgroup \Stab[T]$ and that $\wh\D_0 \subgroup \wtStab[T]$. Applying, Lemma~\ref{LemmaAutTreeAction} we have also proved that the map $\Psi \mapsto f^\Psi$ is the same as the restriction to $\D_0$ of the action $\wtStab[T] \act T$ given in that lemma, namely the unique action assigning to $\Psi$ the unique $\Psi$-twisted equivariant isometry of~$T$. \qed

\begin{enumeratecontinue}
\item \emph{Invariance of Nielsen data:}
\label{Item_f_PsiNielsen} 
For each $\Psi \in \wh\D$ and each $j \in J$ there exists $j' \in J$ such that $N_{j'} = (f^\Psi)_\#(N_j)$, and $f^\Psi$ restricts to an order preserving bijection of basepoint sets $Z_j \xrightarrow{f^\Psi} Z_{j'}$. In particular $f^\Psi_\#$ induces a bijection of the Nielsen paths in~$T$ (Definition~\ref{DefNielsenSet}).
\end{enumeratecontinue}
\vspace{-4pt}\nobreak
This follows immediately from Section~\ref{SectionHomotopyActionOnG} item~\pref{ItemOnNielsen}. \qed

\medskip

Finally, as an immediate consequence of property~\pref{Item_f_PsiNielsen} we have:
\begin{enumeratecontinue}
\item \label{Item_f_StarExtension}
\emph{Extension to $\wh\D_+ \act T^*$:} The semigroup action $\wh\D_+ \act T$ extends uniquely to a semigroup action $\wh\D_+ \act T^*$, denoted \hbox{$f^\Psi_* = f^\Psi_{T^*} \from T^* \to T^*$} for each $\Psi \in \wh\D_+$, such that $f^\Psi_*$ permutes the cone points $P_j$, and $f^\Psi_*$ permutes the cone edges $\overline{P_j Q}$ by isometries. In particular, for each $j \in J$ and $Q \in Z_j$, letting $N_{j'}=f^\Psi_\#(N_j)$ and $Q' = f^\Psi(Q)$, we have $f^\Psi_*(P_j)=P_{j'} \quad\text{and}\quad \Psi\bigl( \overline{P_jQ} \bigr) = \overline{P'_jQ'}$. \qed
\end{enumeratecontinue}

\subsection{Dynamics of the group action $\wh\D_0 \act T^*$}
\label{SectionTStarDynamics}

Isometries of a simplicial tree equipped with a geodesic metric satisfy a dichotomy: each isometry is either elliptic or loxodromic. This dichotomy does not hold for all isometries of Gromov hyperbolic spaces --- in general, there is a third category of isometries, namely the parabolics. In this section we prove Lemma~\ref{LemmaTStarDynamics} which says in part that the dichotomy \emph{does} hold for the action $\wh\D_0 \act T^*$. 

\begin{lemma}
[Dynamics on $T^*$]
\label{LemmaTStarDynamics}
Under the action $\wh\D_0 \act T^*$, each element $\Delta \in \wh\D_0$ acts either loxodromically or elliptically on $T^*$. More precisely, $\wh\D_0$ acts loxodromically on $T^*$ if and only if $\wh\D_0$ acts loxodromically on $T$ and its axis $A \subset T$ is not a Nielsen line (Definition~\ref{DefNielsenSet}), in which case $A$ is a quasi-axis for $\Delta$ in $T^*$. Furthermore, the loxodromic elements of $\wh\D_0$ have ``uniform uniformly positive stable translation length'' in the following sense: there exist constants $\eta,\kappa>0$ such that for each $\Delta \in \wh\D_0$ acting loxodromically on $T^*$, and for each $x \in T^*$ and each integer $k \ge 0$, we have
$$d_*(\Delta^k(x),x) \ge k \, \eta - \kappa
$$
\end{lemma}

\subparagraph{Remarks.}
In this lemma, the terminology ``uniform uniformly positive stable translation length'' refers to the corollary that the stable translation length 
$$\lim_{k \to \infty} \frac{1}{k} d_*(\Delta^k(x),x))
$$
has a uniform positive lower bound $\eta>0$ independent of the choice of a loxodromic element $\Delta \in \wh\D_0$, and the rate at which this positive lower bound is approached is uniform. This property will be applied in Section~\ref{SectionMultiEdgeExtension}, for purposes of characterizing the loxodromic elements of the action $\wh\D \act \S$ that is described in Section~\ref{SectionSuspConstr}.

\smallskip

Lemma~\ref{LemmaTStarDynamics} and proof could already have been presented almost word-for-word back in Section~\ref{SectionThreeProps} for the restricted action $\Inn(F_n) \approx F_n \act T^*$. Other than the methods available in Section~\ref{SectionThreeProps}, the additional things needed to prove the lemma are that the larger action $\wh\D_0 \act T^*$, like its restriction, satisfies twisted equivariance and preserves Nielsen paths and associated objects (Section~\ref{SectionActionOnT}~\pref{Item_f_PsiNielsen}, \pref{Item_f_StarExtension}).

\begin{proof} Throughout the proof we will apply Section~\ref{SectionActionOnT}~\pref{Item_f_PsiNielsen}, \pref{Item_f_StarExtension} regarding the action of $\wh\D_0$ on Nielsen paths, elements of the Nielsen set $\{N_i\}$, the basepoint sets $Z_i \subset N_i$, the cone points $P_i$, and the cone edges $\overline{P_i Q}$, $Q \in Z_i$. In particular, because the action of each $\Delta \in \wh\D_0$ on $T$ preserves Nielsen paths, it takes maximal $\rho^*$-subpaths of any path $\alpha$ to maximal $\rho^*$-subpaths of the path $\Delta(\alpha)$.

Consider $\Delta \in \wh\D_0$. If $\Delta$ acts elliptically on $T$ then it fixes a vertex of $T$; and if in the geometric case $\Delta$ acts loxodromically on $T$ and its axis is a Nielsen line $N_j$, then $\Delta$ fixes the corresponding cone point $P_j$; in either case $\Delta$ acts elliptically on~$T^*$.

Suppose that $\Delta$ acts loxodromically on $T$ and its axis $A \subset T$ is not a Nielsen line (this always happens in the nongeometric case, but the terminology of our proof applies to all cases). For each vertex $x \in T$ and each integer $k \ge 1$ consider the Proposition~\ref{PropDecompInT} decomposition of the path $[x,\Delta^k(x)]$ into ``$\nu$-subpaths'' meaning maximal $\rho^*$ subpaths, alternating with ``$\mu$-subpaths'' each having no $\rho^*$ subpath. Since the intersection $[x,\Delta^k(x)] \intersect A$ contains a concatenation of $k$ fundamental domains for the action of $\Delta$ on $A$, it follows that if $[x,\Delta^k(x)] \intersect A$ contains a $\nu$-subpath of $[x,\Delta^k(x)]$ then it must contain $k-1$ distinct $\nu$-subpaths of $[x,\Delta^k(x)]$, between which there are $k-2$ distinct $\mu$-subpaths of $[x,\Delta^k(x)]$; here we are applying the fact that $\Delta$ takes maximal $\rho^*$ paths to maximal $\rho^*$ paths. As a consequence, the collection of $\mu$-subpaths of $[x,\Delta^k(x)]$ contains at least $k-2$ distinct edges of $T$. Applying Proposition~\ref{PropDecompInT}~\pref{ItemDPFFormula} we obtain $D_\PF(x,\Delta^k(x)) \ge (k-2) \eta'$ where $\eta' = \min\{l_\PF(E) \suchthat \text{$E$ is an edge of $T$}\}$. Applying Proposition~\ref{PropConeQI}, in $T^*$ we have
$$d_*(x,\Delta^k(x)) \ge k \underbrace{\frac{\eta'}{K}}_{= \, \eta} - \biggl(\underbrace{\frac{2\eta'}{K} + C}_{= \, \kappa'}\biggr)
$$
It immediately follows that $\Delta$ acts loxodromically on $T^*$. This estimate is for vertices of $T$, but a similar estimate holds for arbitrary points $x \in T^*$, replacing $\kappa'$ by $\kappa = \kappa' + 2 \delta$ where $\delta$ is an upper bound for the distance in $T^*$ from an arbitrary point of $T^*$ to the vertex set of $T$.
\end{proof}

\section{The suspension action.} 
\label{SectionSuspension}

In this section we complete the proof of the multi-edge case of the Hyperbolic Action Theorem. 

Throughout this section we adopt Notations~\ref{NotationsFlaring} and~\ref{NotationsActionOnT} concerning a rotationless $\phi \in \Out(F_n)$ and a \ct\ representative $f \from G \to G$ with top \eg\ stratum $H_u$, with disintegration group $\D=\D(f) \subgroup \Out(F_n)$, and with free splitting $F_n \act T$ associated to the pair $(G,G_u)$. In particular, in Section~\ref{SectionPathFunctionsReT} we used the coned off free splitting $F_n \act T^*$ to construct the suspension space $\S$, proving that $\S$ is Gromov hyperbolic by using flaring properties of the action $F_n \act T^*$. In Section~\ref{SectionTTSemiAction}  we studied the extended disintegration group $\wh\D \subgroup \Aut(F_n)$ and its subsemigroup $\wh\D_+ \subgroup \wh\D$, and we constructed the train track semigroup action $\wh\D_+ \act T^*$.

In Section~\ref{SectionSuspConstr} we shall show how to suspend the semigroup action $\wh\D_+ \act T^*$ to an isometric action $\wh\D \act \S$ (this is a completely general construction that applies to the disintegration group of any \ct\ having top \eg\ stratum).

In Section~\ref{SectionMultiEdgeExtension} we put the pieces together to complete the proof.


\subsection{The suspension action $\wh\D \act \S$.} 
\label{SectionSuspConstr}

Recall from Notations~\ref{NotationsActionOnT} the short exact sequence
$$1 \mapsto \wh\D_0 \mapsto \wh\D \xrightarrow{\wh\omega_u} \Z \to 1
$$
Choose an automorphism $\Phi \in \wh\D$ representing $\phi \in \D$, and so $\wh\omega_u(\Phi)=\omega(\phi)=1$. It follows that $\Phi$ determines a splitting of the above short exact sequence, and hence a semidirect product structure $\wh\D \approx \wh\D_0 \semidirect_{\Phi} \Z$ expressed in terms of the inner automorphism on $\wh\D$ determined by $\Phi$, namely $I(\Psi) = I_\Phi(\Psi) = \Phi\Psi\Phi^\inv$.  
Thus each element of $\wh\D$ may be written in the form $\Delta \Phi^m$ for unique $\Delta \in \wh\D_0$ and $m \in \Z$, and the group operation in this notation is defined by
$$(\Delta_1 \Phi^{m_1})(\Delta_2 \Phi^{m_2}) = (\Delta_1 I^{m_1}(\Delta_2)) \Phi^{m_1+m_2}
$$


Recall from Definition~\ref{DefSuspensionSpace} the construction of $\S$: using the chosen automorphism $\Phi$ representing $\phi$, the corresponding $\Phi$-twisted equivariant map of $T^*$ is denoted in abbreviated form as $f_* = f^\Phi_{T^*}$, and $\S$ is the suspension space of $f_*$, namely the quotient of $T \times \Z \times [0,1]$ with identifications $[x,k,1] \sim [f_*(x),k+1,0]$. Recall also various other notations associated to $\S$ that are introduced in Definitions~\ref{DefSuspensionSpace} and~\ref{DefSuspensionMetric}.

To define the action on $\S$ of each element $\Psi = \Delta\Phi^k \in \wh\D$ it suffices to carry out the following steps: first we define the group action $\wh\D_0 \act \S$; next we define the action of the element $\Phi \act \S$; then we define the action of each element $\Delta \Phi^k$ by composition; and finally we verify that the two sides of the semidirect product relation $\Phi \Delta = I(\Delta) \Phi$ act identically. 

Let each $\Delta \in \wh\D_0$ act on $[x,k,t] \in \S$ by the equation
$$\Delta \cdot [x,k,t] = [f^{I^k(\Delta)}_*(x), k, t] = [f^{\Phi^k \Delta \Phi^{-k}}_*(x),k,t]
$$
and note that this formula is well defined because, using the properties of the semigroup action $\wh\D_+ \act T$ from Section~\ref{SectionActionOnT}, we have
\begin{align*}
\Delta \cdot [x,k,1] &= [f^{\Phi^k \Delta \Phi^{-k}}_*(x),k,1] 
 = [f^\Phi_* \circ f^{\Phi^k \Delta \Phi^{-k}}_*(x) ,k+1,0] \\
 &= [f^{\Phi^{k+1} \Delta \Phi^{-(k+1)}}_* \circ f^\Phi_*(x),k+1,0] \\
 &= \Delta \cdot [f^\Phi_*(x),k+1,0] \\
 &= \Delta \cdot [f_*(x),k+1,0]
\end{align*}
Each $\Delta \in \wh\D_0$ evidently acts by an isometry. Again using the semigroup action properties, the action equation is satisfied for each $\Delta',\Delta \in \wh\D_0$:
\begin{align*}
\Delta' \cdot (\Delta \cdot [x,k,t]) &= [f^{I^k(\Delta')}_{*} \circ f^{I^k(\Delta)}_{*}(x),k,t] \\
  &= [f^{I^k(\Delta') \circ I^k(\Delta)}_*(x),k,t] \\
  &= [f^{I^k(\Delta'\Delta)}_*(x),k,t] \\
  &= \Delta' \Delta \cdot [x,k,t]
\end{align*}
This completes the definition of the isometric action $\wh\D_0 \act \S$.
%
%

We note that the restriction to $F_n \approx \Inn(F_n)$ of the action $\wh\D_0 \act \S$ agrees with the action $F_n \act \S$ given in Definition~\ref{DefSuspensionSpace}, defined in general by $\gamma \cdot [x,k,r] = [\Phi^k(\gamma) \cdot x, k, r]$ for each $\gamma \in F_n$. In the special case $k=0$, $r=0$, the equation $\gamma \cdot x = f^{i_\gamma}_*(x)$ holds by Lemma~\ref{LemmaTwistedUniqueness}~\pref{ItemLTUInner}. The extension to the general case follows easily.

Next, let $\Phi$ act by shifting downward,
$$\Phi \cdot [x,k,t] = [x,k-1,t]
$$
which is evidently a well-defined isometry. 

As required, for each $\Delta \in \wh\D_0$ the two sides of the semidirect product equation act in the same way:
\begin{align*}
\Phi \cdot \Delta \cdot [x,k,t] &= \Phi \cdot [f^{I^k(\Delta)}_*(x),k,t] \\
 &= [f^{I^k(\Delta)}_*(x),k-1,t] \\
I(\Delta) \cdot \Phi \cdot [x,k,t] &= I(\Delta) \cdot [x,k-1,t] \\
 &= [f^{I^{k-1}(I(\Delta))}_*(x),k-1,t] \\
 &= [f^{I^k(\Delta)}_*(x),k-1,t]
\end{align*}
Notice that since the action of each $\Delta \in \wh\D_0$ preserves each ``horizontal level set'' $\S_s$, and since the action of $\Phi$ has ``vertical translation'' equal to $-1=-\wh\omega(\Phi)$ meaning that $\Phi(\S_s) = \S_{s-\wh\omega(\Phi)}$, it follows that for each $\Psi \in \wh\D$ the integer $-\wh\omega(\Psi)$ is the ``vertical translation length'' for the action of $\Psi$ in $\S$. We record this for later use as:

\begin{lemma}\label{LemmaVertTranslOnS}
For each $\Psi \in \wh\D$ and each fiber $\S_s$ ($s \in \reals$), we have 
$$\Psi(\S_s) = \S_{s - \wh\omega(\Psi)} \quad\text{for any $s \in \reals$}
$$
\qed\end{lemma}

\subsection{Hyperbolic Action Theorem, Multi-edge case: Proof.}
\label{SectionMultiEdgeExtension}

First we set up the notation, based on the formulation of the multi-edge case in Section~\ref{SectionOneEdgeExtension} and the outline of the proof in Section~\ref{SectionMultiEdgeOutline}, and we apply the results of Section~\ref{SectionSuspConstr} to obtain a hyperbolic action $\wh\cH \act \S$. After that, Conclusions~\pref{ItemThmF_EllOrLox}, \pref{ItemThmF_Nonelem} and~\pref{ItemThmF_WWPD} of the Hyperbolic Action Theorem are proved in order. 

\paragraph{Setup of the proof.} \quad
We have a subgroup $\wh\cH \subgroup \Aut(F_n)$, having image $\cH \subgroup \Out(F_n)$ under the quotient homomorphism $\Aut(F_n) \mapsto \Out(F_n)$, and having kernel $J = \kernel(\wh\cH \mapsto \cH)$, which satisfies the hypotheses of the Hyperbolic Action Theorem:  $\cH$ is abelian;  $\wh\cH$ is finitely generated and not virtually abelian;  and no proper, nontrivial free factor of $F_n$ is $\wh\cH$-invariant. 

The conclusion of the Hyperbolic Action Theorem is about the existence of a certain kind of hyperbolic action of a finite index normal subgroup of $\wh\cH$, and so we are free to replace $\wh\cH$ with any finite index subgroup $\wh\cH' \subgroup \wh\cH$, because once the conclusion is proved using some hyperbolic action $\cN' \act \S$ of some finite index normal subgroup $\cN' \normal \wh\cH'$, the same conclusion follows for the restriction of $\cN' \act \S$ to the action $\cN \act \S$ where $\cN$ is the intersection of all of its $\wh\cH$ conjugates of $\cN'$; one need only observe that the conclusions of the Hyperbolic Action Theorem for the action $\cN' \act S$ imply the same conclusions when restricted to any finite index subgroup of~$\cN'$.

We may therefore assume that $\cH$ is a rotationless abelian subgroup, using the existence of an integer constant $k$ such that the $k^{\text{th}}$ power of each element of $\Out(F_n)$ is rotationless, and replace the abelian group $\cH$ by its finite index subgroup of $k^{\text{th}}$ powers.

We have a maximal, proper, $\cH$-invariant free factor system $\B$ of co-edge number $\ge 2$ in $F_n$. Applying the Disintegration Theorem, we obtain $\phi \in \cH$ such that for any \ct\ representative $f \from G \to G$ of $\phi$ with disintegration group $\D=\D(f)$, the subgroup $\cH \intersect \D$ has finite index in~$\cH$. We choose $f$ so that the penultimate filtration element~$G_{u-1}$ represents the free factor system~$\B$. Since the co-edge number of $\B$ is $\ge 2$, the top stratum $H_u$ is \eg. Replacing $\cH$ with $\cH \intersect \D$, we may thus assume $\cH \subgroup \D$. 

We now adopt Notations~\ref{NotationsFlaring} and~\ref{NotationsActionOnT}, regarding the free splitting $F_n \act T$ associated to the marked graph pair $(G,G_{u-1})$, the action $F_n \act T^*$ obtained by coning off the elements of the Nielsen set, the disintegration group $\D=\D(f)$, and its associated extended disintegration group $\wh\D$. Using that $\cH \subgroup \D$ it follows that $\wh\cH \subgroup\wh\D$. Setting $J = \kernel(\wh\cH \mapsto \cH)$ and $K = \kernel(\wh\cH \xrightarrow{\wh\omega_u} \Z)$ we may augment the commutative diagram of Notations~\ref{NotationsFlaring}~\pref{ItemDisintGroup} to obtain the commutative diagram with short exact rows shown in Figure~\ref{FigCommutativeKJ}.
\begin{figure}\label{FigCommutativeKJ}
$$\xymatrix{
1 \ar[r] & K =   \wh\cH \intersect \wh\D_0 \ar[r] & \wh\cH \ar[r]^{\wh\omega_u} \ar@{=}[d] & \Z \ar[r] & 1 \\
1 \ar[r] & J =   \wh\cH \intersect \Inn(F_n) \ar[r] \ar[d]^{\subset} \ar[u]_{\subset} & \wh \cH \ar[r] \ar[d]^{\subset}  & \cH \ar[r] \ar[d]^{\subset} \ar[u]_{\omega_u} & 1 \\
1 \ar[r] & F_n \approx \Inn(F_n) \ar[r] \ar@{=}[d] & \wh\D \ar[r] \ar[d]^{\subset} & \D \ar[r] \ar[d]^{\subset} &1 \\
1 \ar[r] & F_n \approx \Inn(F_n) \ar[r] & \Aut(F_n) \ar[r] & \Out(F_n) \ar[r] & 1 \\
}$$
\caption{The group $\wh\cH$ and associated objects in a commutative diagram with short exact rows. The automorphism $\Phi \in \wh\cH \in \Aut(F_n)$ is a chosen pre-image of $\phi \in \cH$ and hence satisfies $\wh\omega_u(\Phi)=\omega_u(\phi)=1$.}
\end{figure}

From Notations~\ref{NotationsActionOnT}~\pref{ItemSubsemigroups} we also have the subgroups and subsemigroups $\D_0  \subgroup \D_+ \subgroup \D$ and $\wh\D_0 \subgroup \wh\D_+ \subgroup \wh\D$. \

Let $\wh\D_+ \act T^*$ be the semigroup action described in Section~\ref{SectionActionOnT}, associating to each $\Psi \in \wh\D_+$ a map $f^\Psi_{T^*} \from T^* \to T^*$. Pick $\Phi \in \wh\cH \subgroup \wh\D$ to be any pre-image of $\phi \in \cH \subgroup \D$, and so $\wh \omega_u(\Phi) = \omega_u(\phi)=1$ (see Notations~\ref{NotationsActionOnT}~\pref{ItemCoordHomDf}). Let $\S$ be the suspension space of $f^\Phi_{T^*} \from T^* \to T^*$ as constructed in Definition~\ref{DefSuspensionSpace}. Let $\wh\D \act \S$ be the isometric action constructed in Section~\ref{SectionSuspConstr}. We will make heavy use of the integer sections $\S_j$ ($j \in \Z$) and of the $\wh\D_0$-equivariant identification $\S_0 \leftrightarrow T^*$.

\medskip

We shall abuse notation freely by identifying $F_n \approx \Inn(F_n)$ given by $\gamma \leftrightarrow i_\gamma$, where $i_\gamma(\delta)=\gamma\delta\gamma^\inv$. For example, using this identification we often think of $J \subgroup K$ as subgroups of $F_n$. We also note the equation
$$ i_{\Psi(\gamma)} = \Psi \circ i_\gamma \circ \Psi^\inv, \quad \Psi \in \Aut(F_n), \gamma \in F_n
$$
which expresses the fact that the isomorphism $F_n \approx \Inn(F_n)$ is equivariant with respect to two actions of $\Aut(F_n)$: its standard action on $F_n$; and its action by inner automorphisms on its normal subgroup $\Inn(F_n)$.  Combining this equation with~\pref{Item_f_Twisted} \emph{Twisted equivariance} of Section~\ref{SectionActionOnT}, it follows that under the isomorphism $F_n \approx \Inn(F_n)$, the action $F_n \act T$ agrees with the action $\Inn(F_n) \act T$ obtained by restricting the action $\wh\D_0 \act T$. 

\medskip

We turn to the proof of the three conclusions of the Hyperbolic Action Theorem for the action of $\cN = \wh\cH$ on~$\S$.

\paragraph{Proof of Conclusion \pref{ItemThmF_EllOrLox}: Every element of $\wh\cH$ acts either elliptically or loxodromically on $\S$.} We show more generally that every element of $\wh\D$ acts either elliptically or loxodromically on~$\S$. From Section~\ref{SectionSuspConstr} the general element of $\wh\D$ has the form $\Psi = \Delta \Phi^m$ for some $\Delta \in \wh\D_0$, and
$\Psi(\S_j) = \S_{j-m}$ for any $j \in \Z$. If $m \ne 0$ then, by Lemma~\ref{LemmaProjectNonIncrease}, for each $k \in \Z$ and each $x \in \S$ we have $d(x,\Psi^k(x)) \ge \abs{km}$, and so $\Psi$ is loxodromic. Suppose then that $m=0$ and so $\Psi=\Delta$ preserves $\S_j$ for each $j \in \Z$. Consider in particular the restriction $\Delta \restrict \S_0 \approx T^*$ and the further restriction $\Delta \restrict T$. If $\Delta$ is elliptic on $T$ then it fixes a point of $T$, hence $\Delta$ fixes a point of $T^* \approx \S_0 \subset \S$, and so $\Delta$ is elliptic on $T^*$ and on $\S$. If $\Delta$ is loxodromic on $T$ with axis $L_\Delta \subset T$, and if $L_\Delta$ is a Nielsen line $N_j$ then $\Delta$ fixes the corresponding cone point $P_j \in T^* \in \S_0 \subset \S$, and so $\Delta$ is elliptic on $T^*$ and on $\S$. 

It remains to consider those $\Delta \in \wh\D_0$ which act loxodromically on $T$ and whose axes in $T$ are not Nielsen lines. Applying Lemma~\ref{LemmaTStarDynamics}, each such $\Delta$ acts loxodromically on $T^*$ and for each $x \in T^*$ and each integer $k \ge 1$ we have
$$(*) \qquad\qquad d_*(x,\Delta^k(x)) \ge k \, \eta - \kappa \qquad\qquad \hphantom{(*)}
$$
where the constants $\eta,\kappa > 0$ are independent of $\Delta$ and $x$.

Consider the function which assigns to each integer $k$ the minimum translation distance of vertices $v \in \S$ under the action of~$\Delta^k$:
$$\sigma(k) = \inf_{v \in \S} d_\S(v,\Delta^k(v))
$$
To prove that $\Delta$ acts loxodromically on $\S$ we apply the classification of isometries of Gromov hyperbolic spaces \cite{Gromov:hyperbolic}, which says that every isometry is elliptic, parabolic, or loxodromic. But if $\Delta$ is elliptic or parabolic then the function $\sigma(k)$ is bounded. Thus it suffices to prove that $\lim_{k \to \infty} \sigma(k) = \infty$. 

For each integer $i$, consider $\Delta'_i = \Phi^i \Delta \Phi^{-i} \in \wh\D_0$. Consider also the $\Phi^{-i}$-twisted equivariant map $j_i \from T^* = \S_0 \to S_i$ given by $j_i[x,0,0] = [x,i,0]$, which is an isometry from the metric $d_*=d_0$ to the metric $d_i$. This map $j_i$ conjugates the action of $\Delta^k$ on $\S_i$ to the action of $(\Delta'_i)^k = \Phi^i \Delta^k \Phi^{-i} \in \wh\D_0$ on $\S_0$, because
\begin{align*}
\Delta^k j_i[x,0,0] &= \Delta^k[x,i,0] = [f_*^{\Phi^i \Delta^k \Phi^{-i}}(x),i,0] \\
  &= [f_*^{\Delta'_i}(x),i,0]  = j_i [f_*^{\Delta'_i}(x),0,0] \\
  &= j_i \Delta'_i [x,0,0]
\end{align*}
Applying the inequality $(*)$ to $\Delta'_i$, and then applying the twisted equivariant isometric conjugacy $j_i$ between $\Delta^k$ and $(\Delta'_i)^k$, for each vertex $p \in \S_i$ we have 
$$(**) \qquad\qquad d_i(p,\Delta^k(p)) \ge k \, \eta - \kappa \qquad\qquad \vphantom{(**)}
$$
As seen earlier, the uniformly proper maps $\S_i \inject \S$ have a uniform gauge function $\delta \from [0,\infty) \to [0,\infty)$ (see $(*)$ in Section~\ref{SectionSHypProof} just before Lemma~\ref{LemmaUnifProp}). If $\sigma(k)$ does not diverge to $\infty$ then there is a constant $M$ and arbitrarily large values of $k$ such that $d_\S(p,\Delta^k(p)) \le M$ holds for some $i$ and some vertex $p \in \S_i$, implying by that $d_i(p,\Delta^k(p)) \le \delta(M)$, contradicting $(**)$ for sufficiently large~$k$. 

This completes the proof of Conclusion~\pref{ItemThmF_EllOrLox}. 

\medskip

Furthermore, we have proved the following which will be useful later:
\begin{itemize}
\item For each $\Delta \in \wh\D_0$ the following are equivalent: $\Delta$~acts loxodromically on~$\S$; $\Delta$~acts loxodromically on $T^*$; \, $\Delta$ acts loxodromically on $T$ and its axis is not a Nielsen line.
\end{itemize}

\paragraph{Proof of Conclusion \pref{ItemThmF_Nonelem}: The action $\wh\cH \act \S$ is nonelementary.} We shall focus on the restricted actions of the subgroup $J$ (often abusing notation, as warned earlier, by identifying $J \subgroup \Inn(F_n)$ with the corresponding subgroup of $F_n$). We prove first that $J$ has nonelementary action on $T$, then on $T^*$, and then on~$\S$; from the latter it follows that the whole action $\wh\cH \act \S$ is nonelementary.

We shall apply Lemma~\ref{LemmaJNonelementary} and so we must check the hypotheses of that lemma. Let $V^\nt$ be the set of vertices $v \in T$ having nontrivial stabilizer under the action of $F_n$. As shown in Section~\ref{SectionActionOnT}~\pref{Item_f_PsiVertices}, the semigroup action $\wh\D_+ \act T$ restricts to a semigroup action $\wh\D_+ \act V^\nt$ having the property that each element of $\Psi \in \wh\D_+$ acts by a $\Psi$-twisted equivariant bijection of~$V^\nt$, and it follows immediately that this semigroup action extends uniquely to a group action $\wh\D \act V^\nt$ having the same property. Restricting to $\wh\cH$ we obtain an action $\wh\cH \act V^\nt$ satisfying the hypotheses in the first paragraph of Lemma~\ref{LemmaJNonelementary}. By hypothesis of the Hyperbolic Action Theorem, no proper nontrivial free factor of $F_n$ is fixed by $\wh\cH$, and in particular no subgroup $\Stab_{F_n}(v)$ ($v \in V^\nt$) is fixed by $\wh\cH$. Finally, the free group $J$ has rank~$\ge 2$ because otherwise, since $\cH$ is abelian, it would follow that $\wh\cH$ is a solvable subgroup of $\Aut(F_n) \subgroup \Out(F_{n+1})$ and hence is virtually abelian by \SubgroupsThree, contradicting the hypothesis that $\wh\cH$ is not virtually abelian. Having verified all the hypotheses of Lemma~\ref{LemmaJNonelementary}, from its conclusion we have that the action $J \act T$ is nonelementary.

Since the action $J \act T$ is nonelementary, its minimal subtree $T^J$ is not a line, and so $T^J$ contains a finite path $\alpha$ which is not contained in any Nielsen line. Furthermore, $\alpha$ is contained in the axis of some loxodromic element $\gamma \in J$ whose axis $L_\gamma$ is therefore not a Nielsen line. Applying Lemma~\ref{LemmaTStarDynamics}, the action of $\gamma$ on $T^*$ is loxodromic and $L_\gamma$ is a quasi-axis for $\gamma$. Choosing $\delta \in J - \Stab_{F_n}(L_\gamma)$ and letting $\gamma' = \delta\gamma\delta^\inv$, it follows that $\gamma'$ also acts loxodromically on $T$ and on $T^*$, and that its axis $L_{\gamma'}$ in $T$ is also a quasi-axis in~$T^*$. By choice of $\delta$ the intersection $L_\gamma \intersect L_{\gamma'}$ is either empty or finite. Since neither of the lines $L_\gamma$, $L_{\gamma'}$ is a Nielsen axis, each ray in each line has infinite $D_\PF$ diameter and so goes infinitely far away from the other line in $T^*$-distance. It follows that $\gamma,\gamma'$ are independent loxodromic elements on $T^*$, proving that $J \act T^*$ is nonelementary. 

Finally, using the same $\gamma,\gamma'$ as in the previous paragraph whose axes $L_\gamma,L_\gamma'$ in $T$ are not Nielsen lines, we showed in the proof of Conclusion \pref{ItemThmF_EllOrLox} that $\gamma,\gamma'$ act loxodromically on~$\S$. Furthermore, since the lines $L_\gamma,L_\gamma'$ have infinite Hausdorff distance in $T^*$, it follows by they also have infinite Hausdorff distance in~$\S$, as shown in item~$(*)$ of Section~\ref{SectionSHypProof}. This proves that $\gamma,\gamma'$ are independent loxodromics on $\S$ and hence the action $J \act \S$ is nonelementary. 

\paragraph{Proof of Conclusion \pref{ItemThmF_WWPD}: Each element of $J$ that acts loxodromically on $\S$ is a strongly axial, WWPD element of $\wh\cH \act \S$.} \quad

\smallskip
Given $\gamma \in J = \wh\cH \intersect \Inn(F_n) \subgroup K$, as was shown earlier under the proof of Condition~\pref{ItemThmF_EllOrLox}, we know that the action of $\gamma$ on $\S$ is loxodromic if and only if its action on $T^*$ is loxodromic, if and only if its action on $T$ is loxodromic with an axis $L_\gamma \subset T$ that is not a Nielsen line. It follows that $L_\gamma$ is a quasi-axis for the actions of $\gamma$ on $T^*$ and on $\S$.
We shall prove that each such element $\gamma$ is a WWPD element with respect to three actions:

\smallskip Step 1: The action $K \act T$, with respect to the metric $d_u$ (Definition~\ref{DefinitionMetricsOnT}); 

\smallskip Step 2: The action $K \act T^*$\vphantom{$\wh\cH$}, with respect to the metric $d^*$ (Definition~\ref{DefLittleDStar}); 

\smallskip Step 3: The action $\wh\cH \act \S$, with respect to the metric $d_\S$ (Definition~\ref{DefSuspensionMetric}). 

\smallskip\noindent
The proofs in Steps 2 and 3 are bootstrapping arguments, deriving the WWPD property for the current step from the WWPD property of the previous step. 

Once the WWPD arguments are complete, we will verify that $\gamma$ is strongly axial with respect to the action $\wh\cH \act \S$.

\medskip\textbf{Step 1: The action $K \act T$.} 
As shown in Section~\ref{SectionActionOnT}~\pref{Item_f_KernelActs}, the action $J \act T$ is the restriction of the action $K \act T$. 
Since $J$ is normal in $K$, and since $J \act T$ is the restriction to $J$ of the free splitting action $\Inn(F_n) \approx F_n \act T$, we may apply Lemma~\ref{LemmaTreeWWPD} to conclude that $\gamma$ is a WWPD element of the action $K \act T$ with respect to the edge counting metric $d_u$.


\medskip\textbf{Step 2: The action $K \act T^*$.}  The underlying intuition of this bootstrapping proof is that WWPD behaves well with respect to coning operations, for elements whose loxodromicity survives the coning process. We shall use the WWPD property of $\gamma$ with respect to the action $K \act T$ and metric $d_u$ to derive the WWPD property of $\gamma$ with respect to the action $K \act T^*$ and the ``coned off'' metric $d^*$. For this purpose we shall use the original version of WWPD from \cite[Section 2.4]{BBF:SCLonMCG}, referred to in \cite[Proposition 2.3]{HandelMosher:WWPD} as ``WWPD~(3)'':
\begin{description}
\item[WWPD~(3):] Given a hyperbolic action $\Gamma \act X$, a loxodromic element $h \in \Gamma$ satisfies WWPD if and only if for any quasi-axis $\ell$ of $h$ and for any closest point map $\pi \from X \to \ell$ there exists $B \ge 0$ such that for any $g \in \Gamma - \Stab(\bdy h)$, the set $g(\ell)$ has diameter $\le B$.  
\end{description}
\emph{Remark:} This statement is equivalent to any alternate version in which either of the universal quantifiers on $\ell$ or on $\pi$ is replaced with an existential quantifier, because any two quasi-axes of $h$ have finite Hausdorff distance, and any two closest point maps $X \to \ell$ have finite distance in the sup metric. We~shall use these equivalences silently in what follows.

Let $\pi \from T \to L_\gamma$ be the retraction which maps each component $C$ of $T-L_\gamma$ to the point where the closure of $C$ intersects $L_\gamma$. This map $\pi$ is the unique closest point map from $T$ to $L_\gamma$ with respect to the metric $d_u$, and so in the context of Step~1 we can apply WWPD~(3) to conclude that there is a constant $B \ge 0$ such that for each $g \in K - \Stab(\bdy\gamma)$, the set $\pi(g(L_\gamma))$ has $d_u$-diameter $\le B$, and hence the set $\pi(g(L_\gamma))$ has $D^u$-diameter $\le B$ (Corollary~\ref{CorollarySomeFormulas}). 

Recall two facts: the coarse metrics $D_u$~and~$D_\PF$ are quasicomparable on $T$ (Definition~\ref{DefinitionMetricsOnT}); and the inclusion map $T \inject T^*$ is a quasi-isometry with respect to $D_\PF$ on $T$ and $d^*$ on $T^*$ (Proposition~\ref{PropConeQI}). It follows from these two facts that the sets $\pi(g(L_\gamma))$ have uniformly bounded $d^*$-diameter for all $g \in K - \Stab(\bdy\gamma)$. Property WWPD~(3) for $h$ with respect to the action $K \act T^*$ and the metric $d^*$ will therefore be proved if we can demonstrate the following: for any closest point map $\pi^* \from T^* \to L_\gamma$ with respect to the metric~$d^*$, the distance $d^*(\pi(x),\pi^*(x))$ is uniformly bounded as $x$ varies over $T$. For this purpose, using the same two facts above, it suffices to prove that the quasidistance $D_u(\pi(x),\pi^*(x))$ is uniformly bounded as $x$ varies over~$T$. Applying the Coarse Triangle Inequality for the path function $L_u$ (Lemma~\ref{LemmaCTI}), and using that $\pi^*(x)$ minimizes $D^u$ distance from $x$ to $L_\gamma$, we have
\begin{align*}
D_u(\pi(x),\pi^*(x)) &\le D_u(\pi(x),x) + D_u(x,\pi^*(x)) +  C_{\ref{LemmaCTI}} \\
  &\le D_u(\pi(x),x) + D_u(x,\pi(x)) + C_{\ref{LemmaCTI}} \\
  &\le 2B + C_{\ref{LemmaCTI}}
\end{align*}
This completes the proof that $\gamma$ is a WWPD element of the action $K \act T^*$.

\subparagraph{Step 3: The action $\wh\cH \act \S$.} Arguing by contradiction, suppose that $\gamma$ fails to be a WWPD element of the action $\wh\cH \act \S$ with respect to the metric $d_\S$. Fix a closest point map $\pi \from \S \to L_\gamma$. Applying WWPD~(3) (stated above, and see the following remark), we obtain an infinite sequence $\Psi_i \in \wh\cH - \Stab_{\wh\cH}(\bdy\gamma)$ ($i=1,2,\ldots$) such that the diameter of $\pi(\Psi_i(L_\gamma))$ goes to $+\infty$ as $i \to +\infty$. Denote $L_i = \Psi_i(L_\gamma)$ which is a $(k,c)$-quasiaxis in $\S$ for $\Psi_i \gamma \Psi_i^\inv$, where $k \ge 1$, $c \ge 0$ are independent of~$i$. 


Using that the coordinate homomorphism $\wh\omega_u \from \wh\cH \to \Z$ is surjective with kernel~$K$, and that $\wh\omega_u(\Phi)=1$ (see Figure~\ref{FigCommutativeKJ}), for each $i$ we have $\Psi_i = \Delta_i \Phi^{m_i}$ for a unique $\Delta_i \in K$ and a unique integer $m_i$. Choose points $w_i,x_i \in L_\gamma$ and $y_i,z_i \in L_i$ such that $\pi(y_i)=w_i$, $\pi(z_i)=x_i$, the point $w_i$ precedes the point $x_i$ in the oriented axis $L_\gamma$, and \hbox{$d_\S(w_i,x_i) \to +\infty$.} By replacing $\Psi_i$ with $\Delta_i \Phi^{m_i} \gamma^k$ for an appropriate exponent $k$ that depends on $i$, after passing to a subsequence we may assume that the subpaths $[w_i,x_i]$ of $L_\gamma$ are nested:
$$(*) \qquad\qquad [w_1,x_1] \subset \cdots \subset [w_i,x_i] \subset [w_{i+1},x_{i+1}] \subset \cdots
$$

We next prove that the sequence of integers $(m_i)$ takes only finitely many values. Consider the $(k,c)$-quasigeodesic quadrilateral $Q_i$ in $\S$ having the $(k,c)$-quasigeodesics $[w_i,x_i] \subset L_\gamma$ and $[y_i,z_i] \subset L_i$ as one pair of opposite sides, and having geodesics $\overline{w_i y_i}$, $\overline{x_i z_i}$ as the other pair of opposite sides. By a basic result in Gromov hyperbolic geometry \cite{Gromov:hyperbolic}, there exists a finite, connected graph $G_i$ equipped with a path metric, and having exactly four valence~$1$ vertices denoted $\bar w_i, \bar x_i, \bar y_i, \bar z_i$, and there exists a quasi-isometry $f_i \from (Q_i,w_i,x_i,y_i,z_i) \to (G_i,\bar w_i, \bar x_i, \bar y_i, \bar z_i)$ with uniform quasi-isometry constants independent of $i$ (depending only on $k$, $c$, and the hyperbolicity constant of~$\S$). We consider two cases, depending on properties of the graph~$G_i$. In the first case, the paths $[\bar w_i, \bar x_i]$ and $[\bar y_i, \bar z_i]$ are disjoint in $G_i$, and so the points $\bar y_i, \bar z_i$ have the same image under the closest point projection $G_i \mapsto [\bar w_i, \bar x_i]$; it follows that $\pi(y_i)=w_i$ and $\pi(z_i)=x_i$ have uniformly bounded distance, which happens for only a finite number of values of $i$. In the second case the paths $[\bar w_i, \bar x_i]$ and $[\bar y_i, \bar z_i]$ are \emph{not} disjoint in $G_i$, and so the minimum distance between the paths $[w_i,x_i]$ and $[y_i,z_i]$ in $\S$ is uniformly bounded, implying that the minimum distance in $\S$ between $\S_0$ and $\Phi^{m_i}(\S_0)$ is uniformly bounded, but that distance equals $\abs{m_i}$ by Lemma~\ref{LemmaProjectNonIncrease}. In each case it follows that $m_i$ takes on only finitely many values. 

We may therefore pass to a subsequence so that $m_i=m$ is constant, hence $\Psi_i = \Delta_i \Phi^m$ and $L_i \subset \S_m$. By restriction of the action $\wh\cH \act \S$ we obtain an isometric action $K \act \S_m$. Reverting to formal action notation, let $\A_0 \from K \act \S_0$ and $\A_m \from K \act \S_m$ denote the two restrictions of the action $K \act \S$. We have a commutative diagram 
$$\xymatrix{
K \ar[rrr]^{\A_0}  \ar[d]_{i_{\Phi^m}} &&& \Isom(\S_0) \ar[d]^{\text{Ad}_{\Phi^m}} \\
K \ar[rrr]^{\A_m} &&& \Isom(\S_m)
}$$
in which the left vertical arrow is the restriction to the normal subgroup $K \normal \wh\cH$ of the inner automorphism $i_{\Phi^m} \from \wh\cH \to \wh\cH$, and the right vertical arrow is the ``adjoint'' isomorphism given by the formula 
$$\text{Ad}_{\Phi^m}(\Theta)(y) = \Phi^m \cdot \Theta(\Phi^{-m} \cdot y) \quad\text{for each $y \in \S_m$ and each $\Theta \in \Isom(\S_0)$}
$$ 
Using the above commutative diagram, observe that $i_{\Phi^m}$ takes loxodromic WWPD elements of the action $\A_0 \from K \act \S_0$ to loxodromic WWPD elements of the action $\A_m \from K \act \S_m$. Combining this with Step~2, it follows that $\gamma' = i_{\Phi^m}(\gamma)$ is a loxodromic WWPD element of the action $K \act \S_m$. 

Noting that the line $L_{\gamma'} = \Phi^m(L_\gamma)$ is a quasi-axis for $\gamma'$ in $\S_m$, we may apply the property WWPD~(3) to $L_{\gamma'}$, with the conclusion that for each $\Delta \in K - \Stab(\bdy\gamma')$ the image of an $\S_m$-closest point map $\Delta \cdot L_{\gamma'} \mapsto \gamma'$ has uniformly bounded diameter. By careful choice of $\Delta$, namely the members of the sequence $\Delta_i^\inv \Delta^{\vphantom{\inv}}_{i-1} \in K - \Stab(\bdy\gamma')$, we shall use the nesting property $(*)$ to obtain a contradiction. 

Knowing that $[y_{i-1},z_{i-1}] \subset L_{i-1}$ is uniformly Hausdorff close in $\S$ to $[w_{i-1},x_{i-1}]$, and using that the inclusion $\S_m \to \S$ is uniformly property (Lemma~\ref{LemmaUnifProp}), it follows that the diameter of $[y_{i-1},z_{i-1}]$ in $\S_m$ goes to $+\infty$ as $i \to +\infty$. Also, knowing that $[w_{i-1},x_{i-1}]$ is a subpath of $[w_{i},x_{i}]$, and that $[w_{i},x_{i}]$ is uniformly Hausdorff close in $\S$ to $[y_{i},z_{i}] \subset L_{i}$, it follows that $[y_{i-1},z_{i-1}]$ is uniformly Hausdorff close in $\S$ to a subpath of $[y_{i},z_{i}]$. Again using that $\S_m \to \S$ is uniformly proper (Lemma~\ref{LemmaUnifProp}), it follows that $[y_{i-1},z_{i-1}]$ is uniformly Hausdorff close in $\S_m$ to a subpath of $[y_{i},z_{i}]$. It follows that the diameter of the image of the $\S_m$-closest point map $L_{i-1} \mapsto L_{i}$ goes to $+\infty$, because it is greater than $\diam_{\S_m}[y_i,z_i] - C$ for some constant $C$ independent of $i$. Applying the isometric action of the group element $\Delta_i^\inv$, it follows that diameter of the image of the $\S_m$-closest point map from $\Delta_i^\inv \Delta_{i-1} L_{\gamma'} = \Delta_i^\inv L_{i-1}$ to $\Delta_i^\inv L_i = L_{\gamma'}$ goes to $+\infty$ as $i \to +\infty$. This gives the desired contradiction, completing Step~3.

\newcommand\bdyStar{\bdy^*}
\newcommand\bdyS{\bdy^\S}

\subparagraph{The strong axial property.} 
For each $\gamma \in J$ that acts loxodromically on the tree $T$, we regard its axis $L_\gamma$ as an oriented line with positive/negative ideal endpoints equal to the attracting/repelling points $\bdy_+ \gamma$, $\bdy_-\gamma$, respectively. Assuming also $L_\gamma$ is not a Nielsen line, we shall prove that $L_\gamma$ is a strong axis with respect to the action $\wh\cH \act \S$. It suffices to prove this for the set $Z \subset J$ consisting of all root free $\gamma \in J$ such that $L_\gamma$ is not a Nielsen line. 

We have indexed sets of oriented lines $\L = \{L_\gamma \suchthat \gamma \in Z\}$ and of ideal points points $\B = \{\bdy_+ \gamma \suchthat \gamma \in Z\}$, with bijective indexing maps $Z \mapsto L_Z$ and $Z \mapsto \B$. Having shown already that each $\gamma \in Z$ acts loxodromically on $T^*$ and on $\S$, with attracting/repelling pairs denoted $\bdyStar_\pm\gamma \in T^*$ and $\bdyS_\pm\gamma \in \bdy\S$, we obtain indexed sets $\B^* = \{\bdyStar_+\gamma \suchthat \gamma \in Z\}$ and $\B^\S = \{\bdyS_+\gamma \suchthat \gamma \in Z\}$.

We next show that the indexing map $Z \mapsto \B^\S$ is a bijection. The oriented line $L_\gamma$ is also a quasi-axis for $\gamma$ in $\S$, and hence the positive end of $L_\gamma$ limits on $\bdyS_+\gamma$ in~$\bdy\S$. Consider $\gamma \ne \delta \in Z$. The line $L_{\gamma\delta} \subset T$ with ideal endpoints $\bdy_+\gamma,\bdy_+\delta$ in~$\bdy T$ is a concatenation of the form $L_{\gamma\delta} = \bar R_\gamma A R_\delta$ where $R_\gamma$ is a positive ray in $L_\gamma$ and $R_\delta$ is a positive ray in $L_\delta$, hence the two ends of $L_{\gamma\delta}$ limit on $\bdyS_+\gamma,\bdyS_-\delta$ in $\bdy\S$. To show that $\bdyS_+\gamma \ne \bdyS_+\delta$ it therefore suffices to prove that $L_{\gamma\delta}$ is a quasigeodesic line in $\S$. Since $L_{\gamma\delta}$ is a reparameterized quasigeodesic in $T^*$ (Proposition~\ref{PropTStarHyp}), and since its disjoint subrays $R_\gamma$, $R_\delta$ each have infinite $D_\PF$ diameter in $T$ and hence infinite $d^*$ diameter in $T^*$ (Proposition~\ref{PropConeQI}), it follows that $L_{\gamma\delta}$ is a quasigeodesic line in $T^*$; this shows that $\bdyStar_+\gamma \ne \bdyStar_+\delta$. Arguing by contradiction, suppose that $\bdyS_+\gamma = \bdyS_+\delta \in \bdy\S$, and so both ends of $L_{\gamma\delta}$ limit on that point of $\bdy\S$. Since the rays $R_\gamma$, $R_\delta$ are quasigeodesics in $\S$, it follows that there are sequences $(x_i)$ in $R_{\gamma}$ and $(y_i)$ in $R_{\delta}$ such that the distances $d^\S\!(x_i,y_i)$ are bounded. Since $x_i,y_i$ approach $\bdyStar_+\gamma \ne \bdyStar_+\delta$ in $\bdy T^*$, respectively, it follows that $d^*(x_i,y_i)$ approaches $+\infty$. Since the inclusion $T^* \inject \S$ is uniformly proper (Lemma~\ref{LemmaUnifProp}), it also follows that the distances $d^\S(x_i,y_i)$ approach $+\infty$, a contradiction.

We claim that there exist actions of the group $\wh\cH$ on the sets $Z$, $\B$ and $\B^\S$ satisfying the following properties:
\begin{enumerate}
\item\label{ItemHEqBij} The indexing bijections $Z \mapsto \B, \B^\S$ are $\wh\cH$-equivariant.
\item\label{ItemHMinusJ} No element of $\wh\cH - K$ fixes any element of $Z$. 
\end{enumerate}
Once this claim is proved, from \pref{ItemHEqBij} and~\pref{ItemHMinusJ} it follows that for each $\gamma \in Z$ we have $\Stab_K(\bdy_\pm \gamma) = \Stab_K(\bdy^\S_\pm\gamma) = \Stab_{\wh\cH} (\bdy^\S_\pm\gamma)$. In the place where~Lemma~\ref{LemmaTreeWWPD} is applied in Step~1 we may further conclude that each $\gamma \in Z$ is strongly axial with respect to the action $K \act T$, and hence the action of $\Stab_K(\bdy_\pm \gamma)$ on $T$ preserves $L_\gamma$. The action of $\Stab_{\wh\cH} (\bdy^\S_\pm\gamma)$ on $\S$ therefore preserves~$L_\gamma$, proving that $\gamma$ is strongly axial with respect to the action $\wh\cH \act \S$, and we are done. 

For proving the claim, we will for the most part restrict our attention to the semigroup $\wh\cH_+ = \wh\cH \intersect \wh\D_+$ and its semigroup action $\wh\cH_+ \act T$ which is obtained by restricting $\wh\D_+ \act T$ (Section~\ref{SectionActionOnT}).

The action $\wh\cH \act Z$ is defined by restricting to $Z$ the natural inner action of $\wh\cH$ on its own normal subgroup $J = \wh\cH \intersect \Inn(F_n)$, but we must verify that $Z$ is invariant under that action. Given $\gamma \in J$, and $\Psi \in \wh\cH_+ = \wh\cH \intersect \wh\D_+$, we note three things.
First, applying Property~\pref{Item_f_Twisted} \emph{Twisted Equivariance} of Section~\ref{SectionActionOnT}, the action of $\gamma$ on $T$ is loxodromic with axis $L_\gamma$ if and only if the action of $\Psi(\gamma)$ on $T$ is loxodromic with axis $f^\Psi_\#(L_\gamma)=L_{\Psi(\gamma)}$. Second, $\gamma$ is root free if and only if $\Psi(\gamma)$ is root free. Third, applying Property~\pref{Item_f_PsiNielsen} \emph{Invariance of Nielsen data} of Section~\ref{SectionActionOnT}, the axis $L_\gamma$ is a Nielsen line if and only if $L_{\Psi(\gamma)}$ is a Nielsen line. It follows that $\gamma \in Z$ if and only if $\Psi(\gamma) \in Z$, hence $Z$ is $\wh\cH$-invariant.

The argument in the previous paragraph yields a little more: the bijection $Z \leftrightarrow \L$ is equivariant with respect to the action $\wh\cH \act Z$ and the action $\wh\cH \act \L$, the latter of which is defined by $\Psi \cdot L_\gamma = f^\Psi_\#(L_\gamma) = L_{\Psi(\gamma)}$. Existence of actions of $\wh\cH$ on $\B$ and $\B^\S$ satisfying item~\pref{ItemHEqBij} is an immediate consequence: for each $\Psi \in \wh\cH_+$, the ideal point in $\B$ or in $\B^\S$ represented by the positive end of the oriented line $L_\gamma$ is taken by $\Psi \in \wh\cH_+$ to the ideal point represented by the positive end of $L_{\Psi(\gamma)}$.

Finally, item~\pref{ItemHMinusJ} is a special case of the following: for each $\Psi \in \wh\D - \wh\D_0$ and each $\gamma \in F_n$, if $\gamma$ is loxodromic and if its axis $L_\gamma$ is not a Nielsen line then \hbox{$\Psi(\gamma) \ne \gamma$}. Suppose to the contrary that $\Psi(\gamma)=\gamma$, let $\Psi = \Delta\Phi^m$ where \hbox{$\Delta \in \Inn(F_n)$} and $m \ne 0$, and let $c$ be the conjugacy class of $\gamma$ in $F_n$, and so $\phi^m(c)=c$. Let $\sigma$ be the circuit in $G$ represented by $c$ (Notations~\ref{NotationsFlaring}). We may assume $m > 0$, and so $f^m_\#(\sigma)=\sigma$. It follows that $\sigma$ is a concatenation of fixed edges and indivisible Nielsen paths of $f$ (\SubgroupsOne, Fact 1.39). However, no edge of $H_u$ is fixed, and the only indivisible Nielsen path that crosses an edge of $H_u$ is $\rho_u^{\pm 1}$ which has at least one endpoint disjoint from $G_{u-1}$ (Notations~\ref{NotationsFlaring}~\pref{ItemCTiNP}). One of two cases must therefore hold: either $\sigma$ is contained in $G_{u-1}$, implying that $\gamma$ is not loxodromic (Definition~\ref{DefOfT}); or $\rho$ exists and is closed, and $\sigma$ is an iterate of $\rho$ or $\rho^\inv$, implying that $L_\gamma$ is a Nielsen line (Definition~\ref{DefNielsenSet}). In either case, we are done verifying item~\pref{ItemHMinusJ}.

\bibliographystyle{amsalpha} 
\bibliography{mosher} 


\end{document}